\documentclass[aos]{imsart}

\RequirePackage[numbers]{natbib}
\RequirePackage[colorlinks,citecolor=blue,urlcolor=blue]{hyperref}%% uncomment this for coloring bibliography citations and linked URLs
\usepackage{subcaption}
\usepackage{rotating}
\usepackage{old-arrows}
\usepackage{enumitem}
\usepackage{mathtools}
\usepackage{amsmath}

\usepackage{amsfonts}       % blackboard math symbols
\usepackage{nicefrac}       % compact symbols for 1/2, etc.
\usepackage{microtype}      % microtypography

\usepackage{epsf}
\usepackage{epsfig}
\usepackage{fancyhdr}
\usepackage{graphics}
\usepackage{graphicx}

\usepackage{amsthm}
\usepackage{amsmath}
\usepackage{amssymb,bbm}
\usepackage{caption}
\usepackage{multicol}
\usepackage{caption}

\usepackage{microtype}
\usepackage{graphicx}
\usepackage{booktabs}

\usepackage{url}% for url's in bib
\DeclareMathOperator{\diag}{diag}

\DeclareMathOperator{\Dir}{Dir}

\DeclareMathOperator{\Multi}{Multi}

\usepackage{mathrsfs}

\theoremstyle{plain}

\newtheorem{lemma}{Lemma}
\newtheorem{theorem}{Theorem}
\newtheorem{proposition}{Proposition}
\newtheorem{corollary}{Corollary}

\theoremstyle{definition}
\newtheorem{definition}{Definition}

\theoremstyle{remark}

\newtheorem{remark}{Remark}

\long\def\comment#1{}

\newcommand{\Ecal}{\ensuremath{\mathcal{E}}}

\newcommand{\Ocal}{\ensuremath{\mathcal{O}}}
\newcommand{\Ebb}{\ensuremath{\mathbb{E}}}

\newcommand{\Bcal}{\ensuremath{\mathcal{B}}}

\newcommand{\Pcal}{\ensuremath{\mathcal{P}}}

\newcommand{\Rbb}{\ensuremath{\mathbb{R}}}
\newcommand{\Nbb}{\ensuremath{\mathbb{N}}}

\newcommand{\stirlingtwo}[2]{\genfrac{\{}{\}}{0pt}{}{#1}{#2}}

\newcommand{\stirlingone}[2]{\genfrac{[ }{ ]}{0pt}{}{#1}{#2}}

\newcommand{\Span}{\ensuremath{\textnormal{span}}}

\newcommand{\norm}[1]{\left\|#1\right\|}
\newcommand{\innprod}[1]{\left\langle#1\right\rangle}
\allowdisplaybreaks

\endlocaldefs

\begin{document}

\begin{frontmatter}
\title{Dirichlet moment tensors and the correspondence between admixture and mixture of product models
}
%old title: MOMENT TENSORS OF DIRICHLET DISTRIBUTIONS AND POSTERIOR CONTRACTION BEHAVIOR OF THE LATENT DIRICHLET ALLOCATION

%\title{A sample article title with some additional note\thanksref{T1}}
\runtitle{Moment tensors of Dirichlet distributions and learning Latent Dirichlet Allocation}

\begin{aug}
%%%%%%%%%%%%%%%%%%%%%%%%%%%%%%%%%%%%%%%%%%%%%%%
%% Only one address is permitted per author. %%
%% Only division, organization and e-mail is %%
%% included in the address.                  %%
%% Additional information can be included in %%
\author[A]{\fnms{Dat}~\snm{Do}\ead[label=e1]{dodat@uchicago.edu}},
\author[B]{\fnms{Sunrit}~\snm{Chakraborty}\ead[label=e2]{
sunrit.chakraborty@duke.edu}}
\author[C]{\fnms{Jonathan}~\snm{Terhorst}\ead[label=e3]{jonth@umich.edu}}
\and
\author[C]{\fnms{XuanLong}~\snm{Nguyen}\ead[label=e4]{xuanlong@umich.edu}}
%%%%%%%%%%%%%%%%%%%%%%%%%%%%%%%%%%%%%%%%%%%%%%
%% Addresses                                %%
%%%%%%%%%%%%%%%%%%%%%%%%%%%%%%%%%%%%%%%%%%%%%%
\address[A]{Department of Statistics,
University of Chicago, Chicago, IL 60637, USA\printead[presep={,\ }]{e1}}
\address[B]{Department of Statistical Science, Duke University, Durham, NC 27710, USA\printead[presep={,\ }]{e2}}
\address[C]{Department of Statistics, University of Michigan, Ann Arbor, MI 48109, USA\\
\printead[presep={\ }]{e3,e4}}
\end{aug}

\begin{abstract}
Understanding posterior contraction behavior in Bayesian hierarchical models is of fundamental importance, but progress in this question is relatively sparse in comparison to the theory of density estimation. In this paper, we study two classes of hierarchical models for grouped data, where observations within groups are exchangeable. Using moment tensor decomposition of the distribution of the latent variables, we establish a precise equivalence between the class of Admixture models (such as Latent Dirichlet Allocation) and the class of Mixture of products of multinomial distributions. This correspondence enables us to leverage the result from the latter class of models, which are more well-understood, so as to arrive at the identifiability and posterior contraction rates in both classes under conditions much weaker than in existing literature. For instance, our results shed light on cases where the topics are not linearly independent or the number of topics is misspecified in the admixture setting. Finally, we analyze individual documents' latent allocation performance via the borrowing of strength properties of hierarchical Bayesian modeling. Many illustrations and simulations are provided to support the theory.
% Latent Dirichlet Allocation (LDA) is a powerful model for estimating topics' structure in complex discrete heterogeneous data such as text, genetics, and psychological data. Despite its wide applications, scant attention has been paid to LDA's asymptotic theory, especially in a setting where the number of topics is misspecified. The challenge of this problem lies in the complex representation of the LDA model in terms of the Dirichlet moment tensors. In this paper, we provide two sets of algebraic equations to represent the Dirichlet moment tensors (of any order) to simple diagonal tensors and vice versa. This yields an equivalence between the LDA models and the (better-known) Mixture of Multinomial models. From that, we provide the identifiability and parameter estimation rate in both models when the number of topics is either known or misspecified by a larger number. Finally, we also derive the estimation rate for individual documents' Latent Allocation via the borrowing strength properties of hierarchical Bayesian modeling. Many illustrations and simulations are provided to support our theory.
\end{abstract}

\begin{keyword}[class=MSC]
\kwd[Primary ]{62F15}
\kwd{62G05}
\kwd{62E10}
\kwd[; secondary ]{62G20}
\end{keyword}

\begin{keyword}
\kwd{Moment tensor}
\kwd{Dirichlet distribution}
\kwd{Latent Dirichlet Allocation}
\kwd{Mixture models}
\kwd{Identifiability}
\kwd{Contraction rate}
\end{keyword}

\end{frontmatter}
%%%%%%%%%%%%%%%%%%%%%%%%%%%%%%%%%%%%%%%%%%%%%%
%% Please use \tableofcontents for articles %%
%% with 50 pages and more                   %%
%%%%%%%%%%%%%%%%%%%%%%%%%%%%%%%%%%%%%%%%%%%%%%
%\tableofcontents

% {\color{red}\section*{To-do}
% \begin{enumerate}
%     \item Shorten the introduction: Focus on the vision for the theory of hierarchical model and the approach that we take in the paper (then narrow down to LDA later)
%     \item In 2.1, introduce the model hierarchically (involves $X, \theta, q, \Dir_{\alpha}$ etc), then write it as tensors.
%     Then, introduce tensor algebra notations right after that.
%     \item (i) Define symmetric tensors; (ii) check examples on the correspondence between Dirichlet moment tensors and diagonal tensors in cases 3 and 4; (iii) N → log(N). Add a reference to Igor and Antonio et al. work on the random partition. 
% \end{enumerate}
% }

% {\color{red}\section*{To-do for Sunrit}
% \begin{enumerate}
%     \item Write an introduction: (1) Hierarchical models -> Examples -> Want to marginalize out -> Connect to simpler one. (2) Contribution. (3) Organization.   
%     \item Restructure Section 2. Add other lit reviews at the end of Section 2.
%     \item Correct Long points in Sections 1 and 2. 
%     \item Add references section between Sections 2 and 3. 
% \end{enumerate}
% }

\section{Introduction}\label{sec:introduction}
Hierarchical models provide some of the most powerful tools for analyzing heterogeneity in the observed data. Such models use latent variables to capture the variability in the data through a probabilistic data-generating mechanism. Along with their flexibility as a density estimator, such models are also interpretable, and the parameters capture meaningful information about the latent subpopulations in the data~\cite{mclachlan1988mixture, pritchard2000inference, blei2003latent}.  
Hence, understanding parameter estimation behavior in hierarchical models is of fundamental importance. While theoretical characterization related to density estimation can be handled by standard techniques~\cite{geer2000empirical,ghosal2017fundamentals}, parameter estimation behavior is substantially more challenging to pin down due to the presence of latent variables. One of the simplest latent variable models is mixture models \cite{mclachlan1988mixture}---despite its widespread application spanning over 150 years, continued progress in deepening our understanding of the parameter estimation behavior in general mixture models is still being made recently, e.g., \cite{rousseau2011asymptotic,nguyen2013convergence,Ho-Nguyen-EJS-16, Ho-Nguyen-Ann-16,heinrich2018strong,guha2021posterior,wei2023minimum}.
 
In this paper, we study Bayesian parameter estimation in a canonical class of hierarchical models by establishing their precise connection to simpler mixture models. Specifically, we focus on two types of hierarchical models used for modeling observations arising from different observed groups, where observations within a group are modeled as \textit{conditionally} independent and identically distributed (i.i.d.) under a mixture model framework. The \textit{mixture of product} (a.k.a. product mixture) model assumes that conditional on the latent group-specific component assignment, observations in the group are i.i.d. from this component distribution, while the \textit{admixture} model assumes that conditional on the latent group-specific mixture probabilities, the observations in the group are i.i.d. from a mixture model with all the components with these mixture probabilities. As a concrete example, consider a corpus of text documents. Each document represents a group of observations consisting of a collection of words taken from a vocabulary. The heterogeneity in the data is modeled via $K$ latent subpopulations, each corresponding to a \textit{topic}. Under this setting, the mixture of product (of multinomial) model, better known as the mixture of unigram model \cite{Nigam2000-yz}, allows each document to contain only a single topic. On the other hand, for an admixture model, each document is modeled as a conditional mixture of multinomial distributions and the mixing probabilities (called \textit{latent allocations}) are document-specific latent variables, often endowed with a parametric distribution (known as \textit{admixing distribution}). When the admixing distribution is Dirichlet, the admixture model becomes the celebrated Latent Dirichlet Allocation (LDA) \cite{blei2003latent}. The admixture model is popular in various other domains, notably in population genetics as the well-known STRUCTURE software \cite{pritchard2000inference}. Despite its impressive versatility and broad applicability in such diverse domains, theoretical understanding of the LDA remains far from being complete (see Section \ref{sec:existing work} for an overview of existing results), while mixture of product models have received significant theoretical advances in recent years \cite{10.1214/09-AOS689,vandermeulen2019operator,wei2022convergence}.

 A major goal of this paper is to establish a precise algebraic correspondence between the two classes of hierarchical models, the product mixtures and admixtures, and to investigate the implications entailed by such a correspondence. The correspondence between these two model classes are obtained by marginalizing over the latent variables to express the distribution of a group in terms of the moment tensors of the admixing distribution. Combined with a decomposition of the Dirichlet moment tensor, this leads to the aforementioned correspondence between the mixture of product of multinomials model and the LDA. We then exploit this correspondence to study theoretical properties of LDA under much weaker settings than those in existing literature. For just an example, our results shed light on the effect of the document size on the identifiability of the LDA, establish identifiability and rates even when the topics are not linearly independent and also in the over-fitted regime (using more topics than that of the ground-truth distribution). In particular, we find that longer documents can help to achieve identifiability and a good parameter estimation rate in over-fitted settings under minimal conditions.

Our contributions can be summarized as follows. First, we present a moment tensor decomposition formulae (Lemma \ref{lem:whiten-Dirichlet-moment}) -- two sets of algebraic manipulations connecting moment tensors of Dirichlet distributions of any order to diagonal tensors. The equivalence of the geometries of the space of LDA models and mixture models is then deduced (Theorem \ref{thm:equivalence-LDA-mixture}). Thanks to this result, the questions of identifiability (Proposition \ref{prop:identifiability}) and parameter estimation (Theorem \ref{theorem:contraction_posterior}) in both models can be simultaneously studied. Parameter estimation using Bayesian methods is considered here, but similar results can be obtained for the maximum likelihood estimator as well. The center of our analysis is a collection of inequalities linking the distance of the LDA (or mixture of product) models to the Wasserstein distances of corresponding parameters, known as \emph{inverse bounds} (Theorem \ref{thm:inverse-bounds}) --- such inequalities have been established for mixture models in ~\cite{nguyen2013convergence, wei2022convergence, do2025strong}, although not in the setting of the current paper. A treatment on the role of document lengths (number of words in each document) in those bounds is now possible and carefully examined. We also provide the rate for learning the latent allocation associated with a given document for increasing document lengths, which is of interest to domains such as population genetics. Lastly, simulation studies to support the theory will be presented.

%{\bf \color{red} LN: Need to add a comment on the roles of Dirichlet; acknowledge the specific contributions/ limitation of relying on Dirichlet and generality of this work if one goes beyond Dirichlet. Either here or in a concluding remark.} 
A limitation of this work is that most of the results are specific to the use of Dirichlet distributions. However, like Gaussian distributions in multivariate analysis, given the fundamental roles Dirichlet distributions and processes play in the probabilistic modeling of random partitions, and the applicability of topic models in a wide variety of domains, we believe the identifiability and posterior contraction theory presented here is of interest to both theoreticians and practitioners. Moreover, 
the overall approach developed here, one which connects hierarchical models to product mixture models, is general and may be extendable to other classes of hierarchical models beyond the Dirichlet \cite{teh2006hierarchical,Nguyen-ba-10}, including those derived from various exchangeable random partition functions \cite{camerlenghi2019distribution,CatalanoEtAlUnifiedApproachHierarchical2024}.

The remainder of the paper is organized as follows. In Section \ref{sec:main-results}, we discuss the central results of this work, including the tensor decomposition and the correspondence between the two classes of distributions. In Section \ref{sec:identifiability}, we use these results to shed light on the identifiability in LDA models, and also establish inverse bounds. These inverse bounds are then used in Section \ref{sec:contraction-rate}, where posterior contraction rates for the parameters are established, and also for the latent allocation. The simulation studies will be given in Section \ref{sec:experiment}.

\subsection{Notations}\label{subsec:notations-tensors}
% \subsubsection*{Sets, spaces and relations} 
Denote by $\Nbb$ the space of all non-negative integers and $\Nbb_{+} = \Nbb \setminus \{0\}$. For any $K \in \Nbb_{+}$, let $[K] = \{1, \dots, K\}$ be the set of integers from $1$ to $K$. Let $\Rbb^{K}$ and $\Rbb^{K}_{+}$ be the real vector space and its positive part, respectively. Let $\Delta^{K}$ be the $K$ dimensional simplex. For any $a > 0$ and $N \in \Nbb$, $a^{[N]} = a (a + 1)\cdots (a + N -1)$ denotes the rising factorial. For two sequences $(a_n)_{n=1}^{\infty}$ and $(b_n)_{n=1}^{\infty}$, we write $a_n\lesssim b_n$ if $a_n \leq C b_n$ where $C$ is a constant not depending on $n$. We write $a_n \gtrsim b_n$ when $b_n \lesssim a_n$, and $a_n \asymp b_n$ if $a_n \gtrsim b_n$ and $a_n \lesssim b_n$.

% \subsubsection*{Distributions and their distances} 
For a tuple $\alpha = (\alpha_1, \dots, \alpha_K) \in \Rbb_{+}^{K}$, denote by $\Dir_{\alpha}$ the Dirichlet distribution with parameter $\alpha$, $\overline{\alpha} = \sum_{k=1}^{K} \alpha_{k}$ its concentration parameter and $\tilde{\alpha} = \alpha / \overline{\alpha}$ its mean.   
For two density functions $p_1$ and $p_2$ on $\mathcal{X}$ that are commonly dominated by a common measure $\mu$, denote by $d_{H}^2(p_1, p_2) = \frac{1}{2}  \int_{\mathcal{X}} (\sqrt{p_1} - \sqrt{p_2})^2 d\mu$ the square Hellinger distance, $d_{TV}(p_1, p_2) = \frac{1}{2} \int_{\mathcal{X}} |p_1 - p_2| d\mu$ the Total Variation distance, and $d_{KL}(p_1 \| p_2) = \int_{\mathcal{X}} p_1 \log\left( p_1/p_2\right) d\mu$ the K-L divergence between $p_1$ and $p_2$, respectively. The distances are related by $d^{2}_{H} \leq d_{TV} \leq \sqrt{2} d_{H}$ and $d_{H}^2 \leq d_{KL}/2$.

% $\tau: [N_m]\to [N_m]$, we denote the permuted outer product:
% $$\otimes^{\tau} (Q^{1}, \dots, Q^{m}) = \left(\otimes_{j=1}^{m} Q^{j} \right)^{\tau}.$$
% Sometimes we just write the permutation $\tau:[N]\to [N]$ as $(\tau(1), \dots, \tau(N))$ for short. Some examples of the permuted outer product: \begin{enumerate}
%     \item For 3 vectors $\theta_1, \theta_2, \theta_3 \in \Rbb^{K}$, we have $\otimes_{j=1}^{3} \theta_j$ is a tensor in $\Rbb^{K\times K \times K}$, whose $(k_1, k_2, k_3)$-element is $\theta_{1k_1}\theta_{2k_2} \theta_{3k_3}$. For the permutation $\tau: [3]\to [3]$, $\tau(1)=1, \tau(2) = 3, \tau(3)=2$, we have $\otimes^{(1,3,2)}(\theta_1, \theta_2, \theta_3) \in \Rbb^{K\times K\times K}$ having $(k_1, k_2, k_3)$-element being $\theta_{1k_1}\theta_{3k_2}\theta_{2k_3}$.
%     \item For a vector $\theta_1 \in \Rbb^{K}$ and a matrix $\Theta\in \Rbb^{K\times K}$, we have
%     \begin{align}
%     \otimes^{(1,2,3)}(\Theta, \theta) & = \Theta \otimes \theta = (\Theta_{k_1 k_2}\theta_{ k_3})_{k_1, k_2, k_3=1}^{K} \\
%     \otimes^{(1,3,2)}(\Theta, \theta) & = (\Theta_{k_1 k_3}\theta_{k_2})_{k_1, k_2, k_3=1}^{K}\\
%     \otimes^{(2,3,1)}(\Theta, \theta) & = \theta\otimes \Theta = (\Theta_{k_2 k_3}\theta_{k_1})_{k_1, k_2, k_3=1}^{K}.
%     \end{align}
%     So basically, this notation decides to multiply $\theta$ to which face of $\Theta$ (think of 2D or 3D objects such as a rubik!). Please let me know if there are better notations. We will use those in Section 2.1.
% \end{enumerate}
\section{Correspondence between LDA and mixture models}\label{sec:main-results}
\subsection{Hierarchical models and their tensor representations}
The focal point of this section is the correspondence between LDA, a canonical class of hierarchical models, and mixture of product models. %In this section, we introduce the models and demonstrate how we can marginalize over the latent variables to express the likelihood of such models in terms of appropriate moment tensors of the underlying mixing distribution. 
In what follows, the models are introduced in the language of topic modeling in text documents setting. The observations consist of a text corpora containing several documents on a vocabulary of size $V$, where each document contains $N$ words $X_{[N]}=(X_1,\dots,X_N)$ that takes values in $\{1,\dots,V\}$ and the documents are assumed to be independent and identically distributed. In both model classes, we posit that there are $K$ `ancestral' topics $\theta_1,\dots,\theta_K\in\Delta^{V-1}$, which give rise to the observations, where a topic is simply a probability distribution on the vocabulary. The LDA and mixtures of product are models for the joint distribution of such a document based on the latent topics, whose roles are described next.

In a mixture of product model, each document is generated by a topic which is chosen with mixing probabilities $\tilde{\alpha}\in\Delta^{K-1}$. Given the topic, the words in the document are generated independently and identically from a categorical distribution with the chosen topic as the probability parameter. This model can hence be expressed as a hierarchical model in the following way
\begin{align*}
    z | \tilde{\alpha} &\sim \text{Categorical}(\tilde{\alpha}) \\
    X_1,\dots,X_N |\Theta, z=k &\overset{\text{iid}}{\sim} \text{Categorical}(\theta_k),
\end{align*}
for $k=1,\ldots, K$, where $\Theta=(\theta_1,\dots,\theta_K)$ denotes the topics. Marginalizing over the latent variable $z$ yields the following expression for the probability mass function at $x_{[N]}=(x_1,\dots,x_N)\in [V]^N$ under the mixture of product model
\begin{align}
p^{\mathscr{M}}_{\tilde{\alpha},\Theta}(x_{[N]}) = \sum_{k=1}^K \tilde{\alpha}_k \theta_{k x_1}\theta_{k x_2}\dots \theta_{k x_N}.
\end{align}
On the other hand, the LDA assumes that first, a document-specific mixing proportion $q=(q_1,\dots,q_K)\in\Delta^{K-1}$ is drawn from a Dirichlet distribution with parameter $\alpha\in \Rbb_+^K$. Given this latent allocation variable $q$, a topic is drawn from the categorical distribution with parameter $q$ and given the word-specific topic, a word is drawn from a categorical distribution with the chosen topic as the parameter. Expressing 
as a hierarchical model
\begin{align*}
    q|\alpha &\sim \text{Dir}_{\alpha} \\
    z_1,\dots, z_N | q &\overset{\text{iid}}{\sim} \text{Categorical}(q) \\
    X_i | \Theta, z_i=k &\overset{\text{ind}}{\sim} \text{Categorical}(\theta_k), 
\end{align*}
for $i=1,2,\dots,N$ and $k=1,\ldots, K$.
In this case, both $z_1,\dots,z_N$ and $q$ are latent variables. Marginalizing over these variables, the probability function for LDA at $x_{[N]}=(x_1,\dots,x_N)\in [V]^N$ takes the form
\begin{align}\label{eq:LDA_density1}
    p^{\mathscr{L}}_{\alpha,\Theta}(x_{[N]}) &= \Ebb_{q\sim \text{Dir}_{\alpha}} \prod_{j=1}^{N} \left(\sum_{k=1}^K q_k\theta_{k x_j}\right) \nonumber\\
    &=\Ebb_{q\sim \text{Dir}_{\alpha}} \sum_{1\leq k_1,\dots,k_N \leq K} q_{k_1}\dots q_{k_N} \theta_{k_1 x_1} \dots \theta_{k_N x_N} \nonumber\\
    &= \sum_{1\leq k_1,\dots,k_N \leq K} Q_{\alpha,k_1\dots k_N}^{(N)} \theta_{k_1 x_1} \dots \theta_{k_N x_N},
\end{align}
where the last line follows from exchanging expectation and summation, and $Q_{\alpha}^{(N)}$ arises as the $N-$th order moment tensor of Dirichlet distribution with parameter $\alpha$ -- it is a tensor of shape $K\times \dots \times K$ ($N$ times), such that the $(k_1,\dots,k_N)-$th element of the tensor is
\begin{align}\label{eq:dir_moment_tensor}
    Q_{\alpha,k_1\dots k_N}^{(N)} = \Ebb_{q\sim \text{Dir}_{\alpha}}[q_{k_1}\dots q_{k_N}].
\end{align}

Note that similarly, the mixture of product probability function can also be expressed using a `diagonal' tensor
\begin{align}\label{eq:mixture_product_density1}
    p^{\mathscr{M}}_{\tilde{\alpha},\Theta}(x_{[N]}) = \sum_{1\leq k_1,\dots,k_N\leq K} \diag_N(\tilde{\alpha}) \theta_{k_1 x_1}\dots \theta_{k_N x_N}
\end{align} 
where $\diag_N(\tilde{\alpha})$ is another $N$-way tensor of shape $K\times \dots K$ ($N$ times), such that $\diag_N(\tilde{\alpha})_{k_1\dots k_N} = \tilde{\alpha}_k 1_{(k_1=\dots=k_N=k)}.$
Thus, to establish correspondence between the two classes of models, it is important to understand how the Dirichlet moment tensor in equation \eqref{eq:dir_moment_tensor} can be decomposed into simpler diagonal tensors. Before getting into such results, let us introduce basic terminology and notation associated with tensor algebra that we shall need in the remainder of the paper.

\subsubsection{Tensor algebra} A basic concept in tensor algebra that serves as a building block for the LDA and mixture models is $N$-way tensors. An $N$-way (or $N$ dimensional) tensor $Q$ is a multi-dimensional array of some shape $K_1 \times K_2 \dots \times K_N$ whose each element $Q(k_1, \dots, k_N)$ is a real number for every $k_1\in [K_1], \dots, k_N\in [K_N]$. We sometimes write $Q \in \Rbb^{K_1\times\dots\times K_N}$ and its element as $Q_{k_1 \dots k_N}$ for short. For $\alpha = (\alpha_1, \dots, \alpha_K)\in \Rbb^{K}$, denote by $\diag_N(\alpha)$ the diagonal vector of order $N$ with diagonal elements $\alpha_1, \dots, \alpha_K$, i.e. $\diag_N(\alpha)_{k_1\dots k_N} = \alpha_{k_1} 1_{(k_1 = \dots = k_N)}$.

% {\bf \color{red} LN: you have been using N-way, but also N-th order, and $N$ dimensions. Please fix is. It's better to be consistent, or note somewhere in the beginning the equivalent expressions.} {\color{blue} We decided to call it $N$-way or $N$ dimensions. "Order" is for polynomials.}

\textit{Weighted outer product of vectors.} For a tensor $Q$ of shape ${K\times \dots \times K}$ ($N$ times) and a matrix $\Theta = (\theta_{1}, \dots, \theta_{K})^{\top}$ with row vectors $\theta_k = (\theta_{kv})_{v = 1}^{V}\in \Rbb^{V}$, $k\in [K]$, let 
$$Q[\Theta^{\otimes N}] = \sum_{k_1, \dots, k_N= 1}^{K} Q_{k_1,\dots, k_N} \theta_{k_1} \otimes \cdots \otimes \theta_{k_N} $$ 
denote the weighted outer product of rows of $\Theta$. It is a $V\times \dots\times V$ ($N$ times) tensor whose $(v_1, \dots, v_N)$ element is
$$Q[\Theta^{\otimes N}](v_1, \dots, v_N) = \sum_{k_1, \dots, k_N=1}^{K} Q_{k_1,\dots, k_N} \theta_{k_1 v_1} \cdots \theta_{k_N v_N}.$$
When $N=3$ and $Q = \diag_3((1,1,1))$, this reduces to the 3-way tensor studied by Kruskal that plays a fundamental role in the study of identifiability of mixtures of products~\cite{KRUSKAL197795, 10.1214/09-AOS689}.

\textit{Outer product of tensors.} For $m$ tensors $Q^{1} \in \Rbb^{K_1\times \dots K_{N_1}}, Q^2\in \Rbb^{K_{N_1+1}\times \dots K_{N_2}}, \dots, Q^{m}\in \Rbb^{K_{N_{m-1}+1}\times \dots \times K_{N_m}}$ (with $N_1 < N_2 < \dots < N_m$) the outer product of 
tensors is denoted by
$$\otimes_{j=1}^{m} Q^{j} = \left(\prod_{j=1}^{m} Q^{j}_{k_{N_{j-1}+1}\dots k_{N_j}} \right)_{k_1,\dots, k_{N_m}}\in \Rbb^{K_1\times \dots \times K_{N_1}\times \dots \times K_{N_m}}.$$

\textit{Permutations acting on tensors.} A permutation $\tau: [N]\to [N]$ may be written as $(\tau(1), \dots, \tau(N))$ for short. For a tensor $Q \in \Rbb^{K_1\times\dots\times K_N}$, let 
$T_{\tau}(Q)$ denote its transpose with respect to a permutation $\tau: [N]\to [N]$, that is 
$$T_{\tau}( Q)_{k_1 \dots k_N} = Q_{k_{\tau(1)} \dots k_{\tau(N)}} \quad \forall k_1, \dots, k_N\in [K].$$
For example, with $Q\in \Rbb^{K\times K \times K}$, 
$$T_{(2, 3, 1)}(Q)_{k_1 k_2 k_3} = Q_{k_2 k_3 k_1}, \quad T_{(1, 3, 2)}(Q)_{k_1 k_2 k_3} = Q_{k_1  k_3 k_2}\quad\forall k_1, k_2, k_3 \in [K].$$

\textit{Symmetry and permutation.} We will often work with symmetric tensors and their products, which arise from the exchangeability of observations generated by the LDA and mixture models. Likewise, partial exchangeability leads to subsets of dimensions being symmetric. In particular, a tensor $Q$ is said to have dimension $(i_1, \dots, i_j)$ being symmetric if its value does not change under permutation of indices $(i_1, \dots, i_j)$.
Suppose a tensor $Q$ of dimension $d = d_1 + \cdots + d_n$ of which the first $d_1$ dimensions are symmetric, the next $d_2$ dimensions are symmetric,..., last $d_n$ dimensions are symmetric, and any partition $(S_1, \dots, S_n)$ of $[d]$ having $|S_1| = d_1,\dots, |S_n| = d_n$. Clearly, the permutation $\tau : [d] \to [d]$ such that  $\tau: [(d_{i-1} + 1): d_i] \to S_i$ will act on $Q$ the same way regardless of the order of elements in the set $S_i$'s. In this case, with a slight abuse of notation, we may use $T_{\tau}(Q)$ and $T_{(S_1, \dots, S_n)}(Q)$ interchangeably. For example, if $Q_1 \in \Rbb^{K\times K}$ and $Q_2 \in \Rbb^{K\times K}$ are both symmetric, then the product tensor $Q = Q_1 \otimes Q_2\in \Rbb^{K^4}$ has the first two and last two dimensions being symmetric. Take, for instance, $S_1 = \{1, 3\}$ and $S_2 = \{2, 4\}$, and let $\tau$ and $\sigma$ be two permutations such that $(\tau(1), \tau(2), \tau(3), \tau(4)) = (1, 3, 2, 4)$ and $(\sigma(1), \sigma(2), \sigma(3), \sigma(4)) = (3, 1, 4, 2)$. Then
$$T_{\tau}(Q)_{k_1 k_2 k_3 k_4} 
= Q_{k_1 k_3 k_2 k_4} = Q_{k_3 k_1 k_4 k_2}  
= T_{\sigma}(Q)_{k_1 k_2 k_3 k_4}, \quad \forall k_1, k_2, k_3, k_4.$$
Hence, $T_{\tau}(Q) = T_{\sigma}(Q)$ and they are both denoted by $T_{(S_1, S_2)}(Q)$.
% for every permutation $\sigma_1: S_1\to S_1$ and $\sigma_2: S_2\to S_2$. 
% Hence, the actions of all such permutations on $Q$ are the same and can be denoted $T_{(S_1, S_2)}(Q)$ for short. 
% {\bf \color{red} LN: this example needs corrections. E.g., "for short" for what? What is the role of $\tau$, $\tau_1,\tau_2$...what are $S_1$, $S_2$ here, the dimensions don't seem to match up.} {\color{blue} corrected.}

%%% Check redundant here and remove it. TO-DO
We can now express the probability functions associated with the LDA model and the mixture of product model, in terms of weighted outer product of the topic matrix with respect to appropriate tensors. For the mixture of product model,  Eq. \eqref{eq:mixture_product_density1} can be re-expressed as 
\begin{align}\label{eq:MM-tensor}
    p^{\mathscr{M}}_{\tilde{\alpha},\Theta}(x_{[N]}) = \diag_N(\tilde{\alpha})[\Theta^{\otimes N}](x_{[N]}).
\end{align}

Similarly, for the LDA, the probability function in Eq.  \eqref{eq:LDA_density1} can be re-expressed as
\begin{align}\label{eq:LDA-tensor}
    p_{\alpha, \Theta}^{\mathscr{L}}(x_{[N]}) = Q^{(N)}_{\alpha} [\Theta^{\otimes N}](x_{[N]}).
\end{align}

The moment tensor $Q_{\alpha}^{(N)}$ associated with the Dirichlet distribution has a simple closed-form expression. For any $(k_1, \dots, k_N) \in [K]^{N}$, the elements of $Q_{\alpha}^{(N)}$ can be explicitly computed by:
\begin{equation}\label{eq:Dirichlet-moment-tensor}
    Q_{\alpha, k_1 \dots k_N}^{(N)} = \Ebb_{q\sim \Dir_{\alpha}} [q_{k_1} \cdots q_{k_N}] = \Ebb_{q\sim \Dir_{\alpha}} \prod_{k=1}^{K} q_k^{N_k} = \dfrac{\prod_{k=1}^{K} \alpha_k^{[N_k]}}{\overline{\alpha}^{[N]}} 
\end{equation}
where $N_k = \# \{n \in [N]: k_n = k\}$ and $\overline{\alpha} = \sum_{k=1}^{K} \alpha_k$.

Our aim is to characterize a precise correspondence between the class of mixture of product models $p^{\mathscr{M}}_{\tilde{\alpha},\Theta}(\cdot)$ and the class of LDA models $p^{\mathscr{L}}_{\alpha,\Theta}(\cdot)$ for length-$N$ documents. The challenge in this aim lies in unraveling the algebraic structure of the probability function of the LDA, which is a weighted sum of $K^{N}$ tensors as seen in Eq.~\eqref{eq:LDA-tensor}. On the other hand, for mixture models~\eqref{eq:MM-tensor}, one needs to work with the sum of $K$ product tensors. Theoretical issues such as identifiability and posterior contraction are more well-understood in the latter setting, which is related to the uniqueness of tensors' decomposition. In the case $K = K'$, it can be solved by using the notion of Kruskal rank of $\Theta$~\cite{KRUSKAL197795,10.1214/09-AOS689}. 
%For the LDA model, we are working with a weighted sum of $K^{N}$ product tensors where these weights are controlled by $Q_{\alpha}^{(N)}$. It is natural to hope for a better presentation of the Dirichlet moment tensors. If we can represent it through the diagonal tensor of $\alpha$, we can borrow the knowledge of the literature on mixture models to understand the LDA models. We are going to work toward that goal for the rest of this section.
By exploiting a pair of decomposition formulae for the Dirichlet moment tensors, the correspondence between the two model classes can be established, and can be utilized to leverage the recent theoretical advances on the mixture models to address the hierarchical model setting represented by the LDA.

\subsection{On the representation of Dirichlet moment tensors by diagonal tensors and vice versa}
For small $N$ (up to 3), the aforementioned formulae is known in the literature \cite{anandkumar2012spectral,anandkumar2014tensor}. From Eq.~\eqref{eq:LDA-tensor} and~\eqref{eq:MM-tensor}, we can see that this relationship boils down to that whether one can write the moment tensors of Dirichlet distribution $Q^{(N)}_{\alpha}$ in terms of $\diag(\tilde{\alpha})$. For $\alpha \in \Rbb_+^{K}$, if we write $\tilde{\alpha}:= \alpha / \overline{\alpha} \in \Delta^{K-1}$, simple calculations yield:
\begin{equation}\label{eq:tensor-Q-as-diag}
    Q^{(1)}_{\alpha} = \tilde{\alpha}\in \Rbb^{K}, \quad Q^{(2)}_{\alpha} = \dfrac{1}{\overline{\alpha}(\overline{\alpha}+1)} \left(\diag_2(\alpha) + \alpha \otimes \alpha \right) \in \Rbb^{K\times K}.  
\end{equation}
Expressing inversely the diagonal tensors in terms of $Q_\alpha^{(1)}$ and $Q_\alpha^{(2)}$, it is simple to obtain
\begin{equation}\label{eq:tensor-diag-as-Q-12}
    \tilde{\alpha} = Q^{(1)}_{\alpha}, \quad \diag_2(\tilde{\alpha}) = \dfrac{1}{\overline{\alpha}} \left(\overline{\alpha}(\overline{\alpha}+1) Q^{(2)}_{\alpha} - \overline{\alpha}^2 Q^{(1)}_{\alpha} \otimes Q^{(1)}_{\alpha}\right).
\end{equation}
Anandkumar et. el.~\cite{anandkumar2012spectral, anandkumar2014tensor} provided the representation of diagonal tensor in terms of Dirichlet moment tensor when $N=3$ as:
\begin{equation}\label{eq:tensor-diag-as-Q-3}
\diag_3(\tilde{\alpha}) = \dfrac{1}{2\overline{\alpha}} \left(\overline{\alpha}^{[3]} Q^{(3)}_{\alpha} - \overline{\alpha}^{[2]} \overline{\alpha}(T_{(1, 2, 3)} + T_{(2, 3, 1)} + T_{(3, 1, 2)}) (Q^{(2)}_{\alpha}\otimes Q^{(1)}_{\alpha})+ 2\overline{\alpha}^{3} (Q^{(1)}_{\alpha})^{\otimes 3}\right),
\end{equation}
which is central to their analysis of a tensor-based algorithm for learning LDA model. Compared to their original presentation (Theorem 3.5 in~\cite{anandkumar2014tensor}), the permutation operator was used to make the formula concise and suggestive of generalization. Indeed, this relation can be extended for $N=4$ as follows:
\begin{align*}
&\diag_4(\tilde{\alpha})  = \dfrac{1}{6\overline{\alpha}} \bigg( \overline{\alpha}^{[4]}Q_{\alpha}^{(4)} - \overline{\alpha}^{[3]}\overline{\alpha} (T_{(1,2,3,4)} + T_{(1,2,4,3)} + T_{(1,4,3,2)}+ T_{(2,3,4,1)})(Q^{(3)}_{\alpha}\otimes Q^{(1)}_{\alpha}) \nonumber \\
&  - \overline{\alpha}^{[2]}\overline{\alpha}^{[2]} (T_{(1,2,3,4)} + T_{(1,3,2,4)} + T_{(1,4,2,3)})(Q^{(2)}_{\alpha}\otimes Q^{(2)}_{\alpha}) 
+ 2\overline{\alpha}^{[2]}\overline{\alpha}^{2} (T_{(1,2,3,4)} + T_{(1,3,2,4)} \nonumber \\
& + T_{(1,4,2,3)} + T_{(2,3,1,4)} + T_{(2,4,1,3)} + T_{(3,4,1,2)})(Q^{(2)}_{\alpha}\otimes Q^{(1)}_{\alpha}\otimes Q^{(1)}_{\alpha}) - 6\overline{\alpha}^{4} (Q^{(1)}_{\alpha})^{\otimes 4}\bigg). 
\end{align*}
To describe a generalization of this formula for arbitrary $N$, we need additional notation. For each $n\in [N]$, let $\Pcal(N, n)$ denote the set of all partitions of $\{1, 2, \dots, N\}$ into $n$ disjoint unordered non-empty subsets in that partition. We denote a generic partition $A\in \Pcal(N,n)$ as $A=(S_1,\dots,S_n)$, viewed as an unordered collection of the non-empty subsets. For example, $\Pcal(3, 1) = \{(\{1, 2, 3\})\}$ consists of a single partition, $\Pcal(3, 2) = \{(\{1, 2\}, \{3\}), (\{1, 3\}, \{2\}), (\{2, 3\}, \{1\})\}$ consists of 3 partitions and $\Pcal(3, 3) = \{(\{1\}, \{2\}, \{3\})\}$ consists of a single partition again.

% \begin{figure}
%     \centering
%     \includegraphics[width=0.7\linewidth]{figures/S4.png}
%     \caption{Illustration of $\Pcal(4, n)$ for $n\in [4]$ on the Hasse diagram of the partitions of $[4]$.}
%     \label{fig:S4}
% \end{figure}

\begin{lemma}\label{lem:whiten-Dirichlet-moment}
The moment tensors of Dirichlet distribution admit the following decomposition in terms of diagonal tensors
\begin{equation}\label{eq:whiten-Q-tensor-1}
    Q^{(N)}_{\alpha} = \dfrac{1}{\overline{\alpha}^{[N]}}\left[\sum_{n=1}^{N}\overline{\alpha}^{n} \sum_{(S_1, \dots, S_n)} \prod_{i=1}^{n} (|S_i|-1)! T_{(S_1, \dots, S_n)}(\otimes_{i=1}^{n}\diag_{|S_i|}(\tilde{\alpha}) )\right].
    \end{equation} 
Moreover, the diagonal tensors admit the following inverse decomposition
\begin{equation}\label{eq:whiten-Q-tensor-2}
\diag_{N}(\tilde{\alpha}) = \dfrac{1}{(N-1)! \overline{\alpha}} \left[\sum_{n=1}^{N} (-1)^{n-1} (n-1)!\sum_{(S_1, \dots, S_n)} \prod_{i=1}^{n} \overline{\alpha}^{[|S_i|]}  T_{(S_1, \dots, S_n)} \left(\otimes_{i=1}^{n}Q^{(|S_i|)}_{\alpha}\right)\right],
\end{equation}
where the sums over $(S_1, \dots, S_n)$ range over the set of partitions $\Pcal(N, n)$.
\end{lemma}
\begin{remark}
\begin{enumerate}
    \item[(i)] %Eq.~\eqref{eq:whiten-Q-tensor-1} represents the moment tensor of Dirichlet distribution as a combination of diagonal tensors, and Eq.~\eqref{eq:whiten-Q-tensor-2} is the opposite representation. 
    Eq.~\eqref{eq:whiten-Q-tensor-1} may be seen as a consequence of Prop. 4.7 in Ghosal and van der Vaart (2017)~\cite{ghosal2017fundamentals}. Eq.~\eqref{eq:whiten-Q-tensor-2}, which generalizes Eq.~\eqref{eq:tensor-Q-as-diag}, \eqref{eq:tensor-diag-as-Q-12} and \eqref{eq:tensor-diag-as-Q-3} to moments of arbitrary order, has not been reported elsewhere, to the best of our knowledge.
    \item[(ii)] The cardinality of $\Pcal(N, n)$ is the number of ways to partition $[N]$ to $n$ disjoint non-empty subsets, which is exactly the Stirling number of the second kind and denoted by $\stirlingtwo{N}{n}$. 
    \item[(iii)] The coefficients of tensor products in the RHS of Eq.~\eqref{eq:whiten-Q-tensor-1} correspond to an exchangeable partition probability function~\cite{pitman2006combinatorial, Lijoi2010-pg}, which arise in the \textit{Chinese Restaurant Process} or \textit{Ewens sampling distribution}~\cite{ewens1972sampling, 10.1214/15-STS529}. The appearance of this partition probability function reflects the fact that the LDA model is an instance of the Hierarchical Dirichlet Process model with finite supported base measure (more details are given in the sequel, cf. Eq.~\eqref{eq:hdp}).  
\end{enumerate}
\end{remark} 

A direct consequence of the Dirichlet moment tensor representation lemma is the equivalence between the space of mixture and LDA models, as we will present now. 

\subsection{Correspondence between spaces of LDA and mixture models}
%Before stating the main result of the correspondence of the LDA and mixture models, we need a notation for marginal probability functions. 
For every probability function $p$ on $[V]^{N}$ and $x = (x_1, x_2,\dots, x_n)\in [V]^{N}$, for each subset $S=\{s_1, \dots, s_r\}$ of $[N]$, the marginal of $p$ at the coordinate $S$ is denoted by $p(x_{S}) = p(x_{s_1}, \dots, x_{s_r})$. Note that both LDA and mixture models represent probability distributions on sequences of exchangeable observations, the order in $S$ does not matter to their corresponding probability functions, which are denoted by $p^{\mathscr{L}}$ and $p^{\mathscr{M}}$, respectively.  We arrive at our first main result on the correspondence of LDA models and mixture models.
\begin{theorem}\label{thm:equivalence-LDA-mixture}
    For any $\alpha = (\alpha_1, \dots, \alpha_K) \in \Rbb_{+}^{K}$, write  $\alpha = \overline{\alpha} \tilde{\alpha}$, with $\overline{\alpha} = \sum_{k=1}^K \alpha_k$ and $\tilde{\alpha} \in \Delta^{K-1}$. Then, for any $\Theta \in (\Delta^{V-1})^{K}$, we have
\begin{equation}\label{eq:write-LDA-as-mixture}
    p^{\mathscr{L}}_{\alpha, {\Theta}}(x) = \dfrac{1}{\overline{\alpha}^{[N]}}\left[\sum_{n=1}^{N}\overline{\alpha}^{n} \sum_{(S_1, \dots, S_n) \in \Pcal(N, n)} \prod_{i=1}^{n} (|S_i|-1)! {p}^{\mathscr{M}}_{\tilde{\alpha}, \Theta}(x_{S_i})\right].\end{equation} 
Conversely, 
\begin{equation}\label{eq:write-mixture-as-LDA}
{p}^{\mathscr{M}}_{\tilde{\alpha}, {\Theta}}(x) = \dfrac{1}{(N-1)! \overline{\alpha}} \left[\sum_{n=1}^{N} (-1)^{n-1} (n-1)!\sum_{(S_1, \dots, S_n) \in \Pcal(N, n)} \prod_{i=1}^{n} \overline{\alpha}^{[|S_i|]} {p}^{\mathscr{L}}_{\alpha, \Theta}(x_{S_i})\right].
\end{equation}
\end{theorem}   
\begin{proof}
    Taking the weighted outer products of both sides of Eq.~\eqref{eq:whiten-Q-tensor-1} (in Lemma~\ref{lem:whiten-Dirichlet-moment}) with $\Theta$ ($N$ times) yields:
    \begin{equation}\label{eq:eq:whiten-Q-tensor-1-imply}
        Q^{(N)}_{\alpha}[\Theta, \dots, \Theta] = \dfrac{1}{\overline{\alpha}^{[N]}}\sum_{n=1}^{N}\overline{\alpha}^{n} \sum_{(S_1, \dots, S_n)} \prod_{i=1}^{n} (|S_i|-1)! T_{(S_1, \dots, S_n)}(\otimes_{i=1}^{n}\diag_{|S_i|}(\tilde{\alpha}))[\Theta, \dots, \Theta],
    \end{equation}
    as $V\times \dots \times V$ ($N$ times) tensors. The left-hand side (LHS) is $p^{\mathscr{L}}_{\alpha, \Theta}$. Meanwhile, the outer products of $\diag_{|S_i|}(\tilde{\alpha})$ in the right-hand side (RHS) when got taken weighted outer product with $[\Theta, \dots, \Theta]$ become a product of marginals of $p^{\mathscr{M}}_{\tilde{\alpha}, \Theta}$. To see this algebraically, consider any partition $(S_1, \dots, S_n)$ of $[N]$ where $S_i = \{s_{i1}, \dots, s_{i p_i}\}$ and $x = (x_1, \dots, x_N)\in [V]^{N}$, we have
    \begin{align*}
        T_{(S_1, \dots, S_n)}(\otimes_{i=1}^{n}\diag_{|S_i|}(\tilde{\alpha}))[ \underbrace{\Theta, \dots, \Theta}_{N \text{times}}](x) & = \sum_{k_1, \dots, k_N=1}^{K} 
     \prod_{i=1}^{n} 1_{(k_{s_{i1}} = \dots = k_{s_{{i p_i}}})} \tilde{\alpha}_{k_{s_{i1}}} \prod_{j=1}^{p_i} \theta_{k_{s_{ij}} x_{s_{ij}}}\\
     & = \prod_{i=1}^{n} \sum_{k_{s_{i1}}=1}^{K} \tilde{\alpha}_{k_{s_{i1}}} \prod_{j=1}^{p_i} \theta_{k_{s_{i1}} x_{{s_{ij}}}} \\
     & = \prod_{i=1}^{n} p_{\overline{\alpha}, \Theta}^{\mathscr{M}}(x_{S_i}).
    \end{align*}
    Hence, Eq.~\eqref{eq:eq:whiten-Q-tensor-1-imply} is exactly Eq.~\eqref{eq:write-LDA-as-mixture}. Similarly, Eq.~\eqref{eq:write-mixture-as-LDA} follows from Eq.~\eqref{eq:whiten-Q-tensor-2}.
    % Recall the representation 
    % $$p^{\mathscr{L}}_{\alpha, \Theta}(x_S) = Q^{|S|}_{\alpha}[\Theta, \dots, \Theta] \quad \text{and} \quad p^{\mathscr{M}}_{\tilde{\alpha}, \Theta}(x_S) = \diag_{|S|}(\tilde{\alpha})[\Theta, \dots, \Theta],$$
    % where there are $|S|$ numbers of $\Theta$ in the brackets, for any subset $S$ of $[N]$. Hence, this is a direct consequence of Lemma~\ref{lem:whiten-Dirichlet-moment}, where we consider the weighted outer product $[\Theta, \dots, \Theta]$ ($N$ times) with the weights being both sides of Eq.~\eqref{eq:whiten-Q-tensor-1} and~\eqref{eq:whiten-Q-tensor-2}.
\end{proof}

\begin{remark}
    Both sides of the equations~\eqref{eq:write-mixture-as-LDA} and ~\eqref{eq:write-LDA-as-mixture} do not depend on $K$, so it will be a useful device to study the identifiability and estimation rate of the LDA model with possibly misspecified number of topics $K$, by leveraging results for mixture models. 
\end{remark}

A consequence of Theorem ~\ref{thm:equivalence-LDA-mixture} is that when $\overline{\alpha}$ is known, many common metrics or divergences defined on the space of mixture and LDA models are equivalent.
\begin{proposition}\label{prop:metric-equivalent}
\begin{enumerate}
    \item[(a)] Let $d$ be either $d_{TV}, d_{KL}$ or $ d_{H}^{2}$. For any $\alpha\in \Rbb_{+}^{K}, \alpha' \in \Rbb_{+}^{K'}$ having $\overline{\alpha} = \overline{\alpha}'$ and $\Theta \in (\Delta^{V-1})^{K}, \Theta' \in (\Delta^{V-1})^{K'}$, we have
        \begin{equation}\label{eq:L_leq_M}
            d(p^{\mathscr{L}}_{\alpha, \Theta}, p^{\mathscr{L}}_{\alpha', \Theta'}) \leq C_{1}(N, \overline{\alpha}) d(p^{\mathscr{M}}_{\tilde{\alpha}, \Theta}, p^{\mathscr{M}}_{\tilde{\alpha}', \Theta'}),
        \end{equation}
        where $C_1(N, \overline{\alpha})$ only depends on sequence length $N$ and concentration parameter $\overline{\alpha}$.

        \item[(b)] Conversely, we have
        \begin{equation}\label{eq:M_leq_L}
        d_{TV}(p^{\mathscr{M}}_{\tilde{\alpha}, \Theta},p^{\mathscr{M}}_{\tilde{\alpha}', \Theta'}) \leq C_2(N, \overline{\alpha}) d_{TV}(p^{\mathscr{L}}_{\alpha, \Theta}, p^{\mathscr{L}}_{\alpha', \Theta'}),
        \end{equation}
        where $C_2(N, \overline{\alpha})$ only depends on $N$ and  $\overline{\alpha}$.
    \end{enumerate}
\end{proposition}
\begin{proof}
    \begin{enumerate}
        \item[(a)] Recall that for any $d\in \{d_{TV}, d_{KL}, d_{H}^{2}\}$, the following properties hold:
        \begin{enumerate}
            \item[(i)] $d(\cdot, \cdot)$ is a convex function of both arguments;
            \item[(ii)] $d(p(x_{S}), q(x_{S})) \leq d(p, q)$ for any distributions $p, q$ on $[V]^{N}$ and subset $S$ of $[N]$; 
            \item[(iii)] $d(p_1 \otimes p_2, q_1 \otimes q_2) \leq d(p_1, q_1) + d(p_2, q_2)$ for any distributions $p_1, p_2, q_1, q_2$.
        \end{enumerate}
        The last inequality becomes equality for $d = d_{KL}$. For completeness, a proof of these facts is given in Lemma 5 in Appendix F. Because of property (ii), it holds
        $$d(p^{\mathscr{M}}_{\tilde{\alpha}, \Theta}(x_{S}), p^{\mathscr{M}}_{\tilde{\alpha}', \Theta'}(x_{S})) \leq d(p^{\mathscr{M}}_{\tilde{\alpha}, \Theta}, p^{\mathscr{M}}_{\tilde{\alpha}', \Theta'}) \quad \forall S\subset [N].$$
        Moreover, because of (iii), for any partition $S_1, \dots, S_n$ of $[N]$, 
        $$d\left(\prod_{i=1}^{n} p^{\mathscr{M}}_{\tilde{\alpha}, \Theta}(x_{S_i}), \prod_{i=1}^{n} p^{\mathscr{M}}_{\tilde{\alpha}', \Theta'}(x_{S_i})\right) \leq \sum_{i=1}^{n} d\left(p^{\mathscr{M}}_{\tilde{\alpha}, \Theta}(x_{S_i}), p^{\mathscr{M}}_{\tilde{\alpha}', \Theta'}(x_{S_i})\right)\leq n \times d(p^{\mathscr{M}}_{\tilde{\alpha}, \Theta}, p^{\mathscr{M}}_{\tilde{\alpha}', \Theta'}).$$
    Now, by the fact that $\dfrac{1}{\overline{\alpha}^{[N]}}\sum_{n=1}^{N} \overline{\alpha}^n \sum_{(S_1, \dots, S_n)\in \Pcal(N, n)}\prod_{i=1}^{n} (|S_i|-1)! = 1$, it follows from Eq. \eqref{eq:write-LDA-as-mixture} in Theorem~\ref{thm:equivalence-LDA-mixture} and the convex property (i) that
    \begin{align*}
        d(p^{\mathscr{L}}_{\alpha, \Theta}, p^{\mathscr{L}}_{\alpha', \Theta}) & \leq \dfrac{1}{\overline{\alpha}^{[N]}}\sum_{n=1}^{N} \overline{\alpha}^n \sum_{(S_1, \dots, S_n)}\prod_{i=1}^{n} (|S_i|-1)! \times d\left(\prod_{i=1}^{n} p^{\mathscr{M}}_{\tilde{\alpha}, \Theta}(x_{S_i}), \prod_{i=1}^{n} p^{\mathscr{M}}_{\tilde{\alpha}', \Theta'}(x_{S_i})\right)\\
        &\leq \dfrac{1}{\overline{\alpha}^{[N]}}\sum_{n=1}^{N} \overline{\alpha}^n \sum_{(S_1, \dots, S_n)}\prod_{i=1}^{n} (|S_i|-1)! \times n \times d\left( p^{\mathscr{M}}_{\tilde{\alpha}, \Theta}, p^{\mathscr{M}}_{\tilde{\alpha}', \Theta'}\right)\\
        & = C_1(N, \overline{\alpha}) d\left( p^{\mathscr{M}}_{\tilde{\alpha}, \Theta}, p^{\mathscr{M}}_{\tilde{\alpha}', \Theta'}\right),
    \end{align*}
    where 
    \begin{equation*}
        C_1(N, \overline{\alpha}) = \dfrac{1}{\overline{\alpha}^{[N]}}\sum_{n=1}^{N} n \overline{\alpha}^n \sum_{(S_1, \dots, S_n)}\prod_{i=1}^{n} (|S_i|-1)!.
    \end{equation*}
    \item[(b)] Analogous to the argument above and using the triangle inequality instead of convexity of $d$, we have
    \begin{align*}
    d_{TV}(p^{\mathscr{M}}_{\tilde{\alpha}, {\Theta}}, p^{\mathscr{M}}_{\tilde{\alpha}', \Theta'}) & 
    \leq \dfrac{1}{(N-1)! \overline{\alpha}}\left[\sum_{n=1}^{N} (n-1)! \sum_{(S_1, \dots, S_n)} \prod_{i=1}^{n} \overline{\alpha}^{[|S_i|]} d_{TV}(p^{\mathscr{L}}_{\alpha, \Theta}(x_{S_i}), p^{\mathscr{L}}_{\alpha', \Theta'}(x_{S_i}))\right]\\
    & \leq C_2(N, \overline{\alpha})d_{TV}(p^{\mathscr{L}}_{\alpha, \Theta}, p^{\mathscr{L}}_{\alpha', \Theta'}),
    \end{align*}
    in which $C_2(N, \overline{\alpha}) = \dfrac{1}{(N-1)! \overline{\alpha}} \sum_n (n-1)!\sum_{(S_1, \dots, S_n)}\prod_{i=1}^{n} \overline{\alpha}^{[|S_i|]}$ is a constant that only depends on $N$ and $\overline{\alpha}$. 
\end{enumerate}
\end{proof}
\begin{remark}\label{rmk:geometry-LDA-MM}
\begin{enumerate}
    \item[(i)] For all $N$ and $\overline{\alpha}$, an upper bound for $C_1(N, \overline{\alpha})$ is readily available:
    $$C_1(N, \overline{\alpha}) \leq N \dfrac{1}{\overline{\alpha}^{[N]}}\sum_{n=1}^{N} \overline{\alpha}^n \sum_{(S_1, \dots, S_n)}\prod_{i=1}^{n} (|S_i|-1)! = N.$$
    Meanwhile, the constant $C_2(N, \overline{\alpha})$ involves a sum over all partitions of $[N]$ and is harder to bound by a simple expression of $N$ and $\overline{\alpha}$. To understand the tightness of those bounds of $C_1$ and $C_2$, we have conducted simulation studies in Appendix~\ref{app:metric-equivalent} for various settings of $N$, $\overline{\alpha}$ and $K$. When $\overline{\alpha} < 1$, we find that $C_1$ and $C_2$ are both close to 1, which indicates that tighter theoretical bounds might be obtained.
    % By noticing that $C_1(N, \overline{\alpha})$ is the expected number of sets in the random partition of $\{1, \dots, N\}$ with the Ewens distribution, we can explicitly calculate\footnote{Thanks to Dr. Yun Wei for pointing out this improved bound} (see, e.g.,~\cite{10.1214/15-STS529})
    % $$C_1(N, \overline{\alpha}) = \dfrac{\overline{\alpha}}{\overline{\alpha}} + \dfrac{\overline{\alpha}}{\overline{\alpha}+1} +\dots + \dfrac{\overline{\alpha}}{\overline{\alpha}+N-1} \asymp \overline{\alpha} \log(N)$$
    % as $N\to \infty$.
    \item[(ii)] This proposition will prove useful for establishing the identifiability of the LDA model (cf. Corollary~\ref{cor:equivalence-identifiability}), and moreover for investigating its parameter contraction behavior using the following approach. First, the posterior contraction rates for the density functions may be obtained using standard techniques (cf.~\cite{ghosal2017fundamentals}), where the main requirements are an upper bound on the covering number of the space of LDA models and a lower bound of the KL ball around the true model. Both can be provided by combining~\eqref{eq:L_leq_M} with the existing bounds for mixtures of products~\cite{nguyen2016borrowing, wei2022convergence} (see Section~\ref{sec:contraction-rate}). Second, to transfer the density contraction rate to that of the parameter, we are going to lower bound the distance between LDA's densities by the distance between parameters. This step is considerably challenging, but thanks to~\eqref{eq:M_leq_L}, it can be done by lower bounding the distance between mixtures of products' densities instead, which is somewhat more amenable. A collection of those inverse bounds in different settings will be presented in Section~\ref{sec:identifiability}.
    \item[(iii)] It is natural to ask whether the multiplicative constants in these inequalities make the posterior contraction rates of the LDA model worse than the mixture of products, especially in the regime where the document length $N$ is large. We highlight that in the first step of showing density contraction rate, the multiplicative constant in inequality~\eqref{eq:L_leq_M} only affects the covering number of the space of LDA models $C_1(N,\overline{\alpha})$ times larger than the space of mixtures of products, which is $C_1(N,\overline{\alpha})$ times larger than the space of parameters (see, e.g.~\cite{nguyen2016borrowing, wei2022convergence}). Hence, the Le Cam dimension of the model space~\cite{lecam1973convergence, Van_der_Vaart2002-yo}, which is practically the logarithm of the covering number, only worsens by $\log(C_1(N,\overline{\alpha}))$, which only affects the density contraction rate by that same amount. In the second step of lower bounding the distance between densities, it suffices to use the constant $C_2(N, \overline{\alpha})$ for the minimally admissible $N$ for which the inverse bounds of mixtures of products can be obtained. Such $N$ can be as small as $2K-1$ in the general setting and as 3 or 4 in the setting where topics are linearly independent.  
    % \item[(iii)] Note that in part (b) of Proposition~\ref{prop:metric-equivalent}, the alternatively signed sum makes the application of convexity inequalities difficult, and we need to use the triangle inequality as a resort. Hence, the constant $C(N, \overline{\alpha})$ derived by this proof technique is not necessarily sharp, and it does not work for $d_{KL}$. However, it does not affect the results we are building later. 
\end{enumerate}
\end{remark}

\subsection{Recursive formulae for computing moments of linear transformations of Dirichlet's}
Before investigating the statistical properties of the LDA model, we want to present another application of Lemma~\ref{lem:whiten-Dirichlet-moment} in calculating the moments of linear transformations of Dirichlet distributions. Given $q = (q_1, \dots, q_K) \sim \Dir_{\alpha}$ for $\alpha = (\alpha_1, \dots, \alpha_K)\in \Rbb_{+}^{K}$ and a vector $x \in \Rbb^{K}$. We are interested in an exact formula to compute the moment: 
\begin{equation}\label{eq:moment-linear-Dirichlet}
    \mathfrak{M}_N := \Ebb (\innprod{q, x})^{N} = \Ebb (q_1 x_1 + \dots + q_K x_K)^{N},
\end{equation}
for any $N\in \Nbb$. Because $\innprod{q, x}$ does not have a close form distribution, directly computing this expression can be challenging. Therefore, we are developing a recursive formula instead. 
Recall that $\alpha = \overline{\alpha} \tilde{\alpha}$ for $\overline{\alpha} = \sum_k \alpha_k$, and we denote $\overline{x}_{(d)} = \sum_{k=1}^{K} \tilde{\alpha}_k x_k^{d}$ for all $d\in \Nbb$.  
\begin{proposition}\label{prop:recursive-compute-moment-Dir}
    For any $\alpha\in \Rbb_+^{K}$ and $x\in \Rbb^{K}$, denote $\mathfrak{M}_N$ be the $N-$th order moment as in Eq.~\eqref{eq:moment-linear-Dirichlet}, we have
    \begin{equation}\label{eq:recursive-compute-moment-Dir}
        c_{N+1}\mathfrak{M}_{N+1} = \dfrac{1}{N+1} \sum_{\ell=0}^{N} \overline{x}_{(\ell+1)} c_{N-\ell}\mathfrak{M}_{N-\ell}, \quad \forall N\in \Nbb,
    \end{equation}
    where the constants $c_{n} := \overline{\alpha}^{[n]} / n!$ for any $n\in \Nbb$, and $\mathfrak{M}_0 \equiv 1$. 
\end{proposition}
When the concentration parameter $\overline{\alpha} = 1$, we have $c_n = 1$ for all $n$ so that we can drop $c_n$'s on both sides of the recursive relation above for simplicity. It can also be shown that $c_n \asymp n^{\overline{\alpha} - 1}$ as $n\to \infty$. The proof of Proposition~\ref{prop:recursive-compute-moment-Dir} is presented in the appendix, where the key is to utilize the tensor representation 
$\Ebb (\innprod{q, x})^{N} = Q_{\alpha}^{(N)}[x, x,\dots, x]$ and apply Lemma~\ref{lem:whiten-Dirichlet-moment} to expand $Q_{\alpha}^{(N)}$ out. The moment $\mathfrak{M}_N$ is then written as a polynomial of $\overline{x}_1, \dots, \overline{x}_N$, which has a strong connection to the exponential Bell polynomial~\cite{charalambides2002enumerative} but uses the unsigned Stirling number of the first kind instead of the second kind as in the Bell polynomial. From there, we derive the generating function for this polynomial and apply it to get the recursive formula~\eqref{eq:recursive-compute-moment-Dir}. A simulation study for validating this result is presented in Section~\ref{sec:experiment}.  

\subsection{Remarks on related work}\label{sec:existing work} To make reliable inferences, the first requirement for a probabilistic model is identifiability with respect to its parameters. For the mixture of product models, it is shown that one observation for each data ($N=1$) is not enough to ensure identifiability \cite{gyllenberg1994non}. Classical results of Kruksal~\cite{KRUSKAL197795} provide a criterion for the identifiability of the mixture of products of multinomials when $N = 3$. Allman et al.~\cite{10.1214/09-AOS689} further extended this result to the case of higher $N$ and showed that the model is identifiable in the \emph{generic sense}, i.e., for almost surely all parameters. Vandermeulen and Scott~\cite{vandermeulen2019operator} investigated the identifiability for finite mixtures of product models in both settings where the number of components is either known or over-fitted.

The strategy of using whitening (or debiasing/denoising) equations (essentially moment decomposition) for learning hierarchical models has been around in the literature; in the context of LDA, Anandkumar et al.~\cite{anandkumar2012spectral, anandkumar2014tensor} showed how to transform an LDA model with three words per document to a mixture model (cf. Eqs. \eqref{eq:tensor-diag-as-Q-12} and \eqref{eq:tensor-diag-as-Q-3}). The parameters are then learned using a spectral method for tensors. They also proposed similar techniques for Gaussian Mixture Models and Independent Component Analysis. Higher-order whitening equations for Gaussian Mixture Models were developed in~\cite{10.1214/19-AOS1873, wei2023minimum} for one-dimensional and in~\cite{pereira2022tensor} for multi-dimensional data. Our approach for identifiability and posterior contraction analysis using
the general decomposition formulas representing Dirichlet moment tensors by the diagonal tensors and vice versa in arbitrary moment is different, as it allows us to leverage the advances in \cite{vandermeulen2019operator} and \cite{wei2022convergence} to obtain the new results under minimal conditions in identifiability and parameter estimation for the LDA model, among those which have been obtained in the literature. We turn to an overview of such relevant and existing results.

\subsubsection*{Different approaches to topic modeling}
Broadly speaking there are two main ways to make inferences in topic modeling: (i) Treating $q$ as a latent variable stochastically generated from a learnable parametric distribution, e.g., Dirichlet distribution (thereby the name \textit{Latent Dirichlet Allocation})~\cite{blei2003latent} (see also \cite{pritchard2000inference}), and (ii) Treating $q$ as a learnable parameter and putting an additional assumption on the true topics known as "anchor-word" condition~\cite{arora2012learning, arora2013practical, ke2024using, bing2020fast}. 
We refer to the former as \textit{"probabilistic formulation"} and latter as \textit{"matrix factorization formulation"} to topic modeling.
The probabilistic formulation is adopted in this work, which is motivated by the probabilistic assumption that words within each document, as random variables, are exchangeable. Therefore, they are conditionally i.i.d. given some latent variable $q$ by de Finetti's theorem~\cite{aldous1985exchangeability, blei2003latent, wei2022convergence} and the collection of counts (of words from the vocabulary in a document) can be modeled as
\begin{equation}\label{eq:exchangable-intro}
    p_{\nu, \Theta}(X_{[N]}) = \int \prod_{n=1}^{N} \Multi\left(X_n \bigg| \sum_{k=1}^{K} q_k \theta_k \right) d\nu(q),
\end{equation}
where $\nu$ is so-called an admixing distribution on the simplex $\Delta^{K-1}$, the collection of topics is denoted by $\Theta = (\theta_1, \dots, \theta_K)$, and $\Multi$ is the Multinomial distribution with one trial. Given a set of documents being i.i.d. generated from this distribution, our central interest is to make inferences about the topics $\Theta$ and distribution $\nu$. When both $\nu$ and $K$ are known, Nguyen~\cite{nguyen2015} provided the posterior contraction rates for the convex hull of topics as both document length and number of documents tend to infinity using a convex geometry approach. The rates are shown to depend only on the geometric relation of the true topics and the regularity of $\nu$, even though no specific parametric form of $\nu$ is assumed. Wang~\cite{10.1214/18-EJS1516} studied the convergence rate of the MLE of topics as the number of documents diverges while fixing the document length to be as small as $2$ by relying on conditions of the model's moments. When each document has at least two words, it is shown that there exists at most finitely many topic parameters $\Theta'$ such that $p_{\nu, \Theta} = p_{\nu, \Theta'}$, and the MLE converges to one of them at the rate $m^{-1/4}$. When further assuming $\Theta$ is linearly independent and $N\geq 3$, the model is proved to be identifiable, and the MLE rate is improved to the usual parametric rate $m^{-1/2}$. A similar behavior was obtained by Anandkumar et al.~\cite{anandkumar2012spectral, anandkumar2014tensor} for a different estimator of $\Theta$ using a spectral algorithm. 

In the matrix factorization formulation, the mixing proportion $q_i \in \Delta^{K-1}$ for the observed document $X^{i}_{[N]}$ are also treated as model parameters for $i \in [m]$. The model becomes
\begin{equation}\label{eq:deterministic-q}
    X_{1}^{i}, \dots, X_{N}^{i} \overset{iid}{\sim} \Multi\left(\sum_{k=1}^{K} q_{ik} \theta_k \right),
\end{equation}
and was first introduced as Latent Semantic Analysis~\cite{landauer1997solution, landauer1998introduction}. If we write $Q = (q_{ik})_{i, k=1}^{m, K} \in \Rbb^{m\times K}$ and $\Theta = (\theta_1, \dots, \theta_K)\in \Rbb^{K\times V}$, then the matrix $\Pi := Q \Theta\in \Rbb^{m\times V}$ completely identifies the model, as $X^{i}_{j} \sim \Multi(\Pi_{i})$ for $\Pi_i$ being the $i-$th row of $\Pi$ and $i\in [m]$. Hence, in this setting, the identifiability becomes the question of whether we can get a unique factorization $\Pi = Q\Theta$ given $\Pi \in \Rbb^{m\times V}$ and is often guaranteed by assuming the \emph{anchor-word condition} on $\Theta$ and/or $H$ \cite{donoho2003does, arora2012learning, arora2013practical}. Given this strong, arguably unrealistic, condition, several learning algorithms were proposed with provable estimation rate~\cite{bing2020fast, bing2020optimal, ke2024using, wu2023sparse, tran2023sparse}. A set of geometric conditions %called \emph{sufficiently scattered} 
were proposed to show the identifiability of matrix factorization $\Pi = Q\Theta$, given further that the volume of the convex hull of $\theta_1, \dots, \theta_K$ is minimized \cite{chen2022learning, huang2016anchor, jang2019minimum}. 

Although there have been considerable efforts in establishing asymptotic theories of topic modeling in the past decade, most results have been obtained for the matrix factorization formulation~\cite{arora2012learning, arora2013practical, bing2020fast, bing2020optimal, ke2024using, wu2023sparse, tran2023sparse, chen2022learning, huang2016anchor, jang2019minimum}. The probabilistic approach, as epitomized by the LDA model of \cite{blei2003latent} which has collected over 57000 citations on Google Scholar as of today, is probably more widely used in practice, but is theoretically challenging to analyze. This is attested by the scarcity of existing results~\cite{nguyen2015, 10.1214/18-EJS1516, anandkumar2012spectral}, perhaps due to presence of the latent random allocation variables $q$. By exploiting the Dirichlet moment tensors' representations (Lemma~\ref{lem:whiten-Dirichlet-moment}) and the entailed correspondence between the space of LDA models and the space of product mixtures (Theorem \ref{thm:equivalence-LDA-mixture}), we now arrive at a systematic treatment of parameter identifiability and parameter estimation for the former class of models, but leveraging the recently developed results and techniques for the latter.

\section{Identifiability of parameters in LDA}\label{sec:identifiability}
In this section, we present a systematic treatment of parameter identifiability that arises in the LDA models (convergence rates of parameter estimation will be treated in Section \ref{sec:contraction-rate}). The correspondence between the space of LDA models and that of mixture of product models has made it sufficient to work with the latter first and transfer all results to the LDA models. Although some identifiability results are borrowed from the literature on mixture models, many more are newly obtained.
% We will pay special attention to the over-fitted setting because researchers often fit the LDA model with a large number of topics in real-life applications, and it is desired to understand the parameter estimation in this regime.
\subsection{Identifiability}
There are two parts of parameters in the LDA model: the Latent Dirichlet distribution's parameter $\alpha \in \Rbb_{+}^{K}$ and the word-topic distributions $\theta_1, \dots, \theta_K \in \Delta^{V-1}$. From now on, we always assume the concentration parameter $\overline{\alpha} = \sum_{k=1}^{K} \alpha_k$ is known and fixed so that it is only needed to learn $\tilde{\alpha} = \alpha / \overline{\alpha} \in \Delta^{K-1}$. To avoid the label-switching problem of topics, all parameters of the LDA model are represented by the discrete (latent) measure $G = \sum_{k=1}^{K} \tilde{\alpha}_k \delta_{\theta_k}$, so
that the corresponding probability function may be written as
\begin{equation}\label{eq:redenote-LDA}
    p_{G, N}^{\mathscr{L}}(x) := p^{\mathscr{L}}_{\alpha, \Theta}(x) = \sum_{k_1, \dots, k_N=1}^{K} Q^{(N)}_{\alpha}(k_1, \dots, k_N) \theta_{k_1 x_1} \dots \theta_{k_N x_N}\quad \forall x\in [V]^{N},
\end{equation}
and
\begin{equation}\label{eq:redenote-MM}
    p_{G, N}^{\mathscr{M}}(x) := p^{\mathscr{M}}_{\tilde{\alpha}, \Theta}(x) = \sum_{k=1}^{K} \tilde{\alpha}_k \theta_{k x_1} \dots \theta_{k x_N}\quad \forall x\in [V]^{N}.
\end{equation}
A succinct way of expressing the LDA model in terms of the parameter $G = \sum_{k=1}^{K}\tilde{\alpha}_k \delta_{\theta_k}$, is via a hierarchical Dirichlet Process: 
\begin{align}\label{eq:hdp}
\begin{cases}
    Q = \sum_{k=1}^{K} q_k \delta_{\theta_k} \sim \text{DP}(\overline{\alpha}G)\\
    X_1, \dots, X_N | Q  \overset{iid}{\sim} \Multi(x | \theta) dQ(\theta),
\end{cases}
\end{align}
where $\text{DP}(\overline{\alpha}G)$ is the Dirichlet Process with mean measure $G$ and concentration parameter $\overline{\alpha}$~\cite{ferguson1973bayesian, teh2010hierarchical}. Hence, the notation $p^{\mathscr{L}}_{G, N}$ is well defined. 
% Because our main concerns are mostly about the LDA models, we drop the superscript $\mathscr{L}$ for ease of notation. Similarly, denote:
% \begin{equation}\label{eq:redenote-MM}
%     p_{G, N}^{\mathscr{M}}(x) := p^{\mathscr{M}}_{\tilde{\alpha}, \Theta}(x) = \sum_{k=1}^{K} \tilde{\alpha}_{k} \theta_{k x_1} \dots \theta_{k x_N}\quad \forall x\in [V]^{N}.
% \end{equation}
The appearance of the document length $N$ in the subscript of $p$ is noteworthy because it is well established for product mixture models that identifiability conditions are satisfied when $N$ exceeds a certain threshold \cite{10.1214/09-AOS689, vandermeulen2019operator}, and we will show that it is also the case for the LDA models. Let $\Ecal_{K}$ be the space of discrete probability measures with exactly $K$ atoms and $\Ocal_{K} = \cup_{k=1}^{K} \Ecal_{k}$ the space of discrete probability measures with no more than $K$ atoms in $\Delta^{V-1}$. 
\begin{definition}[Exact and over-fitted identifiability]\label{def:identifiability} For a discrete measure $G_0 \in \Ecal_{K_0}$, we say that $p^{\mathscr{L}}_{G_0, N}$ satisfies the exact-fitted identifiability (resp., over-fitted identifiability up to order $K$) condition if $p^{\mathscr{L}}_{G_0, N} = p^{\mathscr{L}}_{G, N}$ implies $G = G_0$ for any $G \in \Ocal_{K_0}$ (resp., $G \in \Ocal_{K}$ for some $K > K_0$).
\end{definition}
Exact- and Over-fitted identifiability for $p^{\mathscr{M}}$ are defined in the same way, i.e., which insists that $p^{\mathscr{M}}_{G_0, N} = p^{\mathscr{M}}_{G, N}$ implies $G = G_0$.
A direct corollary of the equivalence between the LDA and mixture models (Theorem~\ref{thm:equivalence-LDA-mixture}) paves the way for extending identifiability results of product mixture models to the class of LDA models. It is summarized as follows.

\begin{corollary}\label{cor:equivalence-identifiability} For any $N\geq 1$ and $G_0 \in \mathcal{E}_{K_0}$, $p^{\mathscr{L}}_{G_0, N}$ satisfies the exact-fitted identifiability (resp., over-fitted identifiability up to $K$ topics) if and only if $p^{\mathscr{M}}_{G_0, N}$ satisfies the exact-fitted identifiability (resp., over-fitted identifiability up to $K$ components).
% For any $\bar{\alpha}>0$ and $N\geq 1$, $\mathcal{P}^{{\mathscr{M}}}_N$ is identifiable iff $\mathcal{P}^{\mathscr{L}}_{\overline{\alpha},N}$ is identifiable.
\end{corollary}

\begin{proof}
We present the proof for the case of over-fitted identifiability, while the exact-fitted case can be done similarly.
Suppose $G_0 = \sum_{k\in[K_0]} \tilde{\alpha}_k^0\delta_{\theta_k^0}$. Consider any $G=\sum_{k\in[K]}\tilde{\alpha}_k\delta_{\theta_k}\in \Ecal_K$ for some $K>K_0$ such that $p_{G,N}=p_{G_0,N}$, i.e., $p_{\alpha,\Theta}^{\mathscr{L}}(x) = p_{\alpha^0, \Theta^0}^{\mathscr{L}}(x)$ for all $x\in [V]^N$ (using the notations from Section \ref{sec:main-results}), where $\alpha^0=\bar{\alpha}\tilde{\alpha}^0$ and $\alpha=\bar{\alpha}\tilde{\alpha}$. Because two distributions being equal implies all marginals are equal, by plugging both of them in Eq.~\eqref{eq:write-mixture-as-LDA}, we have $p_{\tilde{\alpha},\Theta}^{\mathscr{M}}(x)=p_{\tilde{\alpha}^0,\Theta^0}^{\mathscr{M}}(x)$ for all $x\in[V]^N$, i.e., $P_{G,N}^{\mathscr{M}}=P_{G_0,N}^{\mathscr{M}}$. Similarly, if $P_{G,N}^{\mathscr{M}}=P_{G_0,N}^{\mathscr{M}}$, i.e., $p_{\tilde{\alpha},\Theta}^{\mathscr{M}}(x) = p_{\tilde{\alpha}^0,\Theta^0}^{\mathscr{M}}(x)$ for all $x\in[V]^N$, then using Eq.~\eqref{eq:write-LDA-as-mixture} we see that $P_{G,N}=P_{G_0,N}$. Thus, by definition of over-fitted identifiability, it follows that $P_{G_0,N}$ satisfies identifiability if $P_{G_0,N}^{\mathscr{M}}$ satisfies it and vice-versa.

% Consider any two sets of parameter $(\alpha, \Theta)\in \Rbb_+^{K} \times (\Delta^{V-1})^{K}$ and $(\alpha', \Theta')\in \Rbb_+^{K'} \times (\Delta^{V-1})^{K'}$ such that 
% $$\sum_{j=1}^{K} \alpha_j = \sum_{j=1}^{K'} \alpha_j' \quad \text{and} \quad p^{\mathscr{L}}_{\alpha, \Theta}(x) = p^{\mathscr{L}}_{\alpha', \Theta'}(x)\, \forall x\in [V]^{N}.$$
% Because two distribution equal implies all marginals equal, by plugging both of them in Eq.~\eqref{eq:write-mixture-as-LDA}, we have $p^{\mathscr{M}}_{\tilde{\alpha}, \Theta}(x) = p^{\mathscr{M}}_{\tilde{\alpha'}, \Theta'}(x) \forall x$. The inverse is also correct that $p^{\mathscr{M}}_{\tilde{\alpha}, \Theta} = p^{\mathscr{M}}_{\tilde{\alpha'}, \Theta'}$ implies $p^{\mathscr{L}}_{{\alpha}, \Theta} = p^{\mathscr{L}}_{{\alpha'}, \Theta'}$, using Eq.~\eqref{eq:write-LDA-as-mixture}. Hence, there is a one-to-one correspondence between the class of LDA models with fixed concentration parameter $\{p_{\alpha, \Theta} : \alpha \in \mathbb{R}_{+}^{K}, \theta\in (\Delta^{V-1})^{K} , \sum \alpha_k = \overline{\alpha}\}$ and the class of mixture models $\{\tilde{p}_{q, \theta} : q \in \Delta^{K-1}, \theta\in (\Delta^{V-1})^{K}\}$. Hence, the identifiability of one class can imply the other.
\end{proof}

Since $p_{G_0, N'}^{\mathscr{L}} = p_{G, N'}^{\mathscr{L}}$ implies $p_{G_0, N}^{\mathscr{L}} = p_{G, N}^{\mathscr{L}}$ for all $N' \geq N$ (by marginalizing), $p_{G_0, N'}^{\mathscr{L}}$ satisfies the identifiability conditions whenever $p_{G_0, N}^{\mathscr{L}}$ does. Hence, we are interested in the minimum $N$ for which the identifiability conditions remain satisfied.
\begin{definition}
    For a discrete measure $G_0$, let $\mathfrak{N}_{e}(G_0)$ (resp., $\mathfrak{N}_{o}(G_0, K)$) be the minimum document length $N$ such that $p_{G_0, N}$ satisfies the exact-fitted (resp., over-fitted up to order $K$) identifiability condition in Definition~\ref{def:identifiability}. 
\end{definition}
In practice, knowing an upper bound of $\mathfrak{N}_{e}(G_0)$ (and $\mathfrak{N}_{o}(G_0, K)$) is useful because it provides a sufficient condition for identifiability of the parameters of interest. Suppose that there are $m$ i.i.d. documents generated from a true model $p_{G_0, N}$ with $G_0 = \sum_{k=1}^{K_0} \tilde{\alpha}^{0}_k \delta_{\theta_k^{0}} \in \Ecal_{K_0}$ and we fit it with $K > K_0$ topics to obtain an estimate $\widehat{G}_m = \sum_{j=1}^{K} \widehat{\alpha}_j \delta_{\widehat{\theta}_j} \in \Ocal_K$. If the over-fitted identifiability is satisfied, we at least can expect that $\widehat{G}_{m}$ consistently estimate $G_0$ in a suitable metric for measures \citep{nguyen2013convergence}. Then, it follows that there are some over-fitted \textit{redundant} topics $\widehat{\theta}_i$ and $\widehat{\theta}_j$ that are close to the same true topic $\theta_k^0$ or some \textit{excess probabilities} $\widehat{\alpha}_{j}\to 0$. This can be useful for interpretation and model selection in real-life problems since practitioners typically fit the LDA model with a large number of topics and only look for the most "popular" topics, ones with non-negligible probabilities \citep{blei2003latent}.

While existing results for identifiability in the LDA model are quite scarce, the case for mixture of product models is considerably better understood. Since $p_{G_0}^{\mathscr{M}} = \sum_{k=1}^{K_0} \tilde{\alpha}^{0}_k (\theta_k^{0})^{\otimes N}$ may be seen as a sum of $K_0$ symmetric rank-1 tensors, the identifiability of mixture of product models boils down to the uniqueness of the so-called Waring decompositions in algebraic literature \cite{lovitz2023generalization, chiantini2017generic}. The exact-fitted and over-fitted identifiability for the mixture models is studied in \cite{vandermeulen2019operator} using operator-theoretic techniques. Finer characterization of identifiability is often studied via the \emph{Kruskal rank} of true topics, which is the largest number $R_K$ so that every subset of $R_K$ vectors in $\{\theta_1^0, \dots, \theta_{K_0}^0\}$ is linearly independent \cite{KRUSKAL197795}. Intuitively, if more constraints on $G_0$ are assumed, the identifiability property more likely holds. For example, the true topics $\theta_1^{0}, \dots, \theta_{K_0}^{0}$ are said to satisfy the \textit{anchor-word condition}~\cite{arora2012learning, vandermeulen2019operator} if for every $k\in [K_0]$, there exists an "anchor word" $v = v(k)$ such that $\theta^0_{kv} > 0$ while $\theta^0_{k'v} = 0$ for all $k'\neq k$.

The identifiability results in the mixture of product models literature can be borrowed to obtain corresponding results for LDA model thanks to the equivalence demonstrated in Corollary~\ref{cor:equivalence-identifiability}. This allows us to provide upper bounds for the minimum document lengths required for identifiability in the LDA model for both exact-fitted and over-fitted cases.
\begin{proposition}\label{prop:identifiability}
Given $G_0 = \sum_{j=1}^{K_0} \tilde{\alpha}_j^0 \delta_{\theta_j^0} \in \Ecal_{K_0}$, depends on the condition on $\theta_1^0, \dots, \theta_{K_0}^0$, the minimum identifiable lengths $\mathfrak{N}_{e}(G_0)$ and $\mathfrak{N}_{o}(G_0, K)$ are upper bounded by the corresponding numbers in Table~\ref{table:identifiability}.
\end{proposition}

\begin{table}
    \centering
    \begin{tabular}{|c|c|c|}
    \hline
    Assumptions on $(\theta_k^0)_{k=1}^{K_0}$ & Upper bound of $\mathfrak{N}_{e}(G_0)$ & Upper bound of $\mathfrak{N}_{o}(G_0, K)$ \\
    \hline
    Distinct & $\boldsymbol{2 K_0 - 1}$ & $\boldsymbol{2 K_0}$  \\
    \hline
    Rank $=R$, Kruskal's rank $=R_K$ & $\boldsymbol{\lceil(2 K_0 - 1)/(R_K - 1) \rceil}$ & $\boldsymbol{K + K_0 + 3 - 2R}$ \\
    \hline
    Linearly independent & 3 & \textbf{4}  \\
    \hline
    anchor-word & 2 & \textbf{2} \\
    \hline 
    \end{tabular}
    \caption{\centering Minimum document length for identifiability in MM (therefore, LDA) models. Bold numbers are new results obtained in this paper.}\label{table:identifiability}
\end{table}

\begin{remark}\label{remark:identifiability}
\begin{enumerate}
    \item[(i)] In the exact-fitted setting, the condition on Kruskal's rank~\cite{stegeman2007kruskal, lovitz2023generalization} is most general and can be used to obtain an upper bound of $\mathfrak{N}_e(G_0)$ when true topics are distinct or linearly independent. Indeed, note that the distinction of $\theta_1^0, \dots, \theta_{K_0}^0 \in \Delta^{V-1}$ implies the $R_K \geq 2$, so $\lceil (2K_0 - 1) / (R_K - 1) \rceil\leq 2K_0 - 1$. Besides, when the true topics are linearly independent, we have $R_K = K_0$ so that $\lceil (2K_0 - 1) / (R_K - 1) \rceil\leq 3$. The bound on the Kruskal's rank is also globally tight~\cite{derksen2013kruskal}, meaning that there exists $G_0$ which does not satisfy the Kruskal's rank condition and the corresponding $p_{G_0, N}$ does not satisfy exact-fitted identifiability. In the over-fitted setting, the bound $\mathfrak{N}_{o}(G_0, K)\leq (K + K_0 + 3 - 2R)$ from~\cite{lovitz2023generalization} is also general when over-fitting by one topic (i.e., $K = K_0 + 1$), as $K + K_0 + 3 - 2R \leq 2K_0$ whenever $R_K \geq 2$ (distinct true topics) and equals $4$ when $R = R_K = K_0$ (linearly independent true topics). However, note that the bound $2K_0$ (in the distinct setting) and $4$ (in the linearly independent setting) do not depend on $K$, so they are generally better when the over-fitting level $K$ is high.
    \item[(ii)] It is worth emphasizing that because the equivalence of LDA models and mixture models is only known for $N = 3$ prior to this work, the knowledge of the identifiability of the LDA models is restricted to the exact-fitted setting ($K_0$ is known) under either linearly independence or anchor-word conditions. Table \ref{table:identifiability} highlights how this work considerably expands the scope of model settings where identifiability can still be established. 
    % \item[(iii)] Cook up counter-examples for some settings.
\end{enumerate}
\end{remark}

\subsection{Finer identifiability via inverse bounds}\label{sec:finer_identifiability}
Having established the identifiability for LDA models, we turn to the more challenging question of \emph{parameter estimation}. The standard Bayesian asymptotic theory, cf. \citep{Ghosh-Ramamoorthi-02,ghosal2017fundamentals}, can be applied to derive posterior contraction rates of the \emph{densities} induced by the LDA models. However, to obtain posterior contraction rates for model parameters,
our approach is to derive lower bounds on a suitable distance between LDA models by another distance between the corresponding model parameters. The former will be the total variation distance on probability densities, while the latter is a suitable Wasserstein distance metric~\cite{nguyen2013convergence}. For two discrete probability measures $G = \sum_{i=1}^{K} \tilde{\alpha}_i \delta_{\theta_i}$ and $G = \sum_{i=1}^{K'} \tilde{\alpha}'_i \delta_{\theta'_i}$, the Wasserstein$-r$ distance (for $r\geq 1$) is defined as:
$$W_r^r(G, G') = \inf_{q \in \Pi(\tilde{\alpha},\tilde{\alpha}')} \sum_{i=1}^{K}\sum_{j=1}^{K'} q_{ij} \|\theta_i - \theta_j'\|^{r},$$
where $\Pi(\tilde{\alpha}, \tilde{\alpha}') := \{(q_{ij})_{i,j=1}^{K, K'} | q_{ij} \geq 0, \sum_{j=1}^{K'} q_{ij} =\tilde{\alpha}_i, \sum_{i=1}^{K} q_{ij} =\tilde{\alpha}'_j\}$ represents the space of all couplings between $\tilde{\alpha}$ and $\tilde{\alpha}'$. For true parameter $G_0 = \sum_{i=1}^{K_0} \tilde{\alpha}^0_i \delta_{\theta^0_i}  \in \Ecal_{K_0}$, the aforementioned bounds are called  \emph{inverse bounds} \cite{nguyen2013convergence,wei2022convergence}, which have the form:
\begin{equation}\label{eq:inverse-bound}
    d_{TV}(p_{G_0, N}^{\mathscr{L}}, p_{G, N}^{\mathscr{L}}) \gtrsim W_r^r(G, G_0), \quad \forall G = \sum_{j=1}^{K} \tilde{\alpha}_j \delta_{\theta_j} \in \Ocal_{K},
\end{equation}
for some $r\geq 1$, and the multiplicative constant in this inequality depends only on $G_0$ and $K$. The name \textit{inverse bounds} come from the fact that bounds in the opposite direction of~\eqref{eq:inverse-bound} (i.e., upper bound distance between densities by distance between parameters; Appendix~\ref{subsec:geometry-multinom} and~\ref{subsec:geometry-mixture-multinom}) are relatively easier to obtain. It can be seen that these inverse bounds are stronger than identifiability because the LHS of the inequality~\eqref{eq:inverse-bound} is zero implies the RHS is zero. Thus, inverse bounds represent a finer characterization of identifiability. 

Similar to the handling of identifiability in this paper, our strategy to prove inverse bounds is to first show~\eqref{eq:inverse-bound} for the mixture of product models $p^{\mathscr{M}}$, and then deduce those bounds for $p^{\mathscr{L}}$ thanks to Proposition~\ref{prop:metric-equivalent}. Because $d_{TV}(p^{\mathscr{L}}_{G, N'}, p^{\mathscr{L}}_{G_0, N'}) \geq d_{TV}(p^{\mathscr{L}}_{G, N}, p^{\mathscr{L}}_{G_0, N})$ for every $N'\geq N$, we want to investigate the minimum document length $N$ for which the inverse bounds hold. 
\begin{definition}
    For a discrete measure $G_0$, let $\mathfrak{N}_1(G_0)$ be the smallest $N$ such that the inverse bound~\eqref{eq:inverse-bound} holds for all $G\in \Ecal_{K_0}$ with $r=1$, and $\mathfrak{N}_2(G_0, K)$ to be the smallest $N$ such that the inverse bound~\eqref{eq:inverse-bound} holds for all $G\in \Ocal_{K}$ with $r=2$ and $K > K_0$.
\end{definition} 
We drop $K$ and use $\mathfrak{N}_2(G_0)$ instead if \eqref{eq:inverse-bound} holds for every $K\geq K_0$. It is clear from the discussion in the previous paragraph that $\mathfrak{N}_1(G_0)\geq \mathfrak{N}_e(G_0)$ and $\mathfrak{N}_2(G_0, K)\geq \mathfrak{N}_o(G_0, K)$. To see why convergence in Wasserstein distance implies the convergence of topics, let $I_{k} = \{\theta: \|\theta - \theta_k^0\| \leq \|\theta - \theta_{k'}^0\| \forall k'\neq k\}$ be the Voronoi cell of $\theta_k^0$ in $\Delta^{V-1}$, then
\begin{equation}\label{eq:Wasserstein-equivalent-Voronoi}
    W_r^{r}(G, G_0) \asymp \sum_{k=1}^{K_0} \left|\sum_{\theta_j\in I_k} \tilde{\alpha}_{j} - \tilde{\alpha}^0_{k} \right| +  \sum_{j: \theta_j \in I_k} \tilde{\alpha}_j \|\theta_j - \theta_k^0 \|^{r},  
\end{equation}
as $G\to G_0$ (see, e.g., \cite{ho2019singularity, do2024dendrogram}). Hence, if a sequence of latent measure $G_n$'s satisfies $W_r(G_n, G_0) \to 0$, then it follows that there is a subsequence of $G_n$'s for which each associated topic $\theta_k$ with non-vanishing $\tilde{\alpha}_k$ must converge to some topic associated with $G_0$. In particular, when the sequence $G_n$ has exactly $K_0$ topics (exact-fitted scenario), then each pairs $(\alpha^{n}_k, \theta^{n}_k)$ of $G_n$ will converge to the true pairs $(\alpha^{0}_k, \theta^{0}_k)$ of $G_0$ (up to a permutation). The following theorem provide upper bounds on the fundamental quantities just defined.

\begin{theorem}\label{thm:inverse-bounds}
Let $G_0 \in \Ecal_{K_0}$. We have the following upper bounds:
\begin{enumerate}
    \item[(a)] $\mathfrak{N}_1(G_0)\leq 2K_0 - 1$ and $\mathfrak{N}_2(G_0, K)\leq K_0 + K - 1$.
    \item[(b)] If $\theta_1^0, \dots, \theta_{K_0}^0$ are linearly independent, then $\mathfrak{N}_1(G_0)\leq 3$ and $\mathfrak{N}_2(G_0)\leq 4$.
    \item[(c)] If $\theta_1^0, \dots, \theta_{K_0}^{0}$ satisfy the anchor-word condition, then $\mathfrak{N}_1(G_0)\leq 2$.
\end{enumerate}
\end{theorem}

\begin{remark}
\begin{enumerate}
\item[(i)] Condition such as (b) for $\mathfrak{N}_1(G_0)$ provide sufficient conditions for the inverse bound~\eqref{eq:inverse-bound}, which are referred to as \emph{robust forms} of the Kruskal theorem in the tensor algebraic field~\cite{bhaskara2014uniqueness}. In that literature only the exact-fitted setting has been investigated.
\item[(ii)] In part (a) of the theorem, the upper bound on $\mathfrak{N}_1(G_0)$ in several settings for the finite mixture of Binomial distributions can be found in \cite{heinrich2018strong, do2025strong}. %where it is often proved using Vandermonde matrix's determinant. 
Here, the result is extended to general mixtures of products of multinomial distributions, and also to the LDA models via Proposition~\ref{prop:metric-equivalent}. 
Our bound on $\mathfrak{N}_2(G_0)$ in the general setting (a) is an improvement of the inverse bound implied by Proposition 1 in~\cite{manole2021estimating}, where the authors require $N\geq 3K-1$. This is because, in the proof, we carefully classify the topics $\theta_k^0$'s that are over-fitted by many topics and those that are exact-fitted by only one topic of $G \in \Ocal_K$. 
% A similar idea has been introduced in~\cite{manole2022refined} in the context of improving parameter estimation rate. This technical improvement allows us to weaken the sufficient condition for the inverse bound, which is known as the second-order strong identifiability~\cite{Chen1992, Ho-Nguyen-EJS-16}. Our proof also patches a minor mistake in the proof in~\cite{manole2021estimating}, where they project Multinomial to Binomial distribution so that it requires $\theta_1^0, \dots, \theta_{K}^0$ to be \textit{element-wise distinct} instead of being distinct only.
% \textcolor{red}{LN: Move from "A similar idea...." to the preceding sentence to the proof. This detail will not be comprehended/ appreciated by most readers.}
\item[(iii)] When the true topics are further assumed linearly independent, the bounds on $\mathfrak{N}_1(G_0)$ and $\mathfrak{N}_2(G_0)$ become much tighter and we can utilize tools from multi-linear algebra literature to prove them. It is worth highlighting that the bound for $\mathfrak{N}_2(G_0)$ is new for both the mixture of product models and LDA models.
%A bound on $\mathfrak{N}_1(G_0)$ of the mixture models is shown in  \cite{bhaskara2014uniqueness} by the name \emph{robust Kruskal theorem} where they assume a weaker condition on the Kruskal rank of true topics. 
\item[(iv)] The anchor-word condition is popular in the topic modeling literature~\cite{arora2012learning, arora2013practical}.    
When the true topics satisfy this condition, only 2 words per document are enough to satisfy the inverse bound. This was proved for mixture of product models in~\cite{vandermeulen2019operator} but seems novel for LDA models. 
% \textcolor{red}{LN: Any references on related results? What about the overfitted setting?}
\end{enumerate}
\end{remark}

% \subsection{Uniform inverse bounds} In practice, as we collect more documents, the number of topics may also increase. The inverse bounds for fixed true parameters in Theorem~\ref{thm:inverse-bounds} are not appropriate for understanding the topic estimation in this setting anymore, but perhaps the uniform inverse bounds developed (in mixture modeling context) by Heinrich and Kahn~\cite{heinrich2018strong} are better suited. Here, we show that such inverse bounds also hold for the LDA models and will use them to prove uniform contraction rates in the next section. Similar to above, we utilize Proposition~\ref{prop:metric-equivalent} combined with inverse bounds for mixture models to prove this. These required results are available in the mixture modeling literature~\cite{wei2023minimum}.
% \begin{proposition}
% \begin{enumerate}
%     \item[(a)] (Globally uniform inverse bound) Suppose that $N\geq 2K-1$, we have that
%     \begin{equation}
%         d_{TV}(p_{G, N}, p_{G', N}) \gtrsim W_{2K-1}^{2K-1}(G, G'),\, \forall G, G' \in \Ocal_{K}.
%     \end{equation}
%     \item[(b)] (Locally uniform inverse bound) Given a fixed mixing measure $G_0 \in \Ecal_{K_0}$ with $K_0\leq K$, and denote $d_1 = K - K_0 + 1$. Suppose that $N\geq 2K - 1$, there exists a constant $\epsilon(G_0)>0$ such that
%     \begin{equation}
%         d_{TV}(p_{G, N}, p_{G', N}) \gtrsim W_{2d_1 - 1}^{2d_1 - 1}(G, G'),\, \forall G, G' \in \Ocal_{K}\cap B_{W_1}(G_0, \epsilon(G_0)).
%     \end{equation}
% \end{enumerate}
% \end{proposition}
\section{Posterior contraction rates in LDA}\label{sec:contraction-rate}
Now, we are ready to tackle posterior contraction behaviors for the density and parameters in the LDA model. In particular, posterior contraction rates for the model's induced density function can be obtained via a standard approach (e.g., ~\cite{ghosal2017fundamentals}). Once this is done, posterior contraction rates for model parameters can then be derived by drawing on the inverse bounds obtained in the previous section. Although we focus on a fully Bayesian estimation procedure in this section, it is worth mentioning that parameter estimation rates can be easily obtained with other estimation techniques such as MLE and moment based methods as well, again by appealing to the inverse bounds established in this paper. Finally, we investigate how the LDA as a canonical hierarchical model enables efficient estimation of the latent topic label allocations by encouraging (implicitly) the borrowing of strength across observed document samples.

% dive deeper into the implications of the results obtained in the preceding sections. From a fully Bayesian perspective, we study the properties of the posterior in regards to density estimation and parameter estimation and obtain contraction rates under relatively mild conditions on the appropriate priors. Using the notations from Section \ref{sec:identifiability}, we write $p_{G, N}$ to denote the density of the LDA model, where $G=\sum_{k=1}^K\tilde{\alpha}_k\delta_{\theta_k}$ is the mixing measure encapsulating the parameters of the model - topics $\theta_1,\dots,\theta_K\in\Delta^{V-1}$ and normalized Dirichlet parameter $\tilde{\alpha}\in\Delta^{K-1}$. We also recall that $\mathcal{O}_K=\cup_{k=1}^K \mathcal{E}_K$, where $\mathcal{E}_K$ is the space of discrete probability measures on $\Delta^{V-1}$ with exactly $K$ atoms. We note that we treat $\overline{\alpha}$, the sum of the coordinates of the Dirichlet parameter, as fixed in this discussion. A prior in this setting is placed on $\mathcal{E}_K$, $G\sim \Pi$, which is the same as placing priors on $\tilde{\alpha}$ and $\theta_1,\dots,\theta_K$. We study the density and parameter estimation and then study how borrowing of strength across documents also helps in estimating the latent allocations in the model.

\subsection{Posterior contraction of densities and parameters}
\label{sec:density_parameter_contraction}
Let $G_0=\sum_{k=1}^{K_0} \tilde{\alpha}_k^0\delta_{\theta_k^0}\in \mathcal{E}_{K_0}$ be the unknown true latent measure, consisting of (true) topics $\theta_1^0,\dots,\theta_{K_0}^0$ and normalized Dirichlet parameter $\tilde{\alpha}^0$. Suppose we observe $m$ i.i.d. documents $X_{[N]}^1,\dots, X_{[N]}^m$ from the LDA model $p^{\mathscr{L}}_{G_0, N}$, and the goal is to make inference on $G_0$. In a Bayesian estimation framework, the data is modeled by positing that $X^{1}_{[N]}, \dots, X^{m}_{[N]} | G \overset{iid}{\sim} p^{\mathscr{L}}_{G, N}$, where the parameter $G$ is random and endowed with a prior distribution $\Pi$ on the set $\Ocal_{K}$. A version of the posterior distribution is obtained by Bayes' formula:
\begin{equation}
    \Pi(G \in B | X_{[N]}^{[m]}) = \dfrac{\int_{B} \prod_{i=1}^{m} p^{\mathscr{L}}_{G, N}(X_{[N]}^{i}) d\Pi(G)}{\int \prod_{i=1}^{m} p^{\mathscr{L}}_{G, N} (X_{[N]}^{i}) d\Pi(G)},
\end{equation}
for all Borel set $B$ of $\Ocal_K$. 
In practice, $K_0$ may be known, but a valid upper bound $K \geq K_0$ is given. Thus, the prior distribution $\Pi$ may be composed of a prior $\Pi_{\alpha}$ on $\tilde{\alpha} = (\tilde{\alpha}_{1}, \dots, \tilde{\alpha}_{K})\in \Delta^{K-1}$ and prior $\Pi_{\theta}$ independently on $\theta_1, \dots, \theta_{K}$. This is a common over-fitted setting in the Bayesian modeling literature \citep{rousseau2011asymptotic,guha2021posterior}. To ensure posterior consistency, a standard condition is required such that the prior place sufficient mass everywhere in its domain. For concrete posterior contraction rates, we need the following conditions. 

% place a prior $G\sim \Pi$ on $\mathcal{E}_K$ for $K\geq K_0$ and study the properties of the posterior distribution of the density $p_{G, N}$ or $G$ itself, given the data. The only crucial property of the prior that is needed is the following:
\begin{definition}
    A probability measure $P$ supported on $\Omega\subset\Delta^{D-1}$ is called regular if for some constant $c=c(D)$,
    $$\inf_{\eta^*\in\Omega} P\left(\left\{\eta\in\Omega\mid \norm{\eta-\eta^*}< \epsilon\right\}\right)\geq c\epsilon^{D-1}.$$
\end{definition}
If $P$ has a density with respect to Lebesgue measure on $\Delta^{K-1}$ and the density $p$ is lower bounded, i.e., $p\geq c'$, then $P$ is trivially regular. Dirichlet distribution on $\Delta^{K-1}$ with parameter $\gamma\in\Rbb^K$ is regular if $\gamma_k\leq 1$ for all $k$ (see Lemma 5.2 in \cite{chakraborty2024learning}). The main required assumptions for the posterior consistency are:
\begin{enumerate}
    \item[\textbf{(P1)}]\label{assume:p1} The true topics  
    % and the support of prior $\Pi_{\theta}$'s 
    are bounded away from the boundary of $\Delta^{V-1}$, i.e., there exists $c_0 > 0$ such that $\theta_{kv}^0\geq c_0$
    % and $\theta_{v} \geq c_0$ 
    for all $v \in [V], k \in [K_0]$.
    % and $\theta = (\theta_{v})$ belongs to the support of $\Pi_{\theta}$. 
    \item[\textbf{(P2)}]\label{assume:p2} The priors $\Pi_{\alpha}$ and $\Pi_{\theta}$ are regular, and their support contain the true parameters $\tilde{\alpha}^0$ and topics $\theta_1^{0}, \dots, \theta_{K_0}^0$, respectively.
\end{enumerate}

The following theorem describes the posterior contraction properties for both the model induced density function and its corresponding parameters:

\begin{theorem}\label{theorem:contraction_posterior}
    Suppose that $G_0 = \sum_{k=1}^{K_0} \tilde{\alpha}_k^0 \delta_{\theta_k^0} \in \Ecal_{K_0}$ and the prior $\Pi$ on $\Ocal_K$ satisfies assumptions \hyperref[assume:p1]{(P1)} and \hyperref[assume:p2]{(P2)} , where $K \geq K_0$. 
\begin{enumerate}
    \item[(a)] (Density estimation) There exists a sufficiently large constant $C$ such that
    $$\Pi\left(G\in \Ocal_K \, :\, d_H(p^{\mathscr{L}}_{G,N}, p^{\mathscr{L}}_{G_0,N})\geq C\left(\frac{\log (mN)}{m}\right)^{1/2} \mid X_{[N]}^1,\dots,X_{[N]}^m\right)\to 0$$
    in $\otimes_{i=1}^{\infty} P^{\mathscr{L}}_{G_0,N}$-probability as $m\to\infty$.
    % \item[(b)] (Density estimation with unrestricted prior in exact-fitted case) Suppose the true mixing measure $G_0$ satisfies that each of the component topics are bounded away from the boundary of $\Delta^{V-1}$, i.e., $\min_v \theta_{kv}^0\geq c_0$ for all $k\in[K_0]$ for some constant $c_0>0$. Suppose that the prior $\Pi$ satisfies the following: under $\Pi$, $\theta_k'$s and $\tilde{\alpha}$ are all independent and the induced distributions on each $\theta_k$ and $\tilde{\alpha}$ are respectively regular on $\Delta^{V-1}$ and $\Delta^{K-1}$. In the case $K=K_0$ with $N\geq \mathfrak{N}_1(G_0)$, there exists sufficiently large constant $C$, depending on $V, G_0$ such that
    % $$\Pi\left(G\in\mathcal{E}_K \, :\, d_H(p_{G,N}, p_{G_0,N})\geq C\left(\frac{\log (mN)}{m}\right)^{1/2} \mid X_{[N]}^1,\dots,X_{[N]}^m\right)\to 0$$
    % in $P_{G_0,N}^\infty$-probability as $m\to\infty$.
    \item[(b)] (Parameter estimation in exact-fitted setting) Suppose that $K = K_0$ and $N\geq \mathfrak{N}_1(G_0)$, there exists sufficiently large constant $C$ depending on $G_0$ such that
    $$\Pi\left(G\in\mathcal{O}_{K_0} \, :\, W_1(G, G_0)\geq C\left(\frac{\log (mN)}{m}\right)^{1/2} \mid X_{[N]}^1,\dots,X_{[N]}^m\right)\to 0$$
    in $\otimes_{i=1}^{\infty} P^{\mathscr{L}}_{G_0,N}$-probability as $m\to\infty$.
    \item[(c)] (Parameter estimation in over-fitted setting) Suppose that $K > K_0$ and $N\geq \mathfrak{N}_2(G_0, K)$, there exists sufficiently large constant $C$ depending on $G_0$ and $K$ such that
    $$\Pi\left(G\in\mathcal{O}_K \, :\, W_2(G, G_0)\geq C\left(\frac{\log (mN)}{m}\right)^{1/4} \mid X_{[N]}^1,\dots,X_{[N]}^m\right)\to 0$$
    in $\otimes_{i=1}^{\infty} P^{\mathscr{L}}_{G_0,N}$-probability as $m\to\infty$.
\end{enumerate}
\end{theorem}

The proof of this theorem appears in Appendix~\ref{appD-proof}. We add some remarks.

\begin{remark}
\begin{enumerate}
    \item[(i)] The boundedness of the true parameters from the boundary of the parameter space (Assumption P1) is intrinsic to the likelihood-based inference, such as MLE and Bayesian methods~\cite{10.1214/18-EJS1516, nguyen2015}. 
    % In the theorem, the boundedness of the support of the prior $\Pi_{\theta}$ is also required. We note that it is possible to remove this assumption in the exact-fitted setting given that $N\geq\mathfrak{N}_1(G_0)$ and an extra dependence of the constant $C$ on $G_0$ - under these conditions, the inverse bound applies and hence if the true topics are bounded away from the boundary, then for small $\epsilon$, the $\epsilon-$KL ball around the true $G_0$ consists of $G$ for which the topics are bounded away from the boundary. The precise statement and proof of this remark also appear in Appendix~\ref{appD-proof}.
    \item[(ii)] The result in (a) shows a near parametric rate of $m^{-1/2}$ for density estimation. We note that the presence of $\log(N)$ in the numerator indicates that increasing $N$ deteriorates the posterior contraction rate -- this is due to the fact that it worsens the prior mass condition on the KL-ball around the truth (cf. Proposition \ref{prop:kl_ball_prior} in the Appendix), and increases the Le Cam dimension of the model (cf. Proposition \ref{prop:entropy_model}). Because $p^{\mathscr{L}}_{G, N}$ has the conditionally i.i.d. property across $N$ observation, the rate only worsens by $\log(N)$ when $N$ increases. 
    % The dependence on $V$ shows that pre-possessing by removing too high and low frequent words and ambiguous words helps in learning.  
    \item[(iii)] Results in (b) and (c) indicate near-parametric $m^{-1/2}$ rate in the exact-fitted case and $m^{-1/4}$ in the over-fitted case. By Theorem \ref{thm:inverse-bounds}, for the exact-fitted case, this rate applies whenever $N\geq 3$ (if the true topics are linearly independent) or $N\geq 2K_0-1$ in the general case. Similarly, for the over-fitted case, the result holds whenever $N\geq 4$ (in the linearly independent case) and $N\geq K_0+K-1$ in the general case.
    \item[(iv)] The constants $C$ in all the parts of the above theorem do not depend on $N$ (but only $\mathfrak{N}_e(G_0)$ and $\mathfrak{N}_o(G_0)$). As such, all of the results in the preceding theorem applies even in the case when $N$ is allowed to increase with $m$. Thus the rates tend to 0 and is meaningful as long as $\log(N)/m\to 0$ as $N, m\to\infty$.
    \item[(v)] In practice, a Dirichlet prior on $\Delta^{V-1}$ is typically placed independently for each $\theta_k$ and an independent Dirichlet prior on $\Delta^{K-1}$ is placed for the $\tilde{\alpha}$ -- such a prior specification satisfies all the conditions in parts (b), (c), (d) as long as the Dirichlet prior hyperparameters are less than or equal to 1. %Thus, if the true topics are not on the boundary of $\Delta^{V-1}$, the relevant posterior distributions enjoy the contraction rates as indicated in the preceding theorem.
\end{enumerate}
\end{remark}

\subsection{Borrowing strength in estimating Latent Allocation}
A quantity of interest is the latent topic allocation $q$ for every document, which represents how the document is composed as a convex combination of the topics. Recall that for each $i\in[m]$, there is a latent allocation $q_i\in\Delta^{K_0-1}$ following a $\Dir(\alpha^0)$ distribution so that the words in the $i-$th document $X_{[N]}^i$ are conditionally i.i.d. from Categorical distribution with parameter $\sum_{k\in[K_0]} q_{ik}\theta_k^0$. Making inferences on allocation $q_i$ is of particular interest in a variety of domains, e.g., in population
genetics, where $q_i$ represents the admixing proportion out of ancestral origins of an individual's genetic makeup.
%to biologists in the LDA model for population genetics \cite{pritchard2000inference}, where each document represents a person, the topics represent ancestral populations and $q_i$ represents the admixing proportion of the $i-$th individual in terms of the $K_0$ ancestral populations. 
In this section, we discuss how the LDA model enables the inference on $q$ and characterize the convergence properties of this estimation.

\begin{figure}
    \centering
    \includegraphics[width=0.95\linewidth]{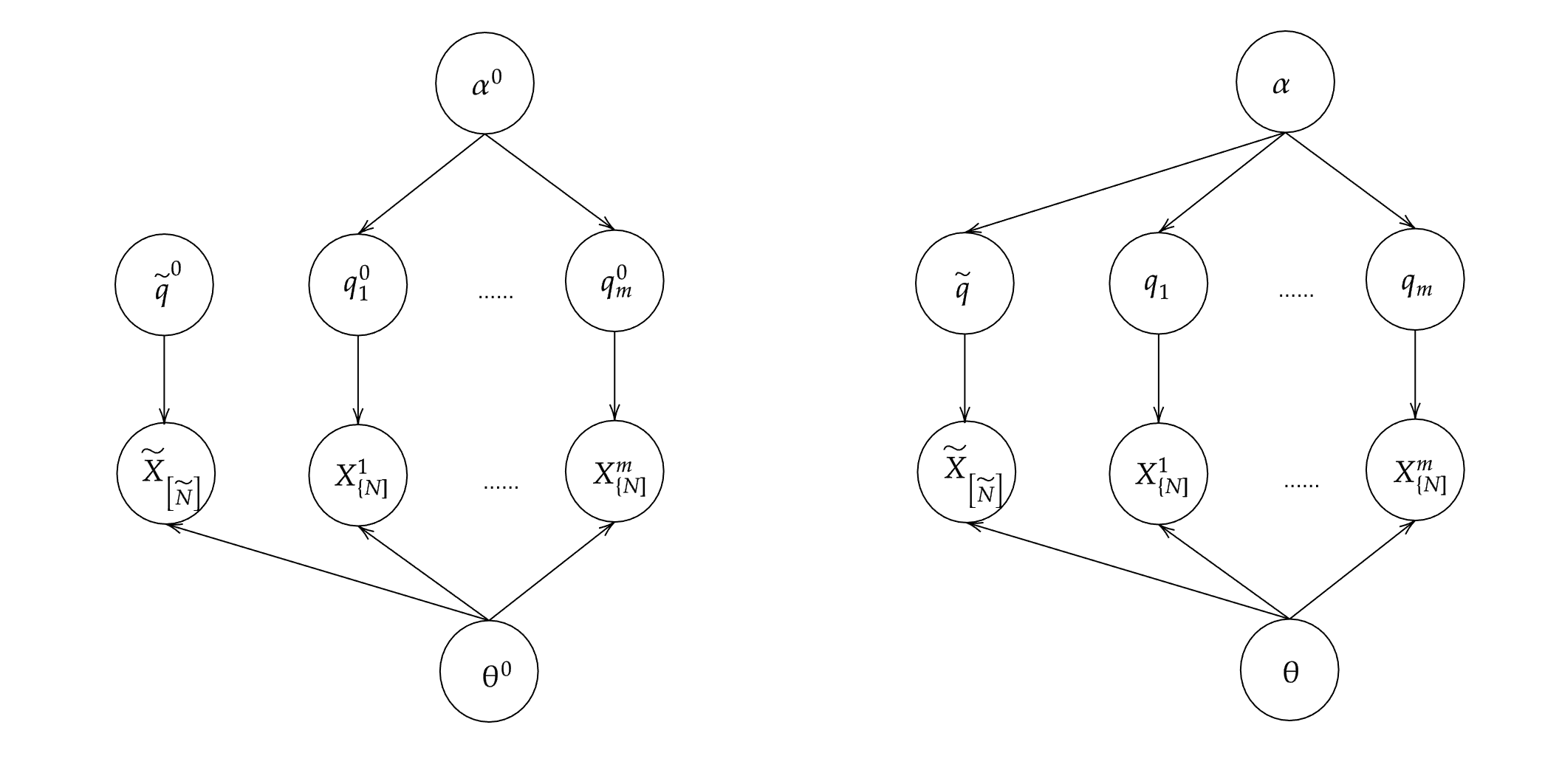}
    \caption{The true generating model (left) and the fitted LDA model (right). Although we assume a ``true'' allocation parameter $\tilde{q}^0$ for the document $\tilde{X}_{[\tilde{N}]}$, the whole corpus of $(m+1)$ documents is jointly modeled under the LDA model. We will show that the posterior distribution of document allocation $\tilde{q}$ under the LDA model contracts around the true parameter $\tilde{q}^0$ thanks to the information of topics $\theta$ and parameter $\alpha$ learned from all documents.}
    \label{fig:placeholder}
\end{figure}

For this section, assume that the true topics $\theta_1^0,\dots,\theta^0_{K_0}$ are linearly independent --- this ensures that each document's parameter can be expressed \textit{uniquely} as a convex combination of the true topics. Secondly, we only consider the exact-fitted setting to ensure that the dimension of $q$ is correctly specified. Recall that in our original model, the $q_i$ for $i\in[m]$ were considered as latent variables and marginalized out in the inference. To illustrate %illuminate 
the convergence behavior of $q$, as in \cite{nguyen2016borrowing}, let us isolate an extra document $\tilde{X}_{[\tilde{N}]}$ of length $\tilde{N}$, which shall be assumed to be generated from the Multinomial distribution $\Multi(\sum \tilde{q}^0_{k} \theta^0_{k})$ for some unknown allocation parameter $\tilde{q}^0 = (\tilde{q}^0_{k})_{k=1}^{K_0}$. The data corpus  $\boldsymbol{X}$ now contains $m+1$ documents, namely, $X^i_{[N]}$ for $i=1,\dots,m$ and $\tilde{X}_{[\tilde{N}]}$. The fitted LDA model stipulates a joint distribution $p(\boldsymbol{X}, \boldsymbol{q}|\alpha, \Theta)$ for the entire data corpus $\boldsymbol{X}$ and the allocations $\boldsymbol{q}=(\tilde{q}, q_1,\dots,q_{m})$. Given a prior $\pi(\alpha, \Theta)$ on the parameters, one obtain the posterior distribution $p(\boldsymbol{q}, \alpha,\Theta|\boldsymbol{X})$. The following result characterizes the contraction of marginal posterior distribution $p(\tilde{q}|\boldsymbol{X})$ induced by $p(\boldsymbol{q}, \alpha,\Theta|\boldsymbol{X})$ around the ground-truth $\tilde{q}^0$ which generated document $\tilde{X}_{[\tilde{N}]}$ as $m,\tilde{N}\to\infty$. %Technically, for the individual document associated with $\tilde{q}$, the distribution $p(\tilde{q}|\boldsymbol{X})$ is not a posterior distribution since $\tilde{q}$ is not a parameter and there is no notion of prior on $\tilde{q}$, but one may view the Dirichlet distribution as the prior. 
Technically, for the individual document associated with $\tilde{q}$, one may view of $p(\tilde{q}|\boldsymbol{X})$ as a posterior distribution induced by the Dirichlet prior for $\tilde{q}$.

\begin{theorem}[Latent topic allocation estimation] \label{theorem:contraction-allocation}
Assume the true mixing measure $G_0 = \sum_{k=1}^{K_0} \tilde{\alpha}_k \delta_{\theta_k}$ satisfies condition \hyperref[assume:p1]{(P1)}, the prior satisfies \hyperref[assume:p2]{(P2)}, and $K=K_0$. Further assume that $\theta_1^0, \dots, \theta_{K_0}^0$ are linearly independent, and $N\geq 3$. As $m, \tilde{N}\to \infty$ such that $\tilde{N} \lesssim mN$, subjecting to a permutation of indices of $\tilde{q}$, we have
    \begin{equation}
        \Ebb \Pi \left( \|\tilde{q} - \tilde{q}^0\| \geq C_{m,\tilde{N}}  \left(\left(\dfrac{\log(mN)}{m}\right)^{1/2} + \left(\dfrac{\log(\tilde{N})}{\tilde{N}}\right)^{1/2}\right) \bigg| \tilde{X}_{[\tilde{N}]}, X_{[N]}^{[m]} \right)\to 0,
    \end{equation}
    for any slowly increasing sequence $C_{m, \tilde{N}}$, where the expectation is with respect to the true generating model $$(\tilde{X}_{[\tilde{N}]}, X_{[N]}^{[m]})\sim\left(\otimes_{j=1}^{\tilde{N}} \Multi(\tilde{X}_{j} |\sum_{k=1}^{K_0} \tilde{q}_{k}^0 \theta_k^0) \right) \otimes \left(\otimes_{i=1}^{m} p_{G_0, N}^{\mathscr{L}}(X^{i}_{[N]})\right).$$ 
    % and the label-switching distance $d_{\ell s}$ between $\tilde{q} = (\tilde{q}_1, \dots, \tilde{q}_{K_0})$ and $\tilde{q}^0 = (\tilde{q}^0_1, \dots, \tilde{q}^0_{K_0})$ is defined as
    % $$d_{\ell s}(\tilde{q}, \tilde{q}^0) := \inf_{\text{permutation } \tau} \left(\sum_{k=1}^{K_0} (\tilde{q}_{\tau(k)} - \tilde{q}_{k}^0)^2 \right)^{1/2}.$$
\end{theorem}

% \begin{proof}[Sketch proof] 
% \begin{enumerate}
%     \item Because of the notice before this theorem + technique in paper Yun \cite{}, can bound the prior mass around the truth + entropy number of $p_{\overline{\alpha}G, N}$ to  prove that the posterior contraction rate for density estimation 
%     $$\Pi\left(d_{TV}(p^{\mathscr{L}}_{\overline{\alpha}G, N}, p^{\mathscr{L}}_{\overline{\alpha}G_0, N})\geq M_m \left(\dfrac{\log(mN)}{m}\right)^{1/2}\bigg| x_{[N]}^{[m]} \right) \to 0,$$
%     as $m, N\to \infty$. Now, we are using the fact that $\theta_1^0, \dots, \theta_{K_0}^0 \in \Delta^{V-1}$ are distinct to have that their Kruskal rank $R_K\geq 2$. Hence, the documents with truncated length $n_0 \geq 2K_0$ will guarantee the identifiability and inverse bound for $p_{G_0, n_0}$. Therefore for all $N\geq n_0$, we have $d_{TV}(p^{\mathscr{L}}_{\overline{\alpha}G_0, N}, p^{\mathscr{L}}_{\overline{\alpha}G, N})\geq d_{TV}(p^{\mathscr{L}}_{\overline{\alpha}G_0, n_0}, p^{\mathscr{L}}_{\overline{\alpha}G, n_0}) \gtrsim W_1(G, G_0)$. This conclude the rate for density and parameter estimation.
%     \item For LDA $q_i$ estimation, we note that up to a row permutation $\|\Theta - \Theta^0\|_{F} \leq \left(\dfrac{\log(mN)}{m} \right)$. This combines with the fact that $\sum_{k=1}^{K_0}q_{ik} \theta_k$ estimate $\sum_{k=1}^{K_0}q^0_{ik} \theta^0_k$ at rate $(\log N/N)^{1/2}$, triangle inequalities and the fact that $\theta_1^0, \dots, \theta_{K_0}^0$ 
%     linearly independent yield the rate for $\norm{q_i - q_i^0}$.
% \end{enumerate}
% \end{proof}
\begin{remark}\begin{enumerate}
    \item The results in the preceding section ensure that the posterior distribution on the mixing measure $G$ contracts to the Dirac measure at $G_0$ as $m\to\infty$, which entails that the estimates for the topics $\theta_1,\dots,\theta_K$ become increasingly accurate. Moreover, for the particular document $\tilde{X}_{[\tilde{N}]}$, as the length $\tilde{N}$ also increases, the posterior for the multinomial parameter $\sum_{k} \tilde{q}^0_{ik} \theta^0_k$ is shown to be consistent. Together with the requirement that the true topics are linearly independent, this implies that $\tilde{q}^0$ can also be learned accurately.
    
    \item We note that since we do not treat the document $\tilde{X}$ differently during modeling, this particular document plays dual roles --- firstly as one of the documents in the corpus, it contributes to the overall contraction of the mixing measure $G$, secondly, it leads to the estimation of the document-specific $\tilde{q}^0$. This along with the fact that the topics $\theta_1,\dots,\theta_K$ are estimated using \textit{all} the documents and then the same topics are used in the estimation of $\tilde{q}^0$ along with the document $\tilde{X}$ couples the estimation in the two stages, which presents challenges in the analysis. We overcome such issues by directly establishing posterior contraction for densities $p_{G,N}$ and $\tilde{p}_{\tilde{q},\Theta}:=\Multi(\cdot|\tilde{q}^\top \Theta)$ simultaneously, which allows us to decouple the two estimation stages.

    % \item \item We emphasize the distinction of the framework considered in the above result from a typical posterior contraction result. Note that $\tilde{q}^0$ is not a parameter in the usual sense and is considered random. However, the statement shows that on average, this $\tilde{q}^0$ can be learned from the data using the same LDA model. We choose this framework to highlight the fact that during modeling, we do not treat $\tilde{X}_{[\tilde{N}]}$ and $X^i_{[N]}$ differently but place the same distribution on them, however such a hierarchical model allows one to learn the allocation for a particular document directly. 
\end{enumerate}
    
\end{remark}
\section{Experiments and Illustrations}\label{sec:experiment}
\subsection{Computing moments of Dirichlet distributions}
We begin by validating our results for computing Dirichlet moment tensors. Firstly, we examine the recursive formula in Proposition~\ref{prop:recursive-compute-moment-Dir} for computing moments of linear transformations of Dirichlet random variables in a concrete example with $\alpha = (0.3, 0.3, 0.4)$, $x = (0.6, 0.7, 0.8)$. The moments $\Ebb_{q\sim \Dir_{\alpha}} (\innprod{x, q})^{N}$ are calculated using the recursive formula~\eqref{eq:recursive-compute-moment-Dir} up to $N=5$ and compared with the empirical moments $\sum_{i=1}^{m} (\innprod{x, q_i})^{N} / m$ with $m$ i.i.d. $q_i \sim \Dir_{\alpha}$. For each sample size $m$, 16 empirical moments are generated to obtain their mean and quartiles. $m$ is ranged from 100 to 10,000 (so that $\log_{10}(m)$ ranging from 2 to 5) and plot the mean of empirical moments in square dots in Figure~\ref{fig:moments}(a) with the error bars being their first and third quartiles. The straight lines are moments calculated using the recursive formula in Proposition~\ref{prop:recursive-compute-moment-Dir}, which we refer to as \textit{theoretical moments}, with orders $4$ and $5$. It is noted that the empirical moments tend to the theoretical moments as $m$ getting large, which verifies Proposition~\ref{prop:recursive-compute-moment-Dir} in this case. 

Next, we randomly generate 100 pairs $(x, \alpha)\in [0, 1]^3\times [0, 1]^3$ and calculate the theoretical and empirical moments (from samples with size 1000) up to order $N=10$. Each point in Figure~\ref{fig:moments}(b) corresponds to a pair of moments, which varies closely around the line $x=y$. It validates Proposition~\ref{prop:recursive-compute-moment-Dir} for a dense set of $(x, \alpha)$. By investigating its proof, we see that for each $\alpha$, Proposition~\ref{prop:recursive-compute-moment-Dir} holds for every $x$ and $N$ if and only if Lemma~\ref{lem:whiten-Dirichlet-moment} is correct. Hence, this simulation also empirically verifies the correctness of Lemma~\ref{lem:whiten-Dirichlet-moment}, which plays a central role in our theory. 

\begin{figure}[h]
      \centering
      \subcaptionbox*{\scriptsize (a) Fourth- and fifth-order moments\par}{\includegraphics[width = 0.48\textwidth]{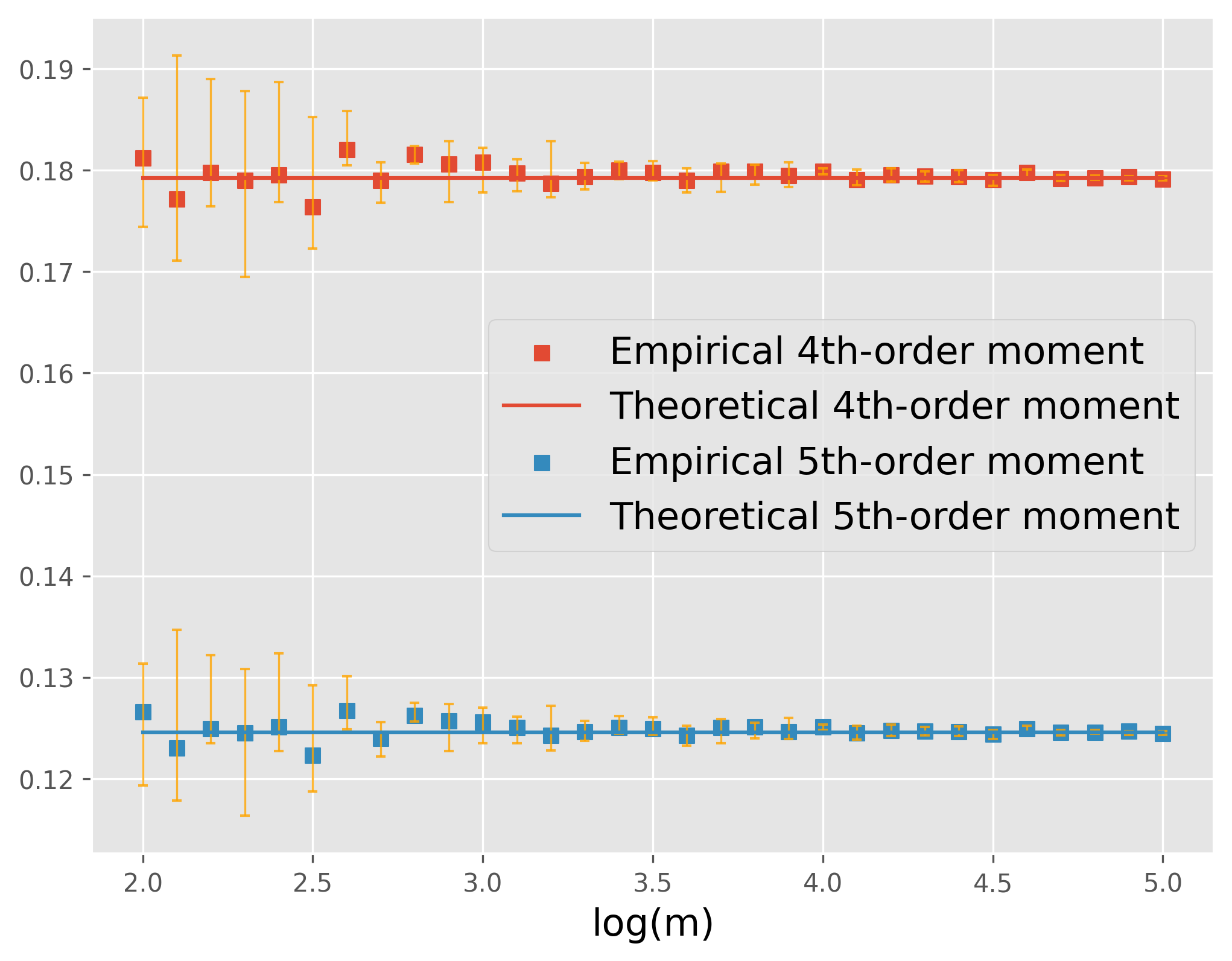}}
      \subcaptionbox*{\scriptsize (b) Comparison moments up to tenth order \par}{\includegraphics[width = 0.48\textwidth]{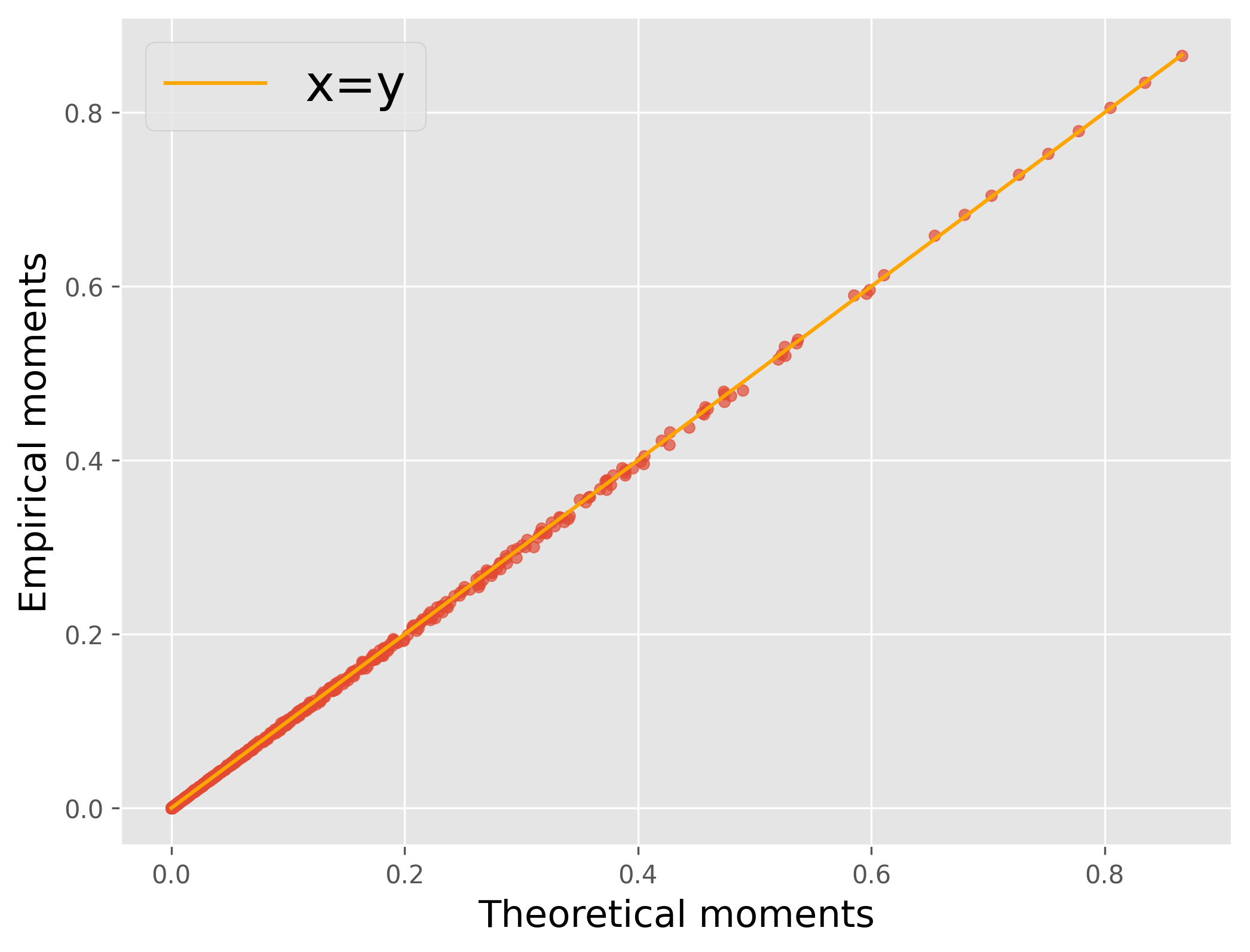}}
      \caption{Moments calculated using recursive formulae in Proposition~\ref{prop:recursive-compute-moment-Dir} and simulated data}
    \label{fig:moments}
\end{figure}

\subsection{Parameter estimation via Bayesian modeling}\label{sec:rate-exact-over}
In the second set of simulations, we verify the parameter estimation rates of the LDA model provided in Theorem~\ref{theorem:contraction_posterior}. The simulation configuration is $N = 20, V=30$, and $K_0=4$. Three true topics $\theta_1^0, \theta_2^0$, and $\theta_3^0$ are generated by an uniform distribution on $\Delta^{V-1}$, and the forth true topics is set to $\theta_4^0 = \frac{1}{3}(\theta_1^0 + \theta_2^0 + \theta_3^0)$. Note that the fourth topic was chosen to be linearly dependent on the first three and used this simulation to show that all four topics can still be learned consistently, as indicated by the theory (Theorem~\ref{theorem:contraction_posterior}).
$m$ documents $x^1, x^2,\dots, x^m \in [V]^{N}$ are generated from the LDA model with true parameters $\overline{\alpha}=0.5$, $\tilde{\alpha}^0 = (1/K_0, 1/K_0, 1/K_0, 1/K_0)$ and topics $\theta_1^0, \theta_2^0, \theta_3^0, \theta_4^0$ as discussed above, where the logarithm of the sample size $\log_{10}(m)$ ranges from $2$ to $4$ (so that $m$ ranges from 100 to 10,000). For each sample size, we conduct an experiment with 16 replications, wherein for each replication, data $x^{[m]}$ are generated according to the true parameters. We then fit a Bayesian model with $K$ components using a uniform prior under exact-fitted ($K = 4 = K_0$) and over-fitted setting ($K = 6 > K_0$). Both settings satisfy the regularity conditions as listed in Theorem~\ref{thm:inverse-bounds} and Theorem~\ref{theorem:contraction_posterior}. Finally, we measure the posterior contraction rate $W_r(G, G_0)$, where $r = 1$ in the exact-fitted setting and $r = 2$ in the over-fitted setting and plot $\log \Ebb [W_r(G, G_0)|x^{[m]}]$ against $\log(m)$ with 50\% errorbar based on 16 replications in Figure~\ref{fig:fix_
alpha_bar}. The results show that the posterior contraction rate is of order $m^{-1/2}$ in the exact-fitted setting and is $m^{-1/4}$ in the over-fitted setting, which matches the theoretical results obtained, up to a logarithmic factor.

\begin{figure}
      \centering
      \subcaptionbox*{\scriptsize (a) Exact-fitted setting\par}{\includegraphics[width = 0.495\textwidth]{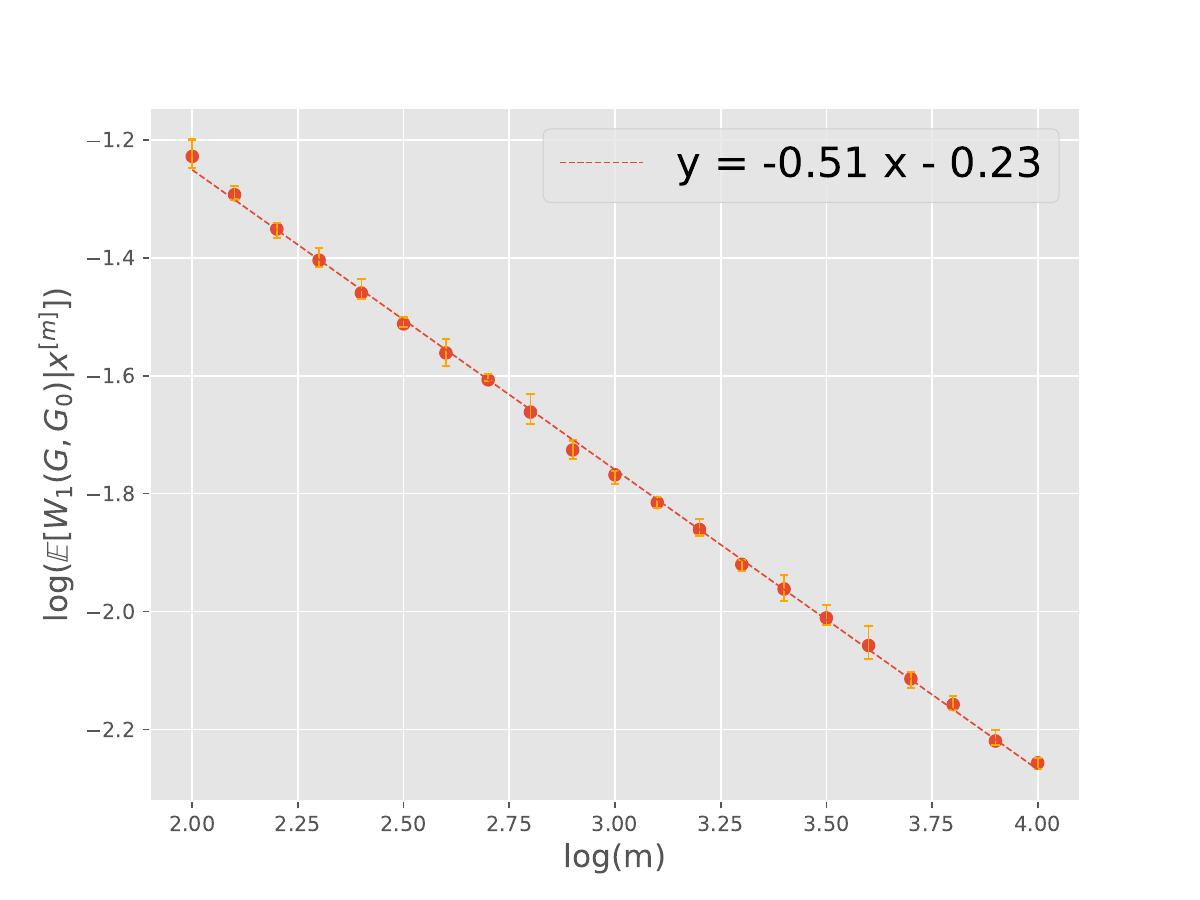}}
      \subcaptionbox*{\scriptsize (b) Over-fitted setting \par}{\includegraphics[width = 0.495\textwidth]{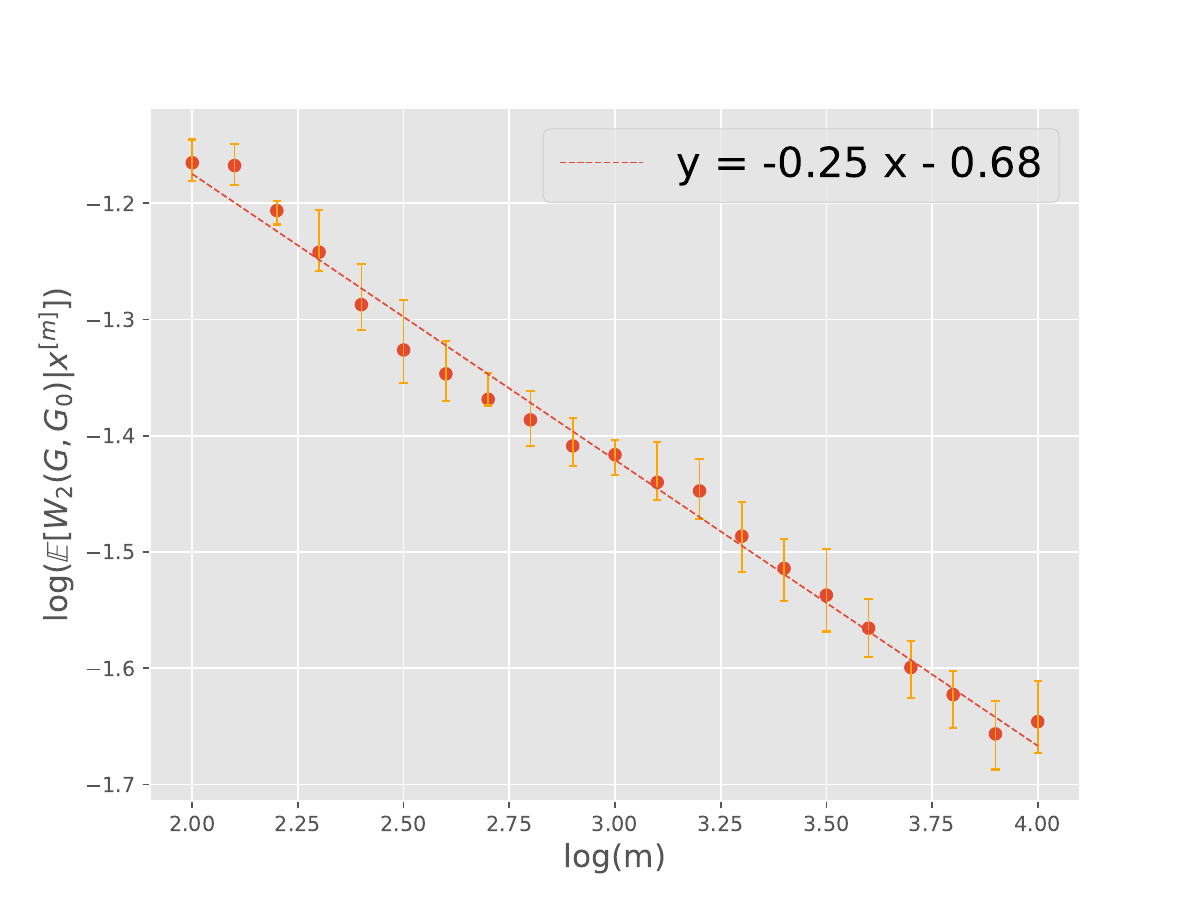}}
      \caption{Posterior contract rates for parameter estimation}
    \label{fig:fix_
alpha_bar}
\end{figure}

% \subsection{Longer documents are more robust to over-fitting} {\color{blue} This is not convincing and needs to be revised.}
% Here, we aim to illustrate that the over-fitted identifiability is easier to satisfy for data with larger $N$. Consider the true data-generating parameters similar to Section~\ref{sec:rate-exact-over} with $K_0=3$ and true topics $\theta_1^0, \theta_2^0, \theta_3^0$. We over-fit the data with $K = 8$ topics in two specific settings: (i) $N=3$ and (ii) $N=30$. As shown earlier, over-fitted identifiability is only guaranteed when $N \geq K + K_0 - 1$ when all topics are linearly independent. Hence, we expect that the parameter estimation in the first setting will be bad since $N =3 < 10 = K + K_0 - 1$, while it is still of order $m^{-1/4}$ in the second setting as $N = 30 > 10 = K_0 + K - 1$. 
% \begin{figure}[t!]
%       \centering
%       \subcaptionbox*{\scriptsize (a) Over-fitted setting with $N=3$\par}{\includegraphics[width = 0.45\textwidth]{figures/overfit_fix_alphabar_N3K8-2.png}}
%       \subcaptionbox*{\scriptsize (b) Over-fitted setting with $N=30$\par}{\includegraphics[width = 0.45\textwidth]{figures/overfit_fix_alphabar_N30K8-2.png}}
%       \caption{Posterior contraction rate with different values of $N$}
%     \label{fig:change_N}
% \end{figure}

% \subsection{Learning the Latent Allocation}

\section{Discussions and future work} In this paper, we presented decompositions of Dirichlet moment tensors in terms of diagonal tensors and vice versa, and utilize such decompositions to study identifiability and posterior contraction behavior of parameters that arise in the LDA model. 
Several future research directions may be considered from this mathematical framework. First, it seems promising to utilize the tensor diagonalization in Lemma~\ref{lem:whiten-Dirichlet-moment} to efficiently design robust higher-order moment methods for the LDA model with theoretical guarantees. Efficient moment methods for the sum of power tensors, as well as mixture models, have been developed in the literature by several authors \cite{pereira2022tensor, anandkumar2012spectral, anandkumar2014tensor}, followed by analysis of optimal rates of estimation  \cite{10.1214/19-AOS1873, wei2023minimum}. Second, the posterior contraction behavior we provided can be thought of as that of the finitely supported base measure of a Hierarchical Dirichlet Process model~\cite{nguyen2016borrowing}. It would be of interest to extend our theory to other hierarchical nonparametric Bayesian models based on the Dirichlet processes and normalized complete random measures~\cite{teh2006hierarchical,Lijoi2010-pg}. Last but not the least, the investigated identifiability and parameter contraction rates can be used to develop model selection methods for LDA models, in a similar spirit to mixture models~\cite{guha2021posterior, manole2021estimating, do2024dendrogram}. 

There remain some limitations in our theory and several avenues for addressing them. First, all our results presented in this work is limited to the use of Dirichlet distribution. However, the approach to precisely link the posterior asymptotics of hierarchical models to that of mixture of product distributions is quite general and potentially fruitful in broader contexts, especially for obtaining sharp rates of convergence under minimal identifiability conditions. Thus, it is of interest to broaden this theory to address classes of distribution for exchangeable random partition probability functions that arise in the hierarchical model.
%When placed in the topic modeling theory literature, the theory that we are developing in this paper still has several limitations. Firstly, we assume strictly that the true generating model for the Latent Allocation $q$ is Dirichlet distribution. This assumption is crucial for developing the properties of the moment tensors and is widely applied in practice. 
Alternatively, if no further algebraic consideration is required on the nature of the induced random partition distributions, the geometric approach developed in~\cite{nguyen2015} becomes useful in understanding the contraction behavior of topics in this scenario and its hierarchical extensions ~\cite{chakraborty2024learning}. Secondly, the tensor decomposition equations require the concentration parameter $\overline{\alpha}$ of the Dirichlet distribution to be known. Although this is a usual assumption in the literature for studying estimation rate for Dirichlet distribution~\cite{anandkumar2012spectral} and Dirichlet Process~\cite{nguyen2016borrowing}, it remains an interesting question how to address the setting where $\overline{\alpha}$ is unknown. %There is a recent interest in the mixture modeling literature on studying the relationship between posterior consistency of the number of components and the concentration parameter of Dirichlet Process prior. 
Recent development suggests it may be be fruitful to consider placing a suitable prior on $\overline{\alpha}$ ~\cite{ohn2023optimal, ascolani2023clustering}, however, the situation is quite different with the LDA model since the concentration parameter here is present in the likelihood of the model instead of the prior.
%We made up for the strong assumption of having been given fixed $\overline{\alpha}$ by demonstrating the posterior contraction rate in this setting with simulated data and call for future investigations in this direction.  

\section{Proof of the Dirichlet moment tensor  decompositions}\label{sec:proofs-whitening}
%In this section, we study the representations of Dirichlet moment tensors by diagonal tensors (and vice versa) and their proofs, which play the principal role in our study. 
Recall that the moment tensors of $\Dir(\alpha_1, \dots, \alpha_K)$ are $Q^{(n)} = (Q^{(n)}_{\alpha, k_1\dots k_n})_{k_1, \dots, k_n=1}^{K}$ with
\begin{equation}\label{eq:moment-tensor-Dirichlet}
    Q^{(n)}_{\alpha,k_1 k_2 \dots k_n} = \dfrac{\prod_{k=1}^{K} \alpha_k^{[n_k]}}{\overline{\alpha}^{[n]}}, 
\end{equation}
where we write tensors' indices as subindices, $n_k := \#\{i \in [n]: k_i = k\}$ for $k\in [K]$, and we use the forward polynomial notation $a^{[n]} = a (a + 1) \dots (a + n - 1), \forall a > 0, n\in \Nbb$.
Let $H_{\alpha}^{(n)} = (H_{k_1 \dots k_n})_{k_1, \dots, k_n=1}^{K} = \overline{\alpha}^{[n]} Q_{\alpha}^{(n)}$, where 
\begin{equation}\label{eq:def-H}
    H_{k_1 k_2 \dots k_n} = \prod_{k=1}^{K} \alpha_k^{[n_k]}
\end{equation}
We drop the dependence on $n$ and $\alpha$ in the notation of $H$ (the superscript and subscript) for ease of notation if there is no confusion. The main objective of  Lemma~\ref{lem:whiten-Dirichlet-moment} is to prove:
\begin{equation}\label{eq:whiten-H-tensor-1}
    H_{\alpha}^{(N)} =\left[\sum_{n=1}^{N} \sum_{(S_1, \dots, S_n) \in \Pcal(N, n)} \prod_{i=1}^{n} (|S_i|-1)! T_{(S_1, \dots, S_n)}(\diag_{|S_1|}({\alpha}), \dots, \diag_{|S_n|}({\alpha}) )\right],\end{equation} 
and \begin{equation}\label{eq:whiten-H-tensor-2}
\diag_{N}({\alpha}) = \dfrac{1}{(N-1)!} \left[\sum_{n=1}^{N} (-1)^{n-1} (n-1)!\sum_{(S_1, \dots, S_n) \in \Pcal(N, n)} 
T_{(S_1, \dots, S_n)}(H_{\alpha}^{(|S_1|)}, \dots H_{\alpha}^{(|S_n|)})\right],
\end{equation}
where we recall that $\Pcal(N, n)$ is the set of all partitions of $[N]$ to $n$ non-empty subsets. Two partitions that are permutations of each other are considered to be identical and are only counted once. For example, $(\{1, 2\}, \{3, 4\}) \equiv (\{2, 1\}, \{3, 4\}) \equiv (\{3, 4\}, \{1, 2\})$. 

\subsection{Proof of Dirichlet moment tensors' decompositions}
\begin{proof}[Proof of Lemma~\ref{lem:whiten-Dirichlet-moment}]
The proof is divided into two parts: Decomposition of Dirichlet moment tensors by diagonal tensors and vice versa. The first part (proving Equation~\eqref{eq:whiten-H-tensor-1}) can be viewed as a corollary of Proposition 4.7 in Ghosal and van der Vaart~\cite{ghosal2017fundamentals}, and we present a proof in the appendix for completeness. We focus on the proof of the second part.

\noindent\textbf{Representing diagonal tensors by Dirichlet moment tensors~\eqref{eq:whiten-H-tensor-2}.} Coordinate-wise, Eq.~\eqref{eq:whiten-H-tensor-2} is equivalent to:
\begin{equation}\label{eq:whiten-H}
    1_{(k_1 = \dots = k_N)} (N-1)! \alpha_{k_1} = \sum_{n=1}^{N} (-1)^{n-1} (n-1)! \sum_{(S_1, \dots, S_n) \in \Pcal(N, n)} H_{k_{S_1}} H_{k_{S_2}} \cdots H_{k_{S_n}}, 
\end{equation}
where $H_{k_{S_{i}}}:= H_{k_{S_{i 1}}\dots k_{S_{i s_i}}}$ for $S_i = \{S_{i1}, \dots, S_{i s_i}\}$ (note that the order of the subindices of $H$ does not matter so the notation is valid). Because of the structure of the LHS, it is natural to consider two cases: when $k_1 = \dots = k_N$ and when there are at least two distinct numbers in $\{k_1, \dots, k_N\}$.

\noindent\textbf{Case 1: $k_1 = k_2 = \dots = k_N$:} We prove the equation above by induction on $N$. For the case of $N=1$, it is trivial. For $N=2$, the equation reads:
    $$\alpha_{k_1} = \alpha_{k_1}^{[2]} - \alpha_{k_1}^{2},$$
    which is also correct. Assume the equation is correct until the case of $(N-1)$, i.e,
\begin{equation}\label{eq:whiten-H-N-1}
    1_{(k_1 = \dots = k_{N-1})} (N-2)! \alpha_{k_1} = \sum_{n=1}^{N-1} (-1)^{n-1} (n-1)! \sum_{(S_1, \dots, S_n) \in \Pcal(N-1, n)} H_{k_{S_1}} H_{k_{S_2}} \cdots H_{k_{S_n}}, 
\end{equation}
and we are going to show that Eq.~\eqref{eq:whiten-H} is correct (the case of $N$). To relate the two cases of $(N-1)$ and $N$, we make a correspondence of every partition $(S_1, \dots, S_n)$ in $\mathcal{P}(N-1, n)$ to a sets of $(n+1)$ partitions in $\mathcal{P}(N, n)$ consisting of $(S_1, \dots, S_n, \{N\})$ and $(S_1, \dots, S_i \cup \{N\}, \dots, S_n)$, for $i = 1, \dots, n$. Summing all the terms related to these partitions in the RHS of Eq.~\eqref{eq:whiten-H} gives:
\begin{align}\label{eq:whiten-H-N-vs-N-1}
        (-1)^{n} n! &  H_{k_{S_1}} \cdots 
        H_{k_{S_n}} H_{k_N} + \sum_{i=1}^{n} (-1)^{n-1} (n-1)! H_{k_{S_1}} \cdots H_{k_{S_i\cup\{N\}}} \cdots H_{k_{S_n}} \nonumber \\
        &= \sum_{i=1}^{n} (-1)^{n-1} (n-1)! H_{k_{S_1}} \cdots \left( H_{k_{S_i\cup\{N\}}} - H_{k_{S_i}}H_{k_N}\right)\cdots H_{k_{S_n}} \nonumber\\
        &\overset{(*)}{=}\sum_{i=1}^{n} (-1)^{n-1} (n-1)! |S_i| 
        H_{k_{S_1}} \cdots H_{k_{S_i}} \cdots H_{k_{S_n}} \nonumber\\
        &= \left(\sum_{i=1}^{n}|S_i|\right) \times (n-1)!  
        H_{k_{S_1}} \cdots H_{k_{S_n}} \nonumber\\
        &= (N-1) \times  (-1)^{n-1} (n-1)!  
        H_{k_{S_1}}  \cdots H_{k_{S_n}},
    \end{align}
where in $(*)$ we use the fact that for every subset $S\subset [N-1]$, 
$$H_{k_{S \cup \{N\}}} - H_{k_{S}} H_{k_N} = \alpha_{k_1}^{[|S|+1]} - \alpha_{k_1}^{[|S|]}\alpha_{k_1} = s \alpha_{k_1}^{[|S|]} = |S| H_{k_{S}}.$$

As $(S_1, \dots, S_n)$ ranges over $\mathcal{P}(N-1, n)$, its corresponding set of partitions in $\mathcal{P}(N, n)$ are disjoint and their union is $\mathcal{P}(N, n)$. Hence, summing LHS and RHS of~\eqref{eq:whiten-H-N-vs-N-1} over $(S_1, \dots, S_n)\in \mathcal{P}(N-1, n)$ yields:
\begin{align*}
    \sum_{n=1}^{N} (-1)^{n-1} & (n-1)! \sum_{(S_1, \dots, S_n) \in \Pcal(N, n)} H_{k_{S_1}} H_{k_{S_2}} \cdots H_{k_{S_n}} \\
    & = (N-1) \times \sum_{n=1}^{N-1} (-1)^{n-1} (n-1)! \sum_{(S_1, \dots, S_n) \in \Pcal(N-1, n)} H_{k_{S_1}} H_{k_{S_2}} \cdots H_{k_{S_n}}\\
    & = (N-1) \times (N-2)! \alpha_{k_1}\\
    & = (N-1)! \alpha_{k_1}.
\end{align*} 
where the second equality uses the inductive hypothesis~\eqref{eq:whiten-H-N-1}. 
Hence, Eq.~\eqref{eq:whiten-H} is proved in the case $k_1 = \dots = k_N$. 
% Hence, we can combine terms in the RHS of Eq.~\eqref{eq:whiten-H} together to get $(N-1)$ times the term in the RHS of case $(N-1)$. This is one-to-one. Therefore, when computing the RHS of Eq.~\eqref{eq:whiten-H} for the case $N$, the result is $(N-1)$ times the result we get in case $(N-1)$, of which the LHS is $(N-2)! \alpha_{k_1}$, by the induction hypothesis. Hence, the RHS of Eq.~\eqref{eq:whiten-H} equals $(N-1)! \alpha_{k_1}$, meaning Eq.~\eqref{eq:whiten-H} is correct in this case.
% \textcolor{red}{LN: this last paragraph is a bit poorly written. Unclear English as well as imprecise mathematical statements. E.g., what exactly is one-to-one? "in case $(N-1)$" sounds a bit crude...} \textcolor{blue}{DD: I have improved the writing in this part. Let's check again.}
    
\noindent\textbf{Case 2: There are at least two different values in $k_1, \dots, k_N$:} Our goal is to show that
\begin{equation}\label{eq:whiten-H-diff}
        \sum_{n=1}^{N} (-1)^{n-1} (n-1)! \sum_{(S_1, \dots, S_n) \in \Pcal(N, n)} H_{k_{S_1}} H_{k_{S_2}} \cdots H_{k_{S_n}} = 0.
    \end{equation}
Because there are at least two different values in $k_1, \dots, k_N$, there exists a subset $A\subset [N]$ satisfying that $\{k_{i}: i\in A\}$ has different values than $\{k_{j}: j\in A^{c}\}$. WLOG, assume that $A = [N_1]$ for some $1 < N_1 < N$, i.e., $\{k_{1}, \dots, k_{N_1}\} \cap \{k_{N_1 + 1}, \dots, k_{N}\} = \varnothing$. 
Let $[(N_1+1):N]$ denotes the set of natural numbers from $(N_1 + 1)$ to $N$.
For any $S \subset [N]$, let $S^1 = S \cap [N_1]$ and $S^2 = S \cap [(N_1+1):N]$, by inspecting the moment formula in~\eqref{eq:def-H}, we observe that
$$H_{k_S} = H_{k_{S^1}} H_{k_{S^2}}.$$
Using this identity, we can decompose every term (without coefficients) in Eq.~\eqref{eq:whiten-H-diff} as 
    \begin{equation}\label{eq:whiten-H-separate}
        H_{k_{S_1}} H_{k_{S_2}} \cdots H_{k_{S_n}} = H_{k_{S^1_1}}\cdots H_{k_{S^1_{n_1}}} H_{k_{S^2_1}}\cdots H_{k_{S^2_{n_2}}},
    \end{equation} 
    where $(S^1_1, \dots, S^1_{n_1})$ is a partition of $[N_1]$, $(S^2_1, \dots, S^2_{n_2})$ is a partition of $[(N_1+1): N]$, and they are results of $S_1, \dots, S_n$ when intersecting with $[N_1]$ and $[(N_1+1): N]$. Note that $n_1$ may differ from $n_2$ because there may exist some set $S_i$ having empty intersections with $[N_1]$ (or $[(N_1+1): N]$).

\begin{figure}[h!]
        \centering
        \includegraphics[width=0.8\linewidth]{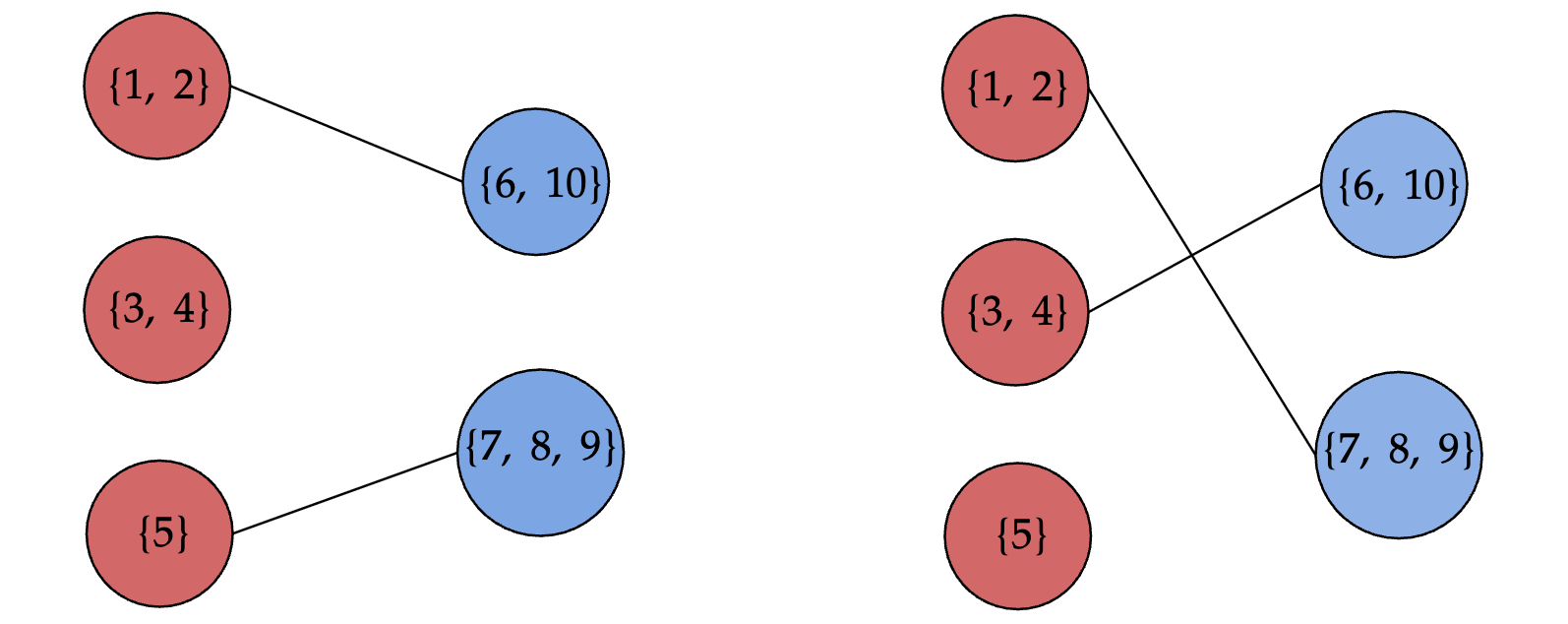}
        \caption{Illustration of Eq.~\eqref{eq:whiten-H-separate} when 
        $N = 10, N_1 = 5 , K = 3, k_1 = \dots = k_5 = 1, k_6 = k_7 = k_8 = 2, k_9 = k_{10} = 3.$ Two different partitions $(\{1, 2, 6, 10\}, \{3, 4\}, \{5, 7, 8, 9\} )$ and $(\{1, 2, 7, 8, 9\}, \{3, 4, 6, 10\}, \{5\} )$ of $[N]$ that have the same decomposition into $\color{red}{S_1^1} = \{1, 2\}, \color{red}{S_2^1} = \{3, 4\}, \color{red}{S_3^1} = \{5\}$ and $\color{blue}{S_1^2} = \{6, 10\}, \color{blue}{S_2^2} = \{7, 8, 9\}$ as in Eq.~\eqref{eq:whiten-H-separate}, resulting in two bipartite graphs having the same set of vertices. There are a total of 6 partitions decomposing into $\color{red}{S_1^1, S_2^1, S_3^1}, \color{blue}{S_2^1, S_2^2}$, and we show that their coefficients in Eq.~\eqref{eq:whiten-H-diff} sum up to 0.}
        \label{fig:LDA-proof-case2}
    \end{figure}
    
Fix an arbitrary set of two partitions $(S^1_1, \dots, S^1_{n_1})$ and $(S^2_1, \dots, S^2_{n_2})$ as in the RHS of Eq.~\eqref{eq:whiten-H-separate}, there may exist several partitions $(S_1, \dots, S_n)$ of $[N]$ decomposing into it (see Figure \ref{fig:LDA-proof-case2}). We will prove that the terms in Eq.~\eqref{eq:whiten-H-diff} corresponding to those partitions sum up to 0. Indeed, consider an undirected bipartite graph with two parts $\{S^1_1, \dots, S^1_{n_1}\}$ and $\{S^2_1, \dots, S^2_{n_2}\}$. Each partition $(S_1, \dots, S_n)$ of $[N]$ in the decomposition Eq.~\eqref{eq:whiten-H-separate} corresponds to a bipartite graph described above with edges connecting two parts (and one node has a degree of maximum 1): $S_k = S_i^1 \cup S_{j}^{2}$ means there is an edge connecting $S_i^1$ and $S_{j}^{2}$. At most $\min\{n_1, n_2\}$ edges are possible. WLOG, we assume $n_1\leq n_2$. For each $k = 1, \dots, n_1$, there are exactly $\binom{n_1}{k} \binom{n_2}{k} k!$ such graphs with $k$ edges (the number of ways to choose $k$ sets from $S^1_1, \dots, S^1_{n_1}$ and from $S^2_1, \dots, S^2_{n_2}$, and there are $k!$ different ways to build edges), where each graph corresponds to a partition of $[N]$ into $n_1 + n_2 - k $ sets. Hence, it suffices to prove that
    \begin{equation}
        \sum_{k=0}^{n_1} (-1)^{n_1 + n_2 - k} (n_1 + n_2 - k - 1)! \binom{n_1}{k} \binom{n_2}{k} k! = 0,
    \end{equation}
for every positive natural number $n_1$
and $n_2$. Dividing LHS by $(-1)^{n_1 + n_2} n_2!$ it is equivalent to
\begin{equation}\label{eq:whiten-H-K=2}
        \sum_{k=0}^{n_1} (-1)^{k} \binom{n_1}{k} (n_2 + n_1 - k - 1)\cdots (n_2-k+1) = 0.
    \end{equation}
    Consider the polynomial $p(X) = X^{n_2 - 1} (X - 1)^{n_1}$. Because this polynomial has root 1 with the multiplicity of $n_1$, its $(n_1 - 1)$--th derivative (denoted by $p^{(n_1-1)}(X)$) still has root 1. Hence, $p^{(n_1-1)}(1) = 0$. Besides,
    $$p^{(n_1-1)}(X) = \sum_{k=0}^{n_1}(-1)^{k} \binom{n_1}{k} (n_2 + n_1 - k - 1)\cdots (n_2-k+1) X^{n_2+k-1}.$$
    By plugging $X = 1$ in, Eq.~\eqref{eq:whiten-H-K=2} is proved. Hence, for every partition $(S^1_1, \dots, S^1_{n_1})$ of $[N_1]$ and $(S^2_1, \dots, S^2_{n_2})$ of $[(N_1 + 1), N]$, we have
    \begin{equation}
        \sum_{n=1}^{N} \sum_{\substack{(S_1, \dots, S_n) \\ \text{satisfies~\eqref{eq:whiten-H-separate}}}}   (-1)^{n} (n-1)! H_{k_{S_1}} H_{k_{S_2}} \cdots H_{k_{S_n}} = 0.
    \end{equation}
    Summing over all possible such partitions of $[N_1]$ and $[(N_1 + 1): N_2]$ is equivalent to summing over all partitions $(S_1, \dots, S_n)$ of $[N]$. Hence, 
    Eq.~\eqref{eq:whiten-H-tensor-2} is correct in the case where there are at least two distinct values in $k_1, \dots, k_N$.

    % \noindent\textbf{Case 2.2: There are more than two distinct values among $k_1, \dots, k_N$:} We can use induction to prove Eq.~\eqref{eq:whiten-H-tensor-2} in this general case. Indeed, we have shown that it holds when there are exactly two distinct values among $k_1, \dots, k_N$. Assume that it is correct whenever $k_1, \dots, k_N$ take value in $\{1, \dots, K-1\}$ for $K\geq 3$, our goal now is to prove Eq.~\eqref{eq:whiten-H-tensor-2} when $k_1, \dots, k_N$ take value in $\{1, \dots, K\}$. When there is $K$ values of $k_1, \dots, k_N$, we can fix any graph of relationship between $S^1_{11}, \dots, S^{1}_{1n_1}, \dots,$ 
    % $S^{K-1}_{11}, \dots, S^{K-1}_{1n_{K-1}}$ and then consider its relationship to $S^K_{11}, \dots, S^{K}_{1n_K}$, where we use the same notations as in the case $K=2$. Consider a new graph where we view each connected component in the graph of $S^{1}, \dots S^{K-1}$ as a new vertex (total $n'_1$ vertices). The result then follows by the case $K=2$ because we already proved that Eq.~\eqref{eq:whiten-H-K=2} is correct for all positive natural numbers $n_1$ and $n_2$ (now replaced by $n_1'$ and $n_K$). Hence, Eq.~\eqref{eq:whiten-H-tensor-2} is completely proved. 
    % \textcolor{red}{LN: this last paragraph is terse and cryptic; worth expanding in further detail if you can.}
\end{proof}

\section*{Acknowledgement}
We are grateful to Dr. Yun Wei (University of Texas at Dallas) for many constructive comments and discussions that helped improve the paper, and for pointing out many typos. Dat Do is grateful to Dr. Yun Wei for discussing Lemma A.8 in~\cite{wei2023minimum} (arXiv v1) and bringing up reference \cite{vandermeulen2019operator}. We thank Dr. Viet Cuong Do (Vietnam National University) for helpful discussions on algebraic geometry and for providing us with many important algebra references.
This research was supported in part by the National Science Foundation, via grants DMS-2015361 (XLN), DMS-2052653 (JT), the National Institute of General Medical Sciences of the NIH under award number R35GM151145, and a research gift from Wells Fargo (XLN). The content is solely the responsibility of the authors and does not necessarily represent the official views of the NIH.
\bibliography{dat}
\bibliographystyle{imsart-number} 

\newpage
%% Multiple Appendixes:                     %%
%%%%%%%%%%%%%%%%%%%%%%%%%%%%%%%%%%%%%%%%%%%%%%
\begin{appendix}

\begin{center}
    \large{Supplemental Material for ``DIRICHLET MOMENT TENSORS AND THE CORRESPONDENCE
    BETWEEN ADMIXTURE AND MIXTURE OF PRODUCT MODELS''}
\end{center}

% \section*{}
\noindent\textbf{Outline of the appendix:} The appendix is organized as follows. In Appendix~\ref{appA}, we provide the proof of the first part of Lemma~\ref{lem:whiten-Dirichlet-moment} regarding representation of the Dirichlet moment tensors by diagonal tensors (by showing Equation~\eqref{eq:whiten-H-tensor-1}). 
In Appendix~\ref{app:metric-equivalent}, we examine the magnitude of constants $C_1$ and $C_2$ in Proposition~\ref{prop:metric-equivalent} by simulations.
In Appendix~\ref{app:recursive-compute-moment-Dir}, we prove Proposition~\ref{prop:recursive-compute-moment-Dir} (recursive computation of the Dirichlet moment tensors) and discuss the connection of this proposition with the tensor representation result from Lemma~\ref{lem:whiten-Dirichlet-moment}. Appendix~\ref{app:proof-inverse-bounds} provides the proof of Theorem~\ref{thm:inverse-bounds}, characterizing the minimal document length required for inverse bounds under exact and over-fitted scenarios under different assumptions on the true topics. In Appendix~\ref{sec:proofs-convergence}, we provide the proof of the posterior contraction results for density and parameter estimation, in particular Theorem~\ref{theorem:contraction_posterior}. This section also contains the relevant results on the geometry of multinomial models, providing upper bounds for various metrics on the space of multinomial densities. Appendix~\ref{sec:proof-allocation} provides the proof for the posterior contraction of the latent allocation, in particular, Theorem~\ref{theorem:contraction-allocation}. Finally, Appendix~\ref{sec:auxiliary-result} contains proofs of other results, including Proposition~\ref{prop:identifiability} on minimal document lengths for identifiability, and Lemma~\ref{lem:fact-distance}, establishing properties on various metrics on the space of distributions, which is used in the proof of Proposition~\ref{prop:metric-equivalent}.

\section{Proof of Lemma~\ref{lem:whiten-Dirichlet-moment} (first part)}\label{appA}

\subsection{Proof of Equation~\eqref{eq:whiten-Q-tensor-1}}
\noindent\textbf{Representing Dirichlet moment tensors by diagonal tensors~\eqref{eq:whiten-H-tensor-1}.} Coordinate-wise, we need to show that:

\begin{equation}\label{eq:write-tensor-LDA-as-mixture}
    H_{k_1,\dots, k_N} = \sum_{n=1}^{N} \sum_{(S_1, \dots, S_n)\in \Pcal(N, n)} \prod_{i=1}^{n} (s_i - 1)! \times 1_{(k_{S_{i1}} = k_{S_{i2}}=\dots = k_{S_{i s_i}})} \alpha_{k_{S_{i1}}}, \,\forall k_1,\dots, k_N\in [K],
\end{equation}
where $S_i = \{S_{i1}, \dots, S_{i s_i}\}, S_{ij}\in [N],\,\forall j\in [s_i]$ and $i \in [n]$. Now fix an arbitrary tuple $(k_1,\dots, k_N) \in [K]^{N}$ and we are going to prove that~\eqref{eq:write-tensor-LDA-as-mixture} is correct. Let $(\mathcal{A}_1,\dots,\mathcal{A}_K)$  be the partition induced by this tuple, where $\mathcal{A}_k = \{t \in [N]: k_t = k\}$ for $k\in [K]$. The LHS of Eq.~\eqref{eq:write-tensor-LDA-as-mixture} equals
\begin{equation}\label{eq:write-tensor-LDA-LHS}
H_{k_1, \dots, k_N} = \alpha_1^{[|\mathcal{A}_1|]} \alpha_2^{[|\mathcal{A}_2|]}\cdots \alpha_{K}^{[|\mathcal{A}_K|]}.    
\end{equation}

Starting from the RHS of Eq. \eqref{eq:write-LDA-as-mixture}, due to the presence of indicator function $1_{(k_{S_{i1}} =\dots = k_{S_{i s_i}})}$, the summand corresponding to partition $(S_1, \dots, S_n)$ is non-zero only if $(S_1, \dots, S_n)$ is a finer partition of the partition $(\mathcal{A}_1, \dots, \mathcal{A}_{K})$, i.e., each $S_i$ is a subset of some $\mathcal{A}_j$. Indeed, if there is $i\in [n]$ and two indices $j_1, j_2\in [K]$ such that $S_i \cap \mathcal{A}_{j_1} \neq \varnothing$ and $S_i \cap \mathcal{A}_{j_2} \neq \varnothing$, then $\{k_{S_{i1}}, \dots, k_{S
_{is_i}}\}$ contains at least two values $j_1$, and $j_2$, which makes the indicator function equals 0. It follows that $(S_1, \dots, S_n)$ is composed of $K$ disjoint subsets $\{S^{1}_1, \dots, S^{1}_{c_1}\}, \dots, \{S^{K}_1, \dots, S^{K}_{c_K}\}$ such that $(S^{j}_1, \dots, S^{j}_{c_j})$ is a partition of $\mathcal{A}_{j}$ for all $j\in [K]$, and $\sum_{j=1}^{K}c_j=n$. The RHS of Eq.~\eqref{eq:write-tensor-LDA-as-mixture} becomes
\begin{equation}
    \sum_{\substack{c_1\in [|\mathcal{A}_1|] \\ (S^{1}_1, \dots, S^{1}_{c_1})\in \Pcal(\mathcal{A}_1, c_1)}}\cdots \sum_{\substack{c_K\in [|\mathcal{A}_K|] \\ (S^{K}_1, \dots, S^{K}_{c_K})\in \Pcal(\mathcal{A}_K, c_K)}} \prod_{j=1}^{K} \left( \prod_{i=1}^{c_j} (|S^{j}_{i}| - 1)! \alpha_{j} \right),
\end{equation}
where $\Pcal(\mathcal{A}_j, c_k)$ is collection of all partitions of $\mathcal{A}_j$ with cardinality $c_j$, for $j\in [K]$ (by slight abuse of notation). By exchanging the product and the sum of the expression in the above display, one obtains
\begin{align}\label{eq:write-tensor-LDA-RHS}
    & \prod_{j=1}^{K} \left(\sum_{\substack{c_j\in [|\mathcal{A}_j|] \\ (S^{j}_1, \dots, S^{j}_{c_j})\in \Pcal(\mathcal{A}_j, c_j)}} \prod_{i=1}^{c_j} (|S_{i}^{j}| - 1)! \alpha_j)\right) \\
    & \overset{}{=} \prod_{j=1}^{K} \left(\sum_{c_j=1}^{|\mathcal{A}_j|} \left(\sum_{(S^{j}_1, \dots, S^{j}_{c_j})\in \Pcal(\mathcal{A}_j, c_j)} \prod_{i=1}^{c_j} (|S_i^{j}| - 1)!\right) \alpha_j^{c_j}\right) \nonumber \\
    & \overset{(a)}{=} \prod_{j=1}^{K} \left(\sum_{c_j=1}^{|\mathcal{A}_j|} \stirlingone{|A_j|}{c_j} \alpha_j^{c_j}\right)\\
    & \overset{(b)}{=} \prod_{j=1}^{K} \alpha_j^{[|\mathcal{A}_j|]},
\end{align}

where in equality (a) we use a property of the unsigned Stirling number of first kind that $\stirlingone{|\mathcal{A}_j|}{c_j}$ is equal to the ways sitting $|\mathcal{A}_j|$ people into $c_j$ round tables, and in the equality (b) is a well-known formula connecting raising factorials and Stirling numbers \cite{charalambides2002enumerative}. Conclude from Eq.~\eqref{eq:write-tensor-LDA-LHS} and~\eqref{eq:write-tensor-LDA-RHS} that the LHS and RHS of Eq.~\eqref{eq:write-tensor-LDA-as-mixture} are equal. Hence, Eq.~\eqref{eq:whiten-Q-tensor-1} is proved.

\section{Discussions and Numerical Experiments for the constants in Proposition~\ref{prop:metric-equivalent}}\label{app:metric-equivalent}
We conduct a simple simulation study to investigate the constants in Proposition~\ref{prop:metric-equivalent}. In particular, we want to understand the magnitude of $C_1$ and $C_2$ in the inequalities:
\begin{equation}
    d_{TV}(p^{\mathcal{L}}_{\alpha, \Theta}, p^{\mathcal{L}}_{\alpha', \Theta'}) \leq C_1(N, \overline{\alpha})\times d_{TV}(p^{\mathcal{M}}_{\tilde{\alpha}, \Theta}, p^{\mathcal{M}}_{\tilde{\alpha}', \Theta'})
\end{equation}
and 
\begin{equation}
    d_{TV}(p^{\mathcal{M}}_{\tilde{\alpha}, \Theta}, p^{\mathcal{M}}_{\tilde{\alpha}', \Theta'}) \leq C_2(N, \overline{\alpha}) \times d_{TV}(p^{\mathcal{L}}_{\alpha, \Theta}, p^{\mathcal{L}}_{\alpha', \Theta'}),
\end{equation}
for every pair of $(\alpha, \Theta)$ and $(\alpha', \Theta')$.

In this simulation study, we set $K = 3, V = 10, N = 5$, and randomly generate pairs $\tilde{\alpha}, \tilde{\alpha}' \sim \mathrm{Dir}_{(1, 1, 1)}$ (uniform on the simplex $\Delta^{K-1}$) and all topics $\Theta_{k}, \Theta'_{k}\sim \mathrm{Dir}_{(1, \dots, 1)}$ (uniform on the simplex $\Delta^{V-1}$) for $k=1, \dots, K$. We then construct $p^{\mathcal{M}}$ and $p^{\mathcal{L}}$ as $V\times V\dots \times V$ ($N$ times) tensor and compute $d_{TV}(p^{\mathcal{M}}_{\tilde{\alpha}, \Theta}, p^{\mathcal{M}}_{\tilde{\alpha}', \Theta'})$ as half of the $\ell_1$ distance between two tensors (similar for $d_{TV}(p^{\mathcal{L}}_{{\alpha}, \Theta}, p^{\mathcal{L}}_{{\alpha}', \Theta'})$). We generate 1000 pairs of parameters and plot them in Figure~\ref{fig:ADM-vs-MM} with the corresponding estimated lower and upper slope $C_1$ and $C_2$. We consider two settings where $\overline{\alpha} = 1$ and $\overline{\alpha} = 0.1$. 

\begin{figure}
      \centering
      \subcaptionbox*{\scriptsize (a) Setting $N = 5, K = 3$, and $\overline{\alpha} = 1$\par}{\includegraphics[width = 0.495\textwidth]{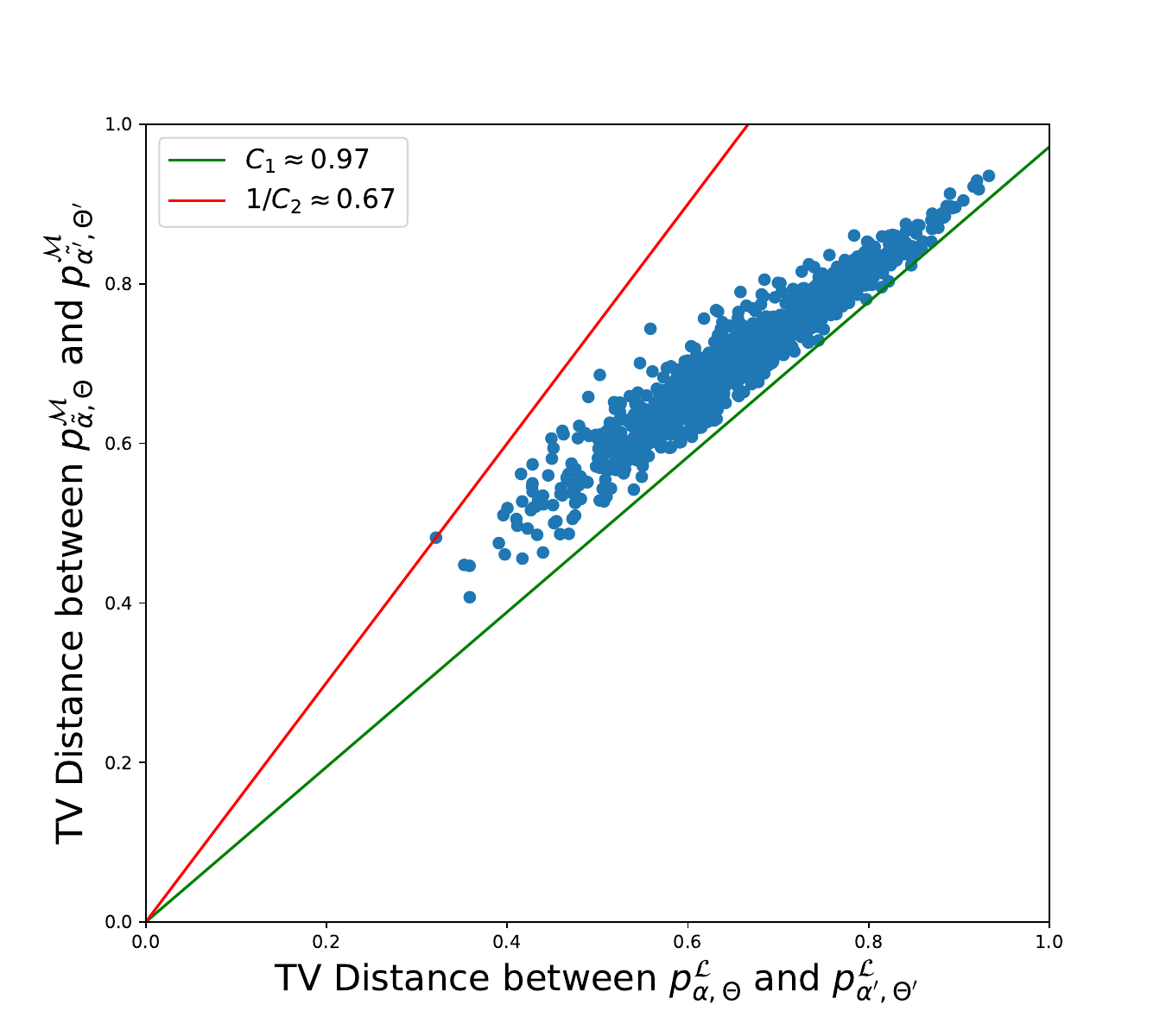}}
      \subcaptionbox*{\scriptsize (b) Setting $N=5, K = 3$, and $\overline{\alpha} = 0.1$ \par}{\includegraphics[width = 0.495\textwidth]{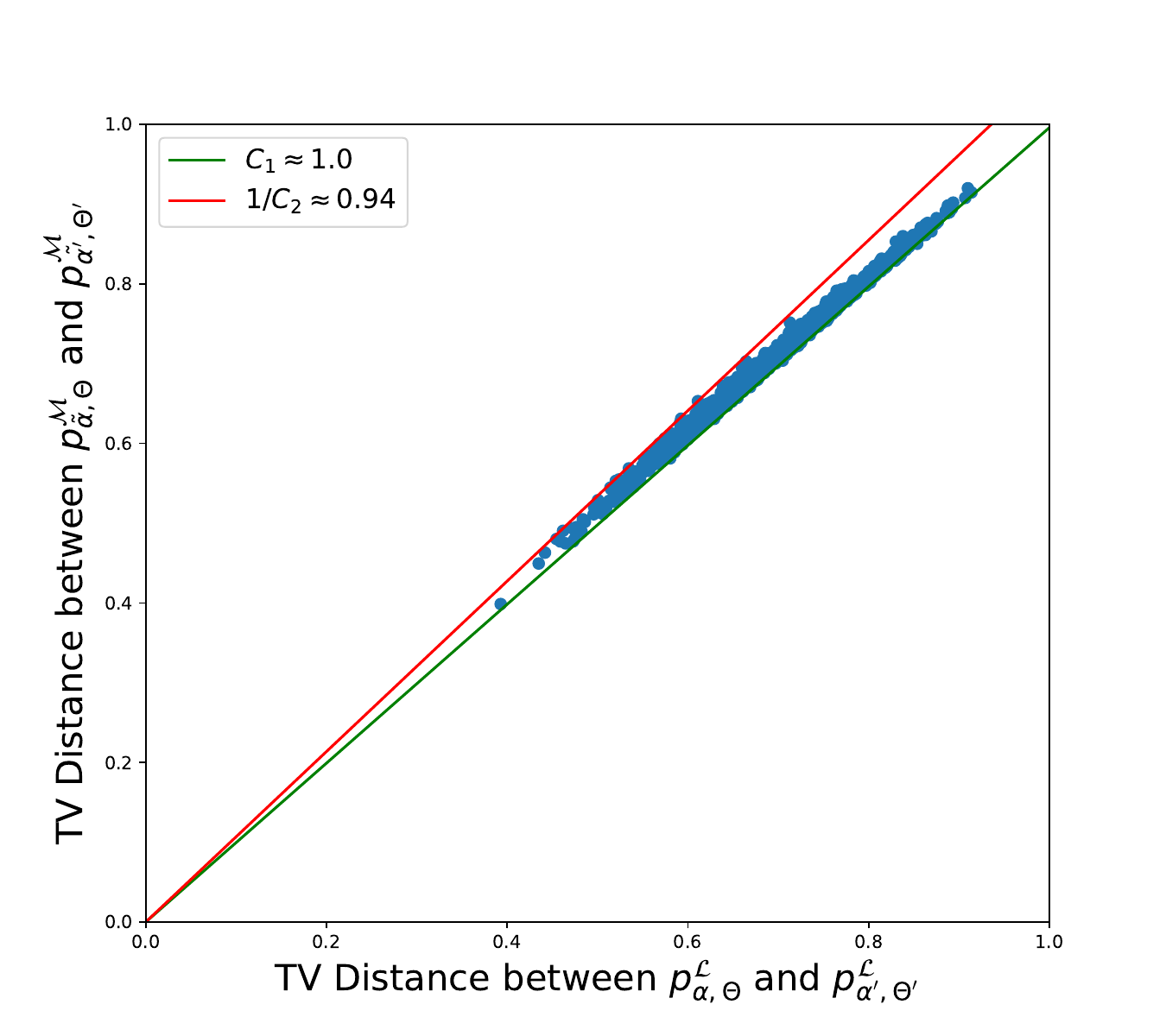}}
      \caption{Comparing $d_{TV}$ between Mixture models and corresponding Admixture models.}
    \label{fig:ADM-vs-MM}
\end{figure}

It can be seen that the estimated $C_1$, which is calculated by the minimum of $\dfrac{d_{TV}(p^{\mathcal{L}}_{{\alpha}, \Theta}, p^{\mathcal{L}}_{{\alpha}', \Theta'})}{d_{TV}(p^{\mathcal{M}}_{\tilde{\alpha}, \Theta}, p^{\mathcal{M}}_{\tilde{\alpha}', \Theta'})}$ over all pairs of $(\alpha, \Theta)$ and $(\alpha', \Theta')$ is very close to 1. On the other hand, $C_2$, which is calculated by the minimum of $\dfrac{d_{TV}(p^{\mathcal{M}}_{\tilde{\alpha}, \Theta}, p^{\mathcal{M}}_{\tilde{\alpha}', \Theta'})}{d_{TV}(p^{\mathcal{L}}_{{\alpha}, \Theta}, p^{\mathcal{L}}_{{\alpha}', \Theta'})}$ over all pairs of $(\alpha, \Theta)$ and $(\alpha', \Theta')$ is around $\dfrac{3}{2}$ when $\overline{\alpha} = 1$ and is close to 1 when $\overline{\alpha} = 0.1$ ($1/C_2$ is displayed in Figure~\ref{fig:ADM-vs-MM} because $d_{TV}(p^{\mathcal{L}}_{{\alpha}, \Theta}, p^{\mathcal{L}}_{{\alpha}', \Theta'})$ is plotted in the x-axis and $d_{TV}(p^{\mathcal{M}}_{\tilde{\alpha}, \Theta}, p^{\mathcal{M}}_{\tilde{\alpha}', \Theta'})$ is plotted in the y-axis). Those estimates together show that the geometry of the mixture models and admixture models are quite close, and a better bound for the constants $C_1$ and $C_2$ could be obtained with more sophisticated mathematical techniques. Finally, it is worth emphasizing that the constants $C_1$ and $C_2$ should not be confused with the constant $C$ appearing in the parameter estimation rates of both models in Theorem~\ref{theorem:contraction_posterior}, which only depends on the smallest $N$ so that the inverse bounds (Theorem~\ref{thm:inverse-bounds}) hold -- this can be as small as 3 (or 4) in the exact-fitted (or over-fitted) setting and linear independent topics (see Remark 6). 

\section{Proof and discussion of Proposition~\ref{prop:recursive-compute-moment-Dir}}\label{app:recursive-compute-moment-Dir}
\subsection{Proof of Proposition~\ref{prop:recursive-compute-moment-Dir}}
\begin{proof}[Proof of Proposition~\ref{prop:recursive-compute-moment-Dir}]
By Lemma~\ref{lem:whiten-Dirichlet-moment},
\begin{align}\label{eq:consequence-prop1}
Q_{\alpha}^{(N)} = \dfrac{1}{\overline{\alpha}^{[N]}} \sum_{\ell=1}^{N}\sum_{(S_1, \dots, S_{\ell})\in \Pcal(N, \ell)} \prod_{i=1}^{\ell}(|S_i| - 1)! T_{(S_1, \dots, S_{\ell})} (\otimes_{i=1}^{\ell} \diag_{|S_i|}(\alpha)) \nonumber \\ 
= \dfrac{1}{\overline{\alpha}^{[N]}} \sum_{\ell=1}^{N}\sum_{\substack{d_1, \dots, d_{\ell}\geq 1 \\ \sum_{i=1}^{\ell} d_i = N}} \sum_{\substack{(S_1, \dots, S_{\ell}): \\ |S_1| = d_1, \dots, |S_{\ell}| = d_{\ell}}} \prod_{i=1}^{\ell}(d_i - 1)! T_{(S_1,\dots, S_{\ell})} (\otimes_{i=1}^{\ell} \diag_{d_i}({\alpha})).
\end{align}
Consider acting $[x, \dots, x]$ ($N$ times) on both sides of this equation. The left-hand side becomes $\Ebb_{q\sim \Dir_{\alpha}} (\sum_{k=1}^{K} q_k x_k)^{N}$. For the RHS, we note that
$$T_{(S_1,\dots, S_{\ell})} (\otimes_{i=1}^{\ell} \diag_{d_i}({\alpha}))[x, x,\dots, x] = \prod_{i=1}^{\ell} \overline{x}_{(d_i)},$$
for any $x\in \Rbb^{K}$ and permutation $T_{(S_1,\dots, S_{\ell})}$ such that $|S_i| = d_i$. Now, we encode each partition by \textit{frequency} instead of \textit{composition}, i.e., we correspond each tuple $(d_1, \dots, d_{\ell})$ to tuple $(p_1, \dots, p_N)$ such that $p_{j}=\#\{i: d_{i} = j\}$. Acting $[x, \dots, x]$ to both sides of Eq.~\eqref{eq:consequence-prop1} yields
\begin{equation}
    \Ebb_{q\sim \Dir_{\alpha}} \left(\sum_{k=1}^{K} q_k x_k\right)^{N} = \dfrac{1}{\overline{\alpha}^{[N]}} \sum_{\ell=1}^{N} \sum b_{\ell}(p_1, \dots, p_N) (\prod_{n=1}^{N} ((n-1)!)^{p_{n}}) \overline{x}_{(1)}^{p_1} \overline{x}_{(2)}^{p_2} \cdots \overline{x}_{(N)}^{p_N},
\end{equation}
where the second sum ranges over all partitions of $N$ into $\ell$ parts, that is, ranging over all non-negative integer solutions $(p_1, \dots, p_N)$ of the equation $p_1 + \dots + p_N = \ell, p_1 + 2p_2 + \dots + N p_N = N$, and $b_{\ell}(p_1, \dots, p_N)$ is the number of ways to partition $N$ into $\ell$ subsets such that there are $p_1$ sets of 1 element, $p_2$ sets of $2$ elements,..., $p_N$ sets of $N$ elements, which can be explicitly calculated by
$$b_{\ell}(p_1, \dots, p_N) =  \dfrac{N!}{\prod_{n=1}^{N} (p_{n}! (n!)^{p_{n}})}$$
Hence, with $c_N = \overline{\alpha}^{[N]} / N!$, we have
\begin{equation}
    c_N \mathfrak{M}_N = B_{N}(0! \overline{x}_{(1)}, 1! \overline{x}_{(2)}, \dots, (N-1)! \overline{x}_{N}) / N!,
\end{equation}
where $B_N$ is the $N$-th Bell polynomial. Recall the recurrent formula for Bell polynomials (see, e.g., \cite{charalambides2002enumerative} Theorem 11.2.) $B_N \equiv B_N(y_1, \dots, y_N)$ for a sequence $y_1, y_2, \dots$ is
\begin{equation}
B_{N+1} = \sum_{\ell=0}^{N} \binom{N}{\ell} y_{\ell+1} B_{N-\ell}.
\end{equation}
Equivalently,
\begin{equation}
\dfrac{B_{N+1}}{(N+1)!} = \dfrac{1}{N+1}\sum_{\ell=0}^{N} \dfrac{y_{\ell+1}}{\ell!} \dfrac{B_{N-\ell}}{(N-\ell)!}.
\end{equation}
Substituting $y_{\ell+1} = (\ell!) \overline{x}_{\ell+1}$ and $c_N \mathfrak{M}_N = B_{N} / N!$, we have
$$c_{N+1} \mathfrak{M}_{N+1} = \dfrac{1}{N+1} \sum_{\ell=0}^{N} \overline{x}_{\ell} c_{N-\ell} \mathfrak{M}_{N-\ell}.$$
Hence, Proposition~\ref{prop:recursive-compute-moment-Dir} is proved.
\end{proof}

We put a few remark on the constant $c_n = \dfrac{\overline{\alpha}^{[n]}}{n!}$. Because $\overline{\alpha}^{[n]} = \dfrac{\Gamma(\overline{\alpha}+n)}{\Gamma(\overline{\alpha})}$, by investigating the behavior of Gamma functions through Stirling's formula, we have
\begin{enumerate}
    \item[(i)] $\overline{\alpha}^{[n]} \asymp \dfrac{1}{\Gamma(\overline{\alpha})}\sqrt{2\pi(\overline{\alpha}+n-1)} \left(\dfrac{\overline{\alpha}+n-1}{e}\right)^{\overline{\alpha}+n-1}$ as $n\to \infty$; 
    \item[(ii)] $\overline{\alpha}^{[n]}$ is an increasing polynomial of $\overline{\alpha}$; 
    \item[(iii)] $\dfrac{\overline{\alpha}^{[n]}}{\overline{\alpha}_{1}^{[n]}} \asymp \dfrac{\Gamma(\overline{\alpha}_1)}{\Gamma(\overline{\alpha})} (\overline{\alpha}_1 + n -1)^{\overline{\alpha} - \overline{\alpha}_1}$ as $n\to \infty$ when $\overline{\alpha} <  \overline{\alpha}_{1}$.
\end{enumerate}
Because $n! = 1^{[n]}$, from (iii) we have that $c_n\asymp n^{\overline{\alpha} - 1}$. 

\subsection{Discussion on the connection between Proposition~\ref{prop:recursive-compute-moment-Dir} and Lemma~\ref{lem:whiten-Dirichlet-moment}} Having seen that Lemma~\ref{lem:whiten-Dirichlet-moment} implies Proposition~\ref{prop:recursive-compute-moment-Dir}, we now talk about the inverse implication. That is, we are going to show that if Proposition~\ref{prop:recursive-compute-moment-Dir} holds, then the Dirichlet moment tensor $Q^{(N)}_{\alpha}$ and the tensor
\begin{equation}
    \tilde{Q}_{\alpha}^{(N)} := \dfrac{1}{\overline{\alpha}^{[N]}} \sum_{\ell=1}^{N} \sum_{(S_1, \dots, S_{\ell})\in \Pcal(N, \ell)} \prod_{i=1}^{\ell}(|S_i| - 1)! T_{(S_1, \dots, S_{\ell})} (\otimes_{i=1}^{\ell} \diag_{|S_i|}(\alpha)),
\end{equation}
must be equal. Indeed, fixed an $\alpha\in \Rbb_+^{K}$, define an $N$-th order polynomial $Q_{N}(x)$ for a variable $x\in \Rbb^{K}$ to be 
$$Q_N(x) := \dfrac{1}{\overline{\alpha}^{[N]}}D_N(\overline{x}_{(1)}, \dots, \overline{x}_{(N)}),$$ for $\overline{x}_{(n)} := \sum_k \tilde{\alpha}_k x_k^{n}$ for all $n, N\in \Nbb$. Taking $Q_0 \equiv 1$ as convention. The recursive formula becomes:
$$Q_{N}(x) = \dfrac{(N-1)!}{\overline{\alpha}^{[N]}} \sum_{\ell=0}^{N-1} \overline{x}_{(\ell+1)} \dfrac{\overline{\alpha}^{[N-1-\ell]}}{(N-1-\ell)!} Q_{N-1-\ell}(x),$$
which uniquely determines the polynomial $Q_{N}(x)$ from $\{Q_{n}(x): n\leq N-1\}$. Proposition~\ref{prop:recursive-compute-moment-Dir} states that $Q^{(N)}_{\alpha}[x, \dots, x]$ is a polynomial that satisfies this recursive formula. However, in the proof of Proposition~\ref{prop:recursive-compute-moment-Dir} after the first step, we only make use of some algebraic transformations of the RHS of Eq.~\eqref{eq:consequence-prop1} to show that $\tilde{Q}^{(N)}_{\alpha}[x, \dots, x]$ also satisfies this recursive formula (without using Lemma~\ref{lem:whiten-Dirichlet-moment}). Hence 
$$Q^{(N)}_{\alpha}[x, \dots, x] = \tilde{Q}^{(N)}_{\alpha}[x, \dots, x], \quad \forall x\in \Rbb^{K}.$$
By using the usual inner product notation of tensors, we can equivalently write
$$\innprod{Q^{(N)}_{\alpha} - \tilde{Q}^{(N)}_{\alpha}, x^{\otimes N}} = 0, \quad\forall x\in \Rbb^{K}.$$
Because $Q^{(N)}_{\alpha} - \tilde{Q}^{(N)}_{\alpha}$ is a symmetric tensor, this implies its spectral norm equals 0 thanks to Banach's theorem~\cite{friedland2018nuclear, banach1938homogene}. Hence, $Q^{(N)}_{\alpha} = \tilde{Q}^{(N)}_{\alpha}$. In conclusion, we can use the result of Proposition~\ref{prop:recursive-compute-moment-Dir} (if correct) to deduce the Eq.~\eqref{eq:whiten-Q-tensor-1} in Lemma~\ref{lem:whiten-Dirichlet-moment} as well.

\section{Proof of Theorem~\ref{thm:inverse-bounds}}\label{app:proof-inverse-bounds}
Before proving the inverse bounds in Theorem~\ref{thm:inverse-bounds}, it is useful to recall some notations for calculating with multi-indices. For a vector $\theta\in \Rbb^{V}$ and a tuple $\alpha \in \Nbb^{V}$, let 
$$\theta^{\alpha} = \prod_{v=1}^{V} \theta_{v}^{\alpha_{v}}, \quad \alpha! = \prod_{v=1}^{V} \alpha_{v}!, \quad |\alpha| = \sum_{v=1}^{V} \alpha_{v},$$
where $\theta^{\alpha}$ denotes a monomial of $\theta$ with degree $|\alpha|$. Given a function $f : \Rbb^{V}\to \Rbb$, let 
$$\dfrac{\partial^{|\alpha|}}{\partial \theta^{\alpha}} f(\theta) = \dfrac{\partial^{|\alpha|}}{\partial \theta_1^{\alpha_1} \partial \theta_2^{\alpha_2}\dots \theta_V^{\alpha_V}} f(\theta).$$
Because the proof of Theorem~\ref{thm:inverse-bounds} concerns various different settings of true topics, we divide it into several subsections. Thanks to the equivalence between $p^{\mathscr{L}}$ and $p^{\mathscr{M}}$ (Proposition~\ref{prop:metric-equivalent}), it suffices to provide the inverse bound for $p^{\mathscr{M}}$. To see this, suppose we have established Theorem~\ref{thm:inverse-bounds} for $p^{\mathscr{M}}$, and we want to prove the bound on $\mathfrak{N}_1(G_0)$ in part (a) for $p^{\mathscr{L}}$, we can argue as follows. Let $\mathfrak{N}^{\mathscr{M}}_1(G_0)$ (resp. $\mathfrak{N}^{\mathscr{L}}_1(G_0)$)  be the associated minimal document length for $p^{\mathscr{M}}$ (resp. $p^{\mathscr{L}})$. By combining Proposition~\ref{prop:metric-equivalent} and the result regarding $p^{\mathscr{M}}$, we get for $N= \mathfrak{N}^{\mathscr{M}}_1(G_0)$
$$W_1(G, G_0) \lesssim d_{TV}(p_{G,N}^{\mathscr{M}}, p_{G_0,N}^{\mathscr{M}}) \leq C_2(\mathfrak{N}^\mathscr{M}_1(G_0),\bar{\alpha})d_{TV}(p_{G,N}^{\mathscr{L}}, p_{G_0,N}^{\mathscr{L}}),$$
which essentially shows that $\mathfrak{N}_1^{\mathscr{L}}(G_0) \leq \mathfrak{N}_1^{\mathscr{M}}(G_0)$, which is upper bounded by $2K_0-1$ using the result for $p^{\mathscr{M}}$. Similarly, for all the results in Theorem~\ref{thm:inverse-bounds}, it is sufficient to prove the statements with $p^{\mathscr{M}}$.

Now, we recall,
$$p_{G, N}^{\mathscr{M}}(x) = \sum_{j=1}^{K} p_k f_{N}(x | \theta_j) = \sum_{j=1}^{K} p_j \theta_{jx_1}\cdots \theta_{j x_N},$$
for mixing measure $G = \sum_{j=1}^{K} p_j \delta_{\theta_j}$ and the product-multinomial kernel $f_N(x | \theta) = \prod_{n=1}^{N} \theta_{x_n}$, for $\theta\in \Delta^{V-1}$ and $x = (x_1, \dots, x_N) \in [V]^{N}$. The proof's outline for each setting is similar to other inverse bounds in the literature~\cite{Chen1992, nguyen2013convergence, Ho-Nguyen-EJS-16} utilizing the strongly identifiable conditions of mixtures. The challenging difference here is the last step: To prove the strong identifiability conditions for the mixture of multinomial distributions. The discreteness of the distributions makes it quite different from the strong identifiability proof of continuous distributions such as Normal or Student distributions. The proof's technique of this step for part (a) is generalized from the technique to compute the Vandermonde matrix's determinant using polynomials. For parts (b) and (c), we can exploit the extra assumption on linear independence and utilize the dual basis of the topics, which is a familiar technique in multi-linear algebra.
\subsection{Proof of Theorem~\ref{thm:inverse-bounds}(a)}
% \begin{proof}[Proof of Theorem~\ref{thm:inverse-bounds}(a)]
Recall that in this part, we only assume that $\theta_1^0, \dots, \theta_{K_0}^{0} \in \Delta^{V-1}$ are distinct, which means that for every pair $(\theta_i^0, \theta_j^0)$ for $i\neq j \in [K_0]$, there exists an index $v = v(i, j)\in V$ such that $\theta_{iv}^0 \neq \theta_{jv}^{0}$.

\subsubsection{Proof of the exact-fitted inverse bound for distinct topics} 
\begin{proof}[Proof of Theorem~\ref{thm:inverse-bounds}(a) regarding $\mathfrak{N}_1(G_0)$]
We aim to show that for $N = 2K_0 - 1$, it holds that
\begin{equation}\label{eq:inverse-bound-N1-distinct} 
d_{TV}(p^{\mathscr{M}}_{G, N}, p^{\mathscr{M}}_{G_0, N}) \gtrsim W_1(G, G_0), 
\end{equation}
as $d_{TV}(p^{\mathscr{M}}_{G, N}, p^{\mathscr{M}}_{G_0, N})\to 0$ for $G \in \Ecal_{K_0}$. The proof is divided into smaller steps. 

\noindent\textbf{Step 1: Proving by contradiction and setup.} Suppose that this bound is not correct. Then there exists a sequence $(G_n) \subset \Ecal_{K_0}$ such that $d_{TV}(p^{\mathscr{M}}_{G_n, N}, p^{\mathscr{M}}_{G_0, N})\to 0$ and 
\begin{equation}\label{eq:contradict-N1-distinct}
    \lim_{n\to \infty} \dfrac{d_{TV}(p^{\mathscr{M}}_{G_n, N}, p^{\mathscr{M}}_{G_0, N})}{W_1(G_n, G_0)} = 0.
\end{equation}
Because $\Delta^{V-1}$ is compact, so is $\Ocal_{K_0}$. Therefore, every subsequence of $(G_n)$ has a further subsequence that converges. Let one of those limits be $G_* \in \Ocal_{K_0}$. Because $d_{TV}$ is continuous in both of its argument, we have $d_{TV}(p^{\mathscr{M}}_{G_*, N}, p^{\mathscr{M}}_{G_0, N}) = 0$, which combined with the fact that $\mathfrak{N}_{e}(G_0) \leq 2K_0 - 1$, implies that $G_* = G_0$. Hence, every subsequence of $(G_n)$, if converges, must go to $G_0$. It is deduced that $G_n \to G_0$ in $W_1$ distance as $n\to \infty$.

\noindent\textbf{Step 2: Taylor expansion around $G_0$.} From the last step, we can have the following representation of $G_n$:
\begin{equation}
    G_n = \sum_{j=1}^{K_0} p_j^{n} \delta_{\theta_j^{n}} \to G_0 = \sum_{j=1}^{K_0} p_j^{0} \delta_{\theta_j^{0}},
\end{equation}
such that $p_j^{n} \to p_j^0$ and $\theta_j^{n} \to \theta_j^{0}$ for all $j\in [K_0]$. For every $x\in [V]^{N}$, using Taylor expansion around each $f_N(x | \theta_j^0)$ gives: 
\begin{align*}
    p^{\mathscr{M}}_{G_n, N}(x) - p^{\mathscr{M}}_{G_0, N}(x) & = \sum_{j=1}^{K_0} (p_j^{n} - p_j^0) f_{N}(x | \theta_j^0) + \sum_{j=1}^{K_0} p_j^{n} \left(f_N(x | \theta_j^{n}) - f_N(x | \theta_j^{0}) \right)  \\
    & = \sum_{j=1}^{K_0} (p_j^{n} - p_j^0) f_{N}(x | \theta_j^0) + \sum_{j=1}^{K_0} \sum_{\substack{\alpha\in \Nbb^{V} \\ |\alpha|=1}} p_j^{n} (\theta_j^{n} - \theta_j^{0})^{\alpha} \dfrac{\partial}{\partial \theta^{\alpha}} f_N(x | \theta_j^{0}) + R_n(x)\\
    & = \sum_{j=1}^{K_0} \sum_{\substack{\alpha\in \Nbb^{V} \\ |\alpha|\leq 1}} a^{n}_{j\alpha}\dfrac{\partial^{|\alpha|}}{\partial \theta^{\alpha}} f_N(x | \theta_j^{0}) + R_n(x),
\end{align*}
where $R_n(x) = o\left(\sum_{j=1}^{K_0}\sum_{v=1}^{V} p_j^{n} |\theta_{jv} - \theta_{jv}|\right) = o\left(\sum p_j^{n} \norm{\theta_j^{n} - \theta_j^{0}}\right) = o(W_1(G_n, G_0))$ for all $x$, and 
$$a_{j\alpha}^{n} = p_j^{n} (\theta_j^{n} - \theta_{j}^{0})^{\alpha},$$
for all tuple $\alpha\in \Nbb^{V}$. 

\noindent\textbf{Step 3: Non-vanishing coefficients.} Because
$$\sum_{j,\alpha} |a_{j\alpha}^{n}| = \sum_{j=1}^{K_0} |p_j^{n} - p_j^0| + p_j^{n} \norm{\theta_j^n - \theta_j^0}_1 \gtrsim W_1(G_n, G_0), $$
we have that $\max_{j, \alpha} |a_{j\alpha}^{n}| / W_1(G_n, G_0)$ is bounded below by a positive constant as $n\to \infty$. Let $d_n = \max_{j, \alpha} |a_{j\alpha}^{n}|$. By compactness argument and extracting a subsequence, if needed, we have
$$\dfrac{a_{j\alpha}^{n}}{d_n} \to a_{j\alpha} \in [-1, 1], \quad \forall j\in [K_0], \alpha\in \Nbb^{V},$$
and at least one of the limit $a_{j\alpha}$ is $\pm 1$. Furthermore, $R_n(x) = o(W_1(G_n, G_0)) = o(d_n)$ for all $x\in [V]^{N}$.

\noindent\textbf{Step 4: Deriving identifiability equation.} From~\eqref{eq:contradict-N1-distinct}, we have
\begin{align*}
    0 & = \lim_{n\to \infty} \dfrac{2d_{TV}(p^{\mathscr{M}}_{G_n, N}, p^{\mathscr{M}}_{G_0, N})}{W_1(G_n, G_0)} \\
    & = \lim_{n\to \infty} \dfrac{2d_{TV}(p^{\mathscr{M}}_{G_n, N}, p^{\mathscr{M}}_{G_0, N})}{d_n}\\
    & = \lim_{n\to \infty} \sum_{x\in [V]^{N}} \left|\sum_{j=1}^{K_0} \sum_{|\alpha|\leq 1} \dfrac{a_{j\alpha}^{n}}{d_n} \dfrac{\partial^{|\alpha|}}{\partial \theta^{\alpha}} f_N(x | \theta_j^0) + \dfrac{R_n(x)}{d_n}\right|\\
    &\geq \sum_{x\in [V]^{N}}\left|\lim_{n\to \infty} \sum_{j=1}^{K_0} \sum_{|\alpha|\leq 1} \dfrac{a_{j\alpha}^{n}}{d_n} \dfrac{\partial^{|\alpha|}}{\partial \theta^{\alpha}} f_N(x | \theta_j^0) + \dfrac{R_n(x)}{d_n}\right|\\
    & = \sum_{x\in [V]^{N}}\left|\sum_{j=1}^{K_0} \sum_{|\alpha|\leq 1} a_{j\alpha} \dfrac{\partial^{|\alpha|}}{\partial \theta^{\alpha}} f_N(x | \theta_j^0) \right|.
\end{align*}
Hence, 
\begin{equation}\label{eq:identifiability-N1-distinct}
    \sum_{j=1}^{K_0} \sum_{|\alpha|\leq 1} a_{j\alpha} \dfrac{\partial^{|\alpha|}}{\partial \theta^{\alpha}} f_N(x | \theta_j^0) = 0, \quad \forall x\in [V]^{N}, 
\end{equation}
with at least one of $a_{j\alpha}$ is different from 0. By marginalizing out some dimensions of $x$, we also obtain
\begin{equation}\label{eq:identifiability-N1-distinct-N-prime}
    \sum_{j=1}^{K_0} \sum_{|\alpha|\leq 1} a_{j\alpha} \dfrac{\partial^{|\alpha|}}{\partial \theta^{\alpha}} f_{N'}(x | \theta_j^0) = 0, \quad \forall x\in [V]^{N'}, 
\end{equation}
for all $N'\leq N$.

\noindent\textbf{Step 5: Deriving a contradiction from the identifiability equation using polynomials manipulation.} For each $\theta \in \Delta^{V-1}$, because
$$f_{N'}(x | \theta) = \prod_{i=1}^{N'} \theta_{x_i}$$
is a monomial of $\theta$ of degree $N'$.
By varying $x\in [V]^{N'}$, we see that $\{f_{N'}(x|\theta): x\in [V]^{N'}\}$ can be viewed as the collection of all monomials of $\theta$ of degree $N'$. Eq.~\eqref{eq:identifiability-N1-distinct-N-prime} implies
\begin{equation}\label{eq:identifiability-N1-distinct-poly}
    \sum_{j=1}^{K_0} \sum_{|\alpha|\leq 1} a_{j\alpha} \dfrac{\partial^{|\alpha|}}{\partial \theta^{\alpha}} P(\theta_j^{0}) = 0,
\end{equation}
for all monomials $P(\theta)$ having degree of $N'\leq N$. Because of the additivity of this equation, we have it also holds for all polynomials $P(\theta)$ having degree of $N'\leq N = 2K_0 - 1$. We are going to prove that $a_{1\alpha} = 0$ for all $\alpha$, and it is similar for the rest of $a_{j\alpha}$ for $j \in [K_0]$. Indeed, because $\theta_1^0,\dots, \theta_{K_0}^0$ are distinct, for every $j \geq 2$, there exists $v(j) \in [V]$ such that $\theta_{1v(j)}^0 \neq \theta_{j v(j)}^0$. Moreover, for each $i\in [V]$, denote by $\alpha(i)$ the tuple in $\Nbb^{V}$ having $i-$th element being 1 and the rest being 0. 

For an index $i\in [V]$ and variable $\theta = (\theta_{1}, \dots, \theta_{V}) \in \Rbb^{V}$, consider the following polynomial of $\theta$:
$$P_{1i}(\theta) = (\theta_{i} - \theta_{1i}^0) \prod_{j\geq 2} (\theta_{v(j)} - \theta^{0}_{j v(j)})^{2},$$
which has degree $1 + 2(K_0 - 1) = 2K_0 - 1$. Because this polynomial has root $\theta_{j v(j)}$ with multiplicity of 2 for all $j\geq 2$, we have that 
$$\dfrac{\partial^{|\alpha|}}{\partial \theta^{\alpha}} P_{1i}(\theta_{j}^0) = 0, \quad \forall j\geq 2, |\alpha| \leq 1.$$
Similarly, because it also has root $\theta_{1i}^0$, we also have $\dfrac{\partial^{|\alpha|}}{\partial \theta^{\alpha}} P_{1i}(\theta_{1}^0) = 0$ for all $\alpha\neq \alpha(i)$. Therefore, by plugging $P = P_{1i}$ into Eq.~\eqref{eq:identifiability-N1-distinct-poly}, only one term that does not vanish, which is $\dfrac{\partial}{\partial \theta^{\alpha(i)}} P_{1i}(\theta_1^0)$. Hence,
$$0 = a_{1\alpha(i)} \dfrac{\partial}{\partial \theta^{\alpha(i)}} P_{1i}(\theta_1^0) =  a_{1\alpha(i)} \prod_{j\geq 2} (\theta_{1v(j)}^0 - \theta_{j v(j)}^{0})^2.$$
Because $(\theta_{1v(j)}^0 - \theta_{j v(j)}^{0})\neq 0$ for all $j\geq 2$, we have $a_{1\alpha(i)} = 0$. As this holds for all $i\in [V]$, we have that $a_{1\alpha} = 0$ for every tuple $\alpha$ such that $|\alpha| = 1$. Finally, to show that $a_{1\alpha} = 0$ for tuple $\alpha = 0 \in \Nbb^{V}$, we can consider the polynomial
$$P_{10}(\theta) = \prod_{j\geq 2} (\theta_{v(j)} - \theta^{0}_{j v(j)})^{2}$$
having degree of $2(K_0 - 1)$. By arguing similarly to the above, we also have 
$$\dfrac{\partial^{|\alpha|}}{\partial \theta^{\alpha}} P_{10}(\theta_{j}^0) = 0, \quad \forall j\geq 2, |\alpha| \leq 1.$$
Hence, by plugging $P = P_{10}$ into Eq.~\eqref{eq:identifiability-N1-distinct-poly} and notice that $a_{1\alpha}=0$ for all $|\alpha|=1$, we are left with 
$$0 = a_{10} \prod_{j\geq 2} (\theta_{1v(j)}^0 - \theta_{j v(j)}^{0})^2,$$
which implies $a_{10} = 0$. Thus $a_{1\alpha} = 0$ for all $|\alpha|\leq 1$. By analogous arguments, we also have $a_{j\alpha} = 0$ for all $j\in [K_0]$. Clearly, this is contradictory to the statement in Step 2 that at least one of $a_{j\alpha}$ equals $\pm 1$. Hence, the inverse bound~\eqref{eq:inverse-bound-N1-distinct} holds by means of proving by contradiction. 
\end{proof}

\subsubsection{Proof of the over-fitted inverse bound for distinct topics} 
\begin{proof}[Proof of Theorem~\ref{thm:inverse-bounds}(a) regarding $\mathfrak{N}_2(G_0)$]
Our goal is to show that for $N = K + K_0 - 1$, it holds that
\begin{equation}\label{eq:inverse-bound-N2-distinct} 
d_{TV}(p^{\mathscr{M}}_{G, N}, p^{\mathscr{M}}_{G_0, N}) \gtrsim W_2^2(G, G_0), 
\end{equation}
as $d_{TV}(p^{\mathscr{M}}_{G, N}, p^{\mathscr{M}}_{G_0, N})\to 0$ for $G \in \Ocal_{K}$. Note that we only need to consider the case $K > K_0$. Otherwise, it follows directly from the previous result as $W_1(G, G_0) \gtrsim W_2^2(G, G_0)$. The proof of this case is also divided into several small steps, many of which are similar to the proof above at a high level but different in the details. To make the proof self-contained and readable, we decide to write all the steps again in details.

\noindent\textbf{Step 1: Proving by contradiction and setup.} Suppose that this bound is not correct. Then there exists a sequence $(G_n) \subset \Ocal_{K}$ such that $d_{TV}(p^{\mathscr{M}}_{G_n, N}, p^{\mathscr{M}}_{G_0, N})\to 0$ and 
\begin{equation}\label{eq:contradict-N2-distinct}
    \lim_{n\to \infty} \dfrac{d_{TV}(p^{\mathscr{M}}_{G_n, N}, p^{\mathscr{M}}_{G_0, N})}{W_2^{2}(G_n, G_0)} = 0.
\end{equation}
Because $\Delta^{V-1}$ is compact, every subsequence of $(G_n)$ has a further subsequence that converges. Let one of those limit to be $G_* \in \Ocal_{K}$. Because $d_{TV}$ is continuous in both of its argument, we have $d_{TV}(p^{\mathscr{M}}_{G_*, N}, p^{\mathscr{M}}_{G_0, N}) = 0$, which combine with the fact that $\mathfrak{N}_{o}(G_0) \leq 2K_0\leq K+K_0-1$, implies $G_* = G_0$. Hence, every subsequence of $(G_n)$, if converges, must go to $G_0$. It is deduced that $G_n \to G_0$ in $W_2$ distance as $n\to \infty$.

Let $G_n = \sum_{j=1}^{K} p_j^{n} \delta_{\theta_j^{n}}\to G_0 = \sum_{j=1}^{K_0} p_j^{0} \delta_{\theta_j^{0}}$. There are three cases: (i) Some atom $\theta_j^{0}$ only has one $\theta_{j'}^{n}$ converges to; (ii) Some atom $\theta_j^{0}$ has several $\theta_{j'}^{n}$'s converges to; (iii) Some weights $p_j^{n}\to 0$ so that the limit of $\theta_{j}^{n}$ does not matter. The first case corresponds to atoms of $G_0$ that are exact-fitted. The second case corresponds to atoms of $G_0$ with several "redundant atoms" in the over-fitted mixing measure $G_n$ converge to. The third case represents some "misplaced masses" that vanish fast. Hence, by extracting a subsequence of $G_n$ if needed, we can assume the following representation:
$$G_n = \sum_{j\in \mathcal{I}_1} p_{j}^{n} \delta_{\theta_j^{n}} + \sum_{j\in \mathcal{I}_2} \sum_{t=1}^{s_j} p_{jt}^{n} \delta_{\theta_{jt}^{n}} + \sum_{j = K_0 + 1}^{\overline{K}} \sum_{t=1}^{s_j} p_{jt}^{n} \delta_{\theta_{jt}^{n}},$$
where $\mathcal{I}_1$ and $\mathcal{I}_2$ is a partition of $[K_0]$ such that 
\begin{align*}
p_j^{n} \to p_j^{0},& \quad \theta_{j}^{n} \to \theta_j^{0}, \quad \forall j\in \mathcal{I}_1, & \quad\text{(exact-fitted components)}, \\
s_j \geq 2, \quad \sum_{t=1}^{s_j} p_{jt}^{n} \to p_j^{0},& \quad \theta_{jt}^{n} \to \theta_j^{0}, \quad \forall t\in [s_j], j\in \mathcal{I}_2, & \quad\text{(redundant components)}, \\
p_{jt}^{n} \to 0, & \quad \theta_{jt}^{n} \to \theta_j^{0},\quad \forall t\in [s_j], j\in [K_0+1, \overline{K}], & \quad\text{(misplaced mass)},
\end{align*}
for $\theta_1^0, \dots, \theta_{K_0}^{0}, \theta_{K_0+1}^{0}, \dots \theta_{\overline{K}}^{0}$ being distinct and $\overline{K} \leq K$. For ease of notation, we also denote $\mathcal{I}_0 = [K_0+1, \overline{K}]$, $p_j^0 = 0$ for all $j\in \mathcal{I}_0$ and $s_j = 1$ for all $j\in \mathcal{I}_1$. Finally, we have that 
$$|\mathcal{I}_1| + |\mathcal{I}_2| = K_0, \quad |\mathcal{I}_1| + 2|\mathcal{I}_2| + |\mathcal{I}_0| \leq K.$$
This implies
$$ 2|\mathcal{I}_1| + 3|\mathcal{I}_2| + |\mathcal{I}_0| \leq K + K_0.$$

\noindent\textbf{Step 2: Taylor expansion around $G_0$.} For every $x\in [V]^{N}$, consider using Taylor expansion up to first order (resp., second order) around each $f_N(x | \theta_j^0)$ for $j\in \mathcal{I}_1$ (resp., $j\in \mathcal{I}_2$). We have
\begin{align*}
    p^{\mathscr{M}}_{G_n, N}(x) - p^{\mathscr{M}}_{G_0, N}(x) & = \sum_{j=1}^{\overline{K}} (p_j^{n} - p_j^0) f_{N}(x | \theta_j^0) + \sum_{j=1}^{\overline{K}} p_j^{n} \left(f_N(x | \theta_j^{n}) - f_N(x | \theta_j^{0}) \right)  \\
    & = \sum_{j=1}^{\overline{K}} (p_j^{n} - p_j^0) f_{N}(x | \theta_j^0) + \sum_{j\in \mathcal{I}_1} \sum_{|\alpha|=1} p_j^{n} (\theta_j^{n} - \theta_j^{0})^{\alpha} \dfrac{\partial}{\partial \theta^{\alpha}} f_N(x | \theta_j^{0})\\ 
    & + \sum_{j\in \mathcal{I}_2} \sum_{t=1}^{s_j} \sum_{1\leq |\alpha|\leq 2}  p_{jt}^{n} (\theta_{jt}^{n} - \theta_j^{0})^{\alpha} \dfrac{\partial}{\partial \theta^{\alpha}} f_N(x | \theta_j^{0}) + R_n(x)\\
    & = \sum_{j\in \mathcal{I}_1} \sum_{|\alpha|\leq 1}a^{n}_{j\alpha}\dfrac{\partial^{|\alpha|}}{\partial \theta^{\alpha}} f_N(x | \theta_j^{0}) + \sum_{j\in \mathcal{I}_2} \sum_{|\alpha|\leq 2} a^{n}_{j\alpha}\dfrac{\partial^{|\alpha|}}{\partial \theta^{\alpha}} f_N(x | \theta_j^{0}) \\
    & + \sum_{j=K_0+1}^{\overline{K}} a_{j0}^{n} f_{N}(x | \theta_j^0) + R_n(x),
\end{align*}
where 
$$a_{j\alpha}^{n} = p_j^{n} (\theta_j^{n} - \theta_{j}^{0})^{\alpha} \forall j\in \mathcal{I}_1, \quad a_{j\alpha}^{n} = \sum_{t=1}^{s_j} p_{jt}^{n} (\theta_{jt}^{n} - \theta_{j}^{0})^{\alpha} \forall j\in \mathcal{I}_2, \quad a_{j\alpha}^{n} = \sum_{t=1}^{s_j} p_{jt}^{n}\forall j\in [K_0+1, \overline{K}],$$
for all appropriate tuple $\alpha\in \Nbb^{V}$ in each case, and
$$R_n(x) = o\left(\sum_{j\in \mathcal{I}_1}\sum_{v=1}^{V} p_j^{n} |\theta_{jv}^{n} - \theta_{jv}^{0}| + \sum_{j\in \mathcal{I}_2} \sum_{t=1}^{s_j} \sum_{v=1}^{V} p_{jt}^{n} |\theta_{jtv}^{n} - \theta_{jv}^{0}|^2\right) = o\left(D_n\right),$$
for all $x\in [V]^{N}$, where
\begin{equation*}
    D_n = \sum_{j=1}^{\overline{K}} \left|\sum_{t=1}^{s_j} p_{jt}^{n} - p_j^0 \right| + \sum_{j\in \mathcal{I}_1} p_{j}^{n} \norm{\theta_{j}^{n} - \theta_{j}^{0}} + \sum_{j\in \mathcal{I}_2} \sum_{t=1}^{s_j} p_{jt}^{n} \norm{\theta_{jt}^{n} - \theta_{j}^{0}}^{2}.
\end{equation*}
From the asymptotic property of Wasserstein distance~\eqref{eq:Wasserstein-equivalent-Voronoi}, we see that $D_n\gtrsim W_2^2(G_n, G_0)$. 

\noindent\textbf{Step 3: Non-vanishing coefficients.} 
For every $v\in [V]$, let $\alpha(v)$ be the tuple in $\Nbb^{V}$ having $v-$th element being 1 and the rest being 0. For every $u, v\in [V]$, let $\alpha(u, v) = \alpha(u) + \alpha(v)$. Hence, $|\alpha(v)| = 1$ and $|\alpha(u, v)| = 2, \forall u, v\in [V]$. So that,
\begin{align*}
\sum_{j,\alpha} |a_{j\alpha}^{n}| & \geq \sum_{j=1}^{\overline{K}} \left|a_{j0}^{n}\right| + \sum_{j\in \mathcal{I}_1} \sum_{v=1}^{V} \left|a_{j\alpha(v)}^{n}\right| + \sum_{j\in \mathcal{I}_2} \sum_{v=1}^{V} \left|a_{j\alpha(v, v)}^{n}\right| \\
& \gtrsim  
\sum_{j=1}^{\overline{K}} \left|\sum_{t=1}^{s_j} p_{jt}^{n} - p_j^0 \right| + \sum_{j\in \mathcal{I}_1} p_{j}^{n} \norm{\theta_{j}^{n} - \theta_{j}^{0}} + \sum_{j\in \mathcal{I}_2} \sum_{t=1}^{s_j} p_{jt}^{n} \norm{\theta_{jt}^{n} - \theta_{j}^{0}}^{2}\\
& = D_n.
\end{align*}
Therefore, we have that $\max_{j, \alpha} |a_{j\alpha}^{n}| / D_n$ is bounded below by a positive constant as $n\to \infty$. Let $m_n = \max_{j, \alpha} |a_{j\alpha}^{n}|$. By compactness argument and extracting a subsequence, if needed, we have
$$\dfrac{a_{j\alpha}^{n}}{m_n} \to a_{j\alpha} \in [-1, 1], \quad \forall j\in [\overline{K}], \alpha\in \Nbb^{V},$$
and at least one of the limit $a_{j\alpha}$ is $\pm 1$. Furthermore, $R_n(x) = o(D_n) = o(m_n)$ for all $x\in [V]^{N}$.

\noindent\textbf{Step 4: Deriving identifiability equation.} From~\eqref{eq:contradict-N2-distinct}, we have
\begin{align*}
    0 & = \lim_{n\to \infty} \dfrac{2d_{TV}(p^{\mathscr{M}}_{G_n, N}, p^{\mathscr{M}}_{G_0, N})}{W_2^2(G_n, G_0)} \\
    & = \lim_{n\to \infty} \dfrac{2d_{TV}(p^{\mathscr{M}}_{G_n, N}, p^{\mathscr{M}}_{G_0, N})}{m_n}\\
    & = \lim_{n\to \infty} \sum_{x\in [V]^{N}} \left|\sum_{j\in \mathcal{I}_1} \sum_{|\alpha|\leq 1} \dfrac{a^{n}_{j\alpha}}{m_n}\dfrac{\partial^{|\alpha|}}{\partial \theta^{\alpha}} f_N(x | \theta_j^{0}) + \sum_{j\in \mathcal{I}_2} \sum_{|\alpha|\leq 2} \dfrac{a^{n}_{j\alpha}}{m_n}\dfrac{\partial^{|\alpha|}}{\partial \theta^{\alpha}} f_N(x | \theta_j^{0})\right. \\
    & \hspace{5cm} \left. + \sum_{j\in \mathcal{I}_0} \dfrac{a_{j0}^{n}}{m_n} f_{N}(x | \theta_j^0) + \dfrac{R_n(x)}{m_n}\right|\\
    &\geq \sum_{x\in [V]^{N}} \left|\lim_{n\to \infty}\sum_{j\in \mathcal{I}_1} \sum_{|\alpha|\leq 1} \dfrac{a^{n}_{j\alpha}}{m_n}\dfrac{\partial^{|\alpha|}}{\partial \theta^{\alpha}} f_N(x | \theta_j^{0}) + \sum_{j\in \mathcal{I}_2} \sum_{|\alpha|\leq 2} \dfrac{a^{n}_{j\alpha}}{m_n}\dfrac{\partial^{|\alpha|}}{\partial \theta^{\alpha}} f_N(x | \theta_j^{0})\right. \\
    & \hspace{5cm} \left. + \sum_{j\in \mathcal{I}_0} \dfrac{a_{j0}^{n}}{m_n} f_{N}(x | \theta_j^0) + \dfrac{R_n(x)}{m_n}\right|\\
    & = \sum_{x\in [V]^{N}}\left|\sum_{j\in \mathcal{I}_1} \sum_{|\alpha|\leq 1} a_{j\alpha}\dfrac{\partial^{|\alpha|}}{\partial \theta^{\alpha}} f_N(x | \theta_j^{0}) + \sum_{j\in \mathcal{I}_2} \sum_{|\alpha|\leq 2} a_{j\alpha}\dfrac{\partial^{|\alpha|}}{\partial \theta^{\alpha}} f_N(x | \theta_j^{0})\right. \\
    & \hspace{5cm} \left. + \sum_{j\in \mathcal{I}_0} a_{j0} f_{N}(x | \theta_j^0)\right|.
\end{align*}
Hence, 
\begin{equation}\label{eq:identifiability-N2-distinct}
    \sum_{j\in \mathcal{I}_1} \sum_{|\alpha|\leq 1} a_{j\alpha}\dfrac{\partial^{|\alpha|}}{\partial \theta^{\alpha}} f_N(x | \theta_j^{0}) + \sum_{j\in \mathcal{I}_2} \sum_{|\alpha|\leq 2} a_{j\alpha}\dfrac{\partial^{|\alpha|}}{\partial \theta^{\alpha}} f_N(x | \theta_j^{0}) + \sum_{j\in \mathcal{I}_0} a_{j0} f_{N}(x | \theta_j^0) = 0,
\end{equation}
for all $x\in [V]^{N}$, with at least one of $a_{j\alpha}$ is different from 0. By marginalizing out some dimensions of $x$, we have that:
\begin{equation}\label{eq:identifiability-N2-distinct-N-prime}
    \sum_{j\in \mathcal{I}_1} \sum_{|\alpha|\leq 1} a_{j\alpha}\dfrac{\partial^{|\alpha|}}{\partial \theta^{\alpha}} f_{N'}(x | \theta_j^{0}) + \sum_{j\in \mathcal{I}_2} \sum_{|\alpha|\leq 2} a_{j\alpha}\dfrac{\partial^{|\alpha|}}{\partial \theta^{\alpha}} f_{N'}(x | \theta_j^{0}) + \sum_{j\in \mathcal{I}_0} a_{j0} f_{N'}(x | \theta_j^0) = 0, 
\end{equation}
for all $N'\leq N$.

\noindent\textbf{Step 5: Deriving a contradiction from the identifiability equation using polynomials manipulation.} For each $\theta \in \Delta^{V-1}$, because
$$f_{N'}(x | \theta) = \prod_{i=1}^{N'} \theta_{x_i}$$
is a monomial of $\theta$ of degree $N'$.
By varying $x\in [V]^{N'}$, we see that $\{f_{N'}(x|\theta): x\in [V]^{N'}\}$ can be viewed as the collection of all monomials of $\theta$ of degree $N'$. The Eq.~\eqref{eq:identifiability-N2-distinct-N-prime} becomes
\begin{equation}\label{eq:identifiability-N2-distinct-poly}
    \sum_{j\in \mathcal{I}_1} \sum_{|\alpha|\leq 1} a_{j\alpha}\dfrac{\partial^{|\alpha|}}{\partial \theta^{\alpha}} P(\theta_j^{0}) + \sum_{j\in \mathcal{I}_2} \sum_{|\alpha|\leq 2} a_{j\alpha}\dfrac{\partial^{|\alpha|}}{\partial \theta^{\alpha}} P(\theta_j^{0}) + \sum_{j\in \mathcal{I}_0} a_{j0} P(\theta_j^0) = 0,
\end{equation}
for all monomials $P(\theta)$ having degree of $N'\leq N$. Because of the additivity of this equation, it also holds for all polynomials $P(\theta)$ having degree of $N'\leq N = K_0 + K - 1$. Now, fix an $i \in [\overline{K}]$, we are going to prove that $a_{i\alpha} = 0$ for all $\alpha$. All 3 cases of $i$ are considered as follows.

\noindent\textbf{Case 0: $i \in \mathcal{I}_0$.}
Because $\theta_1^0,\dots, \theta_{\overline{K}}^0$ are distinct, for every $j \neq i$, there exists $v(j) \in [V]$ such that $\theta_{jv(j)}^0 \neq \theta_{i v(j)}^0$.
Consider the following polynomial of $\theta = (\theta_{1}, \dots, \theta_{V}) \in \Rbb^{V}$:
$$P_{i0}(\theta) = \prod_{\substack{j\neq i \\ j\in \mathcal{I}_0}} (\theta_{v(j)} - \theta^{0}_{j v(j)}) \prod_{j\in \mathcal{I}_1} (\theta_{v(j)} - \theta^{0}_{j v(j)})^{2} \prod_{j\in \mathcal{I}_2} (\theta_{v(j)} - \theta^{0}_{j v(j)})^{3}$$
having the degree of $|\mathcal{I}_0| - 1 + 2|\mathcal{I}_1| + 3|\mathcal{I}_2|\leq K+K_0-1$. Because this polynomial has root $\theta_{j v(j)}$ for all $j\in \mathcal{I}_0\setminus \{i\}$, root $\theta_{j v(j)}$ with multiplicity of 2 for all $j \in \mathcal{I}_1$, and root $\theta_{j v(j)}$ with multiplicity of 3 for all $j \in \mathcal{I}_2$, we have that 
\begin{align*}
    P_{1i}(\theta_{j}^0) & = 0, \quad  \forall j\in \mathcal{I}_0\setminus\{i\},\\
    \dfrac{\partial^{|\alpha|}}{\partial \theta^{\alpha}} P_{1i}(\theta_{j}^0) & = 0, \quad \forall j\in \mathcal{I}_1, |\alpha|\leq 1.\\
    \dfrac{\partial^{|\alpha|}}{\partial \theta^{\alpha}} P_{1i}(\theta_{j}^0) & = 0, \quad \forall j\in \mathcal{I}_2, |\alpha|\leq 2.
\end{align*}
Therefore, by plugging $P = P_{i0}$ into Eq.~\eqref{eq:identifiability-N2-distinct-poly}, only one term that does not vanish, which is $P_{1i}(\theta_{i}^0)$. Hence,
$$0 = a_{i0} P_{1i}(\theta_{i}^0) =  a_{i0} \prod_{\substack{j\neq i \\ j\in \mathcal{I}_0}} (\theta_{iv(j)}^0 - \theta^{0}_{j v(j)}) \prod_{j\in \mathcal{I}_1} (\theta_{iv(j)}^0 - \theta^{0}_{j v(j)})^{2} \prod_{j\in \mathcal{I}_2} (\theta_{iv(j)}^0 - \theta^{0}_{j v(j)})^{3}.$$
Because $(\theta_{iv(j)}^0 - \theta_{j v(j)}^{0})\neq 0$ for all $j\neq i$, we have $a_{i0} = 0$.

\noindent\textbf{Case 1: $i \in \mathcal{I}_1$.}
Again, because $\theta_1^0,\dots, \theta_{\overline{K}}^0$ are distinct, for every $j \neq i$, there exists $v(j) \in [V]$ such that $\theta_{jv(j)}^0 \neq \theta_{i v(j)}^0$.
For every $v\in [N]$, consider the following polynomial of $\theta = (\theta_{1}, \dots, \theta_{V}) \in \Rbb^{V}$:
$$P_{iv}(\theta) = (\theta_{v} - \theta_{iv}^0) \prod_{\substack{j\neq i \\ j\in \mathcal{I}_1}} (\theta_{v(j)} - \theta^{0}_{j v(j)})^2 \prod_{j\in \mathcal{I}_0} (\theta_{v(j)} - \theta^{0}_{j v(j)}) \prod_{j\in \mathcal{I}_2} (\theta_{v(j)} - \theta^{0}_{j v(j)})^{3}$$
having the degree of $2|\mathcal{I}_1| - 1 + |\mathcal{I}_0| + 3|\mathcal{I}_2|\leq K+K_0-1$. Applying the same argument as above, we have $a_{i\alpha(v)} = 0$ for all $v\in [V]$, which means $a_{i\alpha} = 0$ for all $|\alpha|=1$. Next, consider the polynomial 
$$P_{i0}(\theta) = \prod_{\substack{j\neq i \\ j\in \mathcal{I}_1}} (\theta_{v(j)} - \theta^{0}_{j v(j)})^2 \prod_{j\in \mathcal{I}_0} (\theta_{v(j)} - \theta^{0}_{j v(j)}) \prod_{j\in \mathcal{I}_2} (\theta_{v(j)} - \theta^{0}_{j v(j)})^{3}$$
to get $a_{i0} = 0$. Hence, $a_{i\alpha} = 0$ for all $|\alpha|\leq 1$

\noindent\textbf{Case 2: $i \in \mathcal{I}_2$.} Similar to the argument above, we first can consider the polynomial
$$P_{iuv}(\theta) = (\theta_{u} - \theta_{iu}^0)(\theta_{v} - \theta_{iv}^0) \prod_{\substack{j\neq i \\ j\in \mathcal{I}_2}} (\theta_{v(j)} - \theta^{0}_{j v(j)})^3 \prod_{j\in \mathcal{I}_0} (\theta_{v(j)} - \theta^{0}_{j v(j)}) \prod_{j\in \mathcal{I}_1} (\theta_{v(j)} - \theta^{0}_{j v(j)})^{2}$$
To have that $a_{i\alpha(u, v)} = 0$ for all $u, v\in [V]$, which means $a_{i\alpha} = 0$ for all $|\alpha|=2$. Next, consecutively consider
$$P_{iv}(\theta) = (\theta_{v} - \theta_{iv}^0) \prod_{\substack{j\neq i \\ j\in \mathcal{I}_2}} (\theta_{v(j)} - \theta^{0}_{j v(j)})^3 \prod_{j\in \mathcal{I}_0} (\theta_{v(j)} - \theta^{0}_{j v(j)}) \prod_{j\in \mathcal{I}_1} (\theta_{v(j)} - \theta^{0}_{j v(j)})^{2},$$
and
$$P_{i0}(\theta) = \prod_{\substack{j\neq i \\ j\in \mathcal{I}_2}} (\theta_{v(j)} - \theta^{0}_{j v(j)})^3 \prod_{j\in \mathcal{I}_0} (\theta_{v(j)} - \theta^{0}_{j v(j)}) \prod_{j\in \mathcal{I}_1} (\theta_{v(j)} - \theta^{0}_{j v(j)})^{2}$$
to get $a_{i\alpha} = 0$ for all $|\alpha|\leq 2$.

Hence, from all 3 cases, we have $a_{i\alpha} = 0$ for all $i\in [\overline{K}]$ and tuple $\alpha$. Clearly, this is contradictory to the statement in Step 2 that at least one of $a_{i\alpha}$ equals $\pm 1$. Hence, the inverse bound~\eqref{eq:inverse-bound-N2-distinct} holds by means of proving by contradiction. 
\end{proof}

\begin{remark}
    The need to use multi-index tuples and polynomial arguments in the proof of Theorem~\ref{thm:inverse-bounds}(a) is because of the general condition of true topics. Recall that the condition of two topics $\theta_1^0$ and $\theta_2^0$ being distinct only requires them to be different in one dimension. Hence, it requires us to carefully choose the polynomials to prove that all the coefficients of the identifiability equation are equal to 0 in Step 5 of each proof. 
    
    % This step also patches the mistake in the proof of Proposition 1 of \cite{manole2021estimating}. 
\end{remark}

\subsection{Proof of Theorem~\ref{thm:inverse-bounds} (b): Linearly independent topics}
In contrast to part (a), the condition in part (b) is global and geometric: $\theta_1^0, \dots, \theta_{K_0}^0$ are linearly independent. Hence, it will be more useful to consider the tensors' representation of the models and use multi-linear algebra techniques.

\subsubsection{Proof of the exact-fitted inverse bound for linearly independent topics} 
\begin{proof}[Proof of Theorem~\ref{thm:inverse-bounds}(b) regarding $\mathfrak{N}_1(G_0)$.]
We aim to prove that the exact-fitted inverse bound holds for $N=3$ when $\theta_1^0, \dots, \theta_{K_0}^0$ are linearly independent, i.e.
\begin{equation}\label{eq:inverse-bound-N1-independent} 
    d_{TV}(p^{\mathscr{M}}_{G, 3}, p^{\mathscr{M}}_{G_0, 3}) \gtrsim W_1(G, G_0),
\end{equation}
for $d_{TV}(p^{\mathscr{M}}_{G, 3}, p^{\mathscr{M}}_{G_0, 3})\to 0$ and $G\in \Ecal_{K_0}$. 

\noindent\textbf{Step 1: Proving by contradiction and setup.} Suppose that~\eqref{eq:inverse-bound-N1-independent} does not hold, so that there exists a sequence $(G_n)\subset \Ecal_{K_0}$ such that $d_{TV}(p^{\mathscr{M}}_{G_n, 3}, p^{\mathscr{M}}_{G_0, 3})\to 0$ and 
\begin{equation}\label{eq:contradict-N1-independent} 
    \dfrac{d_{TV}(p^{\mathscr{M}}_{G_n, 3}, p^{\mathscr{M}}_{G_0, 3})}{W_1(G_n, G_0)} \to 0,
\end{equation}
as $n\to \infty$. Similar to part (a), because of the compactness of $\Ocal_{K_0}$, we can extract a subsequence of $(G_n)$, which we can re-index to be $(G_n)$ itself, that converges to some $G_*\in \Ocal_{K_0}$. Hence,
$$d_{TV}(p^{\mathscr{M}}_{G_*, 3}, p^{\mathscr{M}}_{G_0, 3}) = \lim_{n\to \infty} d_{TV}(p^{\mathscr{M}}_{G_n, 3}, p^{\mathscr{M}}_{G_0, 3}) = 0,$$ 
which means $p^{\mathscr{M}}_{G_*, 3} = p^{\mathscr{M}}_{G_0, 3}$. Combining with the fact that $\mathfrak{N}_e(G_0)\leq 3$, we have that $G_* = G_0$. Hence, $G_n\to G_0$ in $W_1$. We can have the following representation for $G_n$:
$$G_n = \sum_{j=1}^{K_0} p_j^n \delta_{\theta_j^n},$$
such that $p_j^n \to p_j^0$ and $\theta_j^n \to \theta_j^0$ as $n\to \infty$.

\noindent\textbf{Step 2: Multi-linear expansion.} Recall that we have the representation: 
$$p^{\mathscr{M}}_{G_n, 3} = \sum_{j=1}^{K_0} p_j^{n} (\theta_j^{n})^{\otimes 3} \in \Rbb^{V\times V\times V},$$
as a multiway array. Let $\epsilon_j^n = p_j^{n} - p_j^0\in \Rbb$ and $\delta_n^j = \theta_j^{n} - \theta_j^0\in \Rbb^{V}$ for all $j\in [K_0]$, we have that 
\begin{equation}\label{eq:N1-independent-delta-prop}
\sum_{j=1}^{K_0} \epsilon_j^{n} = 0, \quad \sum_{v=1}^{V} \delta_{jv}^{n} = 0, \quad \forall j\in [K_0],n\in \Nbb,\end{equation}
and $\epsilon_j^{n} \to 0, \delta_{j}^{n}\to 0\in \Rbb^{V}$ as $n\to \infty$. Besides, we have 
\begin{align*}
    2d_{TV}(p^{\mathscr{M}}_{G_n, 3}, p^{\mathscr{M}}_{G_0, 3}) & = \norm{\sum_{j=1}^{K_0} p_j^{n} (\theta_j^{n})^{\otimes 3} - \sum_{j=1}^{K_0} p_j^{0} (\theta_j^{0})^{\otimes 3}}_1\\
    & = \norm{\sum_{j=1}^{K_0} \left(p_j^{n} - p_j^0\right) (\theta_j^{0})^{\otimes 3} + \sum_{j=1}^{K_0} p_j^{n} \left((\theta_j^{0} + \delta_j^{n})^{\otimes 3} - (\theta_j^{0})^{\otimes 3}\right)}_1\\
    & = \left\|\sum_{j=1}^{K_0} \epsilon_j^{n} (\theta_j^{0})^{\otimes 3} + \sum_{j=1}^{K_0} p_j^{n} \left(\Delta_1^{n} + \Delta_2^{n} + \Delta_3^{n}\right)\right\|_1,
\end{align*}
where
$$\Delta_1^{n} = \sum_{j=1}^{K_0} \delta_j^{n} \otimes \theta_j^0 \otimes \theta_j^0 +  \theta_j^0 \otimes \delta_j^{n} \otimes  \theta_j^0 + \theta_j^0 \otimes \theta_j^0  \otimes\delta_j^{n},$$
$$\Delta_2^{n} = \sum_{j=1}^{K_0} \delta_j^{n} \otimes \delta_j^{n} \otimes \theta_j^0 +  \delta_j^{n} \otimes \theta_j^0 \otimes \delta_j^{n}  + \delta_j^{n} \otimes \delta_j^{n} \otimes \theta_j^0,$$
and
$$\Delta_3^{n} = \sum_{j=1}^{K_0} (\delta_j^{n})^{\otimes 3}.$$
\noindent\textbf{Step 3: Non-vanishing coefficients.} Recall that 
$$W_1(G_n, G_0)\asymp \sum_{j=1}^{K_0} \left|\epsilon_j^{n}\right| + p_j^{n} \norm{\delta_j^{n}},$$
which is of order 1 with respect to $\delta_j^{n}$'s. Therefore, $\Delta_2^{n} / W_1(G_n, G_0)\to 0, \Delta_3^{n} / W_1(G_n, G_0)\to 0$ as $n\to \infty$. Besides, let $m_n = \max \left\{\left|\epsilon_j^{n}\right|, \left|\delta_{jv}^{n}\right| \right\}_{j=1}^{K_0}$, we have $m_n \gtrsim W_1(G_n, G_0)$. By extracting a subsequence if needed, we can assume that
$$\dfrac{\epsilon_j^{n}}{m_n} \to \epsilon_j \in [-1, 1], \quad \dfrac{\delta_j^{n}}{m_n} \to \delta_j \in [-1, 1]^{V},$$
and at least one of them is not zero. Moreover, because of~\eqref{eq:N1-independent-delta-prop}, we have that
\begin{equation}\label{eq:N1-independent-delta0-prop}
\sum_{j=1}^{K_0} \epsilon_j = 0, \quad \sum_{v=1}^{V} \delta_{jv} = 0, \quad \forall j\in [K_0].\end{equation}

\noindent\textbf{Step 4: Deriving the identifiability equation.} Putting the limits above together, we have
\begin{align*}
    0 & = \lim_{n\to \infty}\dfrac{2d_{TV}(p^{\mathscr{M}}_{G_n, 3}, p^{\mathscr{M}}_{G_0, 3})}{W_1(G_n, G_0)}\\
    & = \lim_{n\to \infty}\dfrac{2d_{TV}(p^{\mathscr{M}}_{G_n, 3}, p^{\mathscr{M}}_{G_0, 3})}{m_n}\\
    & = \norm{\sum_{j=1}^{K_0} \epsilon_j (\theta_j^0)^{\otimes 3} + \sum_{j=1}^{K_0} p_j^{0}\left(\delta_j \otimes \theta_j^0 \otimes \theta_j^0 + \theta_j^0 \otimes \delta_j \otimes \theta_j^0 + \theta_j^0 \otimes \theta_j^0\otimes \delta_j \right)}_1.
\end{align*}
Hence,
\begin{equation}\label{eq:identifiability-N1-independent}
    \sum_{j=1}^{K_0} \epsilon_j (\theta_j^0)^{\otimes 3} + \sum_{j=1}^{K_0} p_j^{0}\left(\delta_j \otimes \theta_j^0 \otimes \theta_j^0 + \theta_j^0 \otimes \delta_j \otimes \theta_j^0 + \theta_j^0 \otimes \theta_j^0\otimes \delta_j \right) = 0.
\end{equation}
\noindent\textbf{Step 5: Deriving a contradiction.} Because $\theta_1^0, \dots, \theta_{K_0}^0$ are linearly independent, there exists its dual basis, i.e., a set of vectors $\eta_1, \dots, \eta_{K_0}\in \Rbb^{V}$ such that $\innprod{\theta_i^0, \eta_j} = 1_{(i=j)}$ for any $i, j\in [K_0]$. Taking inner product of $\eta_i$ with the first and second dimensions of the LHS of~\eqref{eq:identifiability-N1-independent}, we have
\begin{equation*}
    \sum_{j=1}^{K_0}\epsilon_i (\innprod{\eta_i, \theta_j^0})^2\theta_j^0 + \sum_{j=1}^{K_0} p_j^{0}\left(2\innprod{\eta_i, \delta_j} \innprod{\eta_i, \theta_j^0} \theta_j^0 + (\innprod{\eta_i, \theta_j^0})^2 \delta_j \right)= 0 \in \Rbb^{V},
\end{equation*}
which can be simplified as
\begin{equation*}
    \epsilon_i \theta_i^0 + p_i^0 \left(2\innprod{\eta_i, \delta_i} \theta_i^0 + \delta_i \right) = 0 \in \Rbb^{V},
\end{equation*}
and can also be written as
\begin{equation*}
     \left(\epsilon_i + 2\innprod{\eta_i, \delta_i}p_i^0\right) \theta_{iv}^0 + p_i^0 \delta_{iv} = 0, \text{ for all } v\in[V].
\end{equation*}
Summing over $v$, using the fact that $\theta_i^0\in \Delta^{V-1}$ and $\sum_v \delta_{iv}=0$ by Equation~\eqref{eq:N1-independent-delta0-prop}, we have $\epsilon_i + 2\innprod{\eta_i, \delta_i}p_i^0 = 0$. Plugging this back into the equation above leads to $\delta_i = 0$. This, in turn, implies that $\epsilon_i = 0$. Hence, all the coefficients $\epsilon_i = 0$ and $\delta_i = 0\in \Rbb^{V}$, which is contradictory to the statement at the end of the third step. Hence, the inverse bound~\eqref{eq:inverse-bound-N1-independent} holds by means of proving by contradiction.
\end{proof}

\subsubsection{Proof of the over-fitted inverse bound for linearly independent topics} 
\begin{proof}[Proof of Theorem~\ref{thm:inverse-bounds}(b) regarding $\mathfrak{N}_2(G_0)$.]
We aim to prove that the over-fitted inverse bound holds for $N=4$ when $\theta_1^0, \dots, \theta_{K_0}^0$ are linearly independent, i.e.
\begin{equation}\label{eq:inverse-bound-N2-independent} 
    d_{TV}(p^{\mathscr{M}}_{G, 4}, p^{\mathscr{M}}_{G_0, 4}) \gtrsim W_2^2(G, G_0),
\end{equation}
for $d_{TV}(p^{\mathscr{M}}_{G, 4}, p^{\mathscr{M}}_{G_0, 4})\to 0$, $G\in \Ocal_{K}$ and some $K > K_0$. 

\noindent\textbf{Step 1: Proving by contradiction and setup.} Suppose that~\eqref{eq:inverse-bound-N2-independent} does not hold, so that there exists a sequence $(G_n)\subset \Ecal_{K_0}$ such that $d_{TV}(p^{\mathscr{M}}_{G_n, 4}, p^{\mathscr{M}}_{G_0, 4})\to 0$ and 
\begin{equation}\label{eq:contradict-N2-independent} 
    \dfrac{d_{TV}(p^{\mathscr{M}}_{G_n, 4}, p^{\mathscr{M}}_{G_0, 4})}{W_2^2(G_n, G_0)} \to 0,
\end{equation}
as $n\to \infty$. Because of the compactness of $\Ocal_{K}$, we can extract a subsequence of $(G_n)$, which we can re-index to be $(G_n)$ itself, that converges to some $G_*\in \Ocal_{K}$. Hence,
$$d_{TV}(p^{\mathscr{M}}_{G_*, 4}, p^{\mathscr{M}}_{G_0, 4}) = \lim_{n\to \infty} d_{TV}(p^{\mathscr{M}}_{G_n, 4}, p^{\mathscr{M}}_{G_0, 4}) = 0,$$ 
which means $p^{\mathscr{M}}_{G_*, 4} = p^{\mathscr{M}}_{G_0, 4}$. Combining with the fact that $\mathfrak{N}_o(G_0)\leq 4$, we have that $G_* = G_0$. Hence, $G_n\to G_0$ in $W_2$. We can have the following representation for $G_n$:
$$G_n = \sum_{j=1}^{\overline{K}} \sum_{t=1}^{s_j} p_{jt}^n \delta_{\theta_{jt}^n},$$
such that 
$$\sum_{t=1}^{s_j} p_{jt}^n \to p_j^0 \quad \text{and} \quad \theta_{jt}^n \to \theta_j^0, \quad \forall t\in[s_j], j\in [\overline{K}],$$ 
where $\theta^0_{K_0 + 1}, \dots, \theta^0_{\overline{K}}$ are distinct and different from $(\theta_j^0)_{j=1}^{K_0}$, and $p_j^0 = 0$ for all $j\in [K_0 + 1, \overline{K}]$.

\noindent\textbf{Step 2: Multi-linear expansion.} Recall that we have the representation: 
$$p^{\mathscr{M}}_{G_n, 4} = \sum_{j=1}^{\overline{K}} \sum_{t=1}^{s_j} p_{jt}^{n} (\theta_{jt}^{n})^{\otimes 4} \in \Rbb^{V\times V\times V\times V},$$
as a multiway array. Let $\epsilon_j^n = \sum_{t} p_{jt}^{n} - p_j^0\in \Rbb$ and $\delta_n^{jt} = \theta_{jt}^{n} - \theta_j^0\in \Rbb^{V}$ for all $t, j$. We have that 
\begin{equation}\label{eq:N2-independent-delta-prop}
\sum_{j=1}^{\overline{K}} \epsilon_j^{n} = 0, \quad \sum_{v=1}^{V} \delta_{jtv}^{n} = 0, \quad \forall t\in [s_j], j\in [\overline{K}], n\in \Nbb,\end{equation}
and $\epsilon_j^{n} \to 0, \delta_{jt}^{n}\to 0\in \Rbb^{V}$ as $n\to \infty$. Moreover, $\epsilon_j^{n} \geq 0$ for all $j\in [K_0 + 1, \overline{K}]$. With this notation, we have
\begin{align*}
    2d_{TV}(p^{\mathscr{M}}_{G_n, 4}, p^{\mathscr{M}}_{G_0, 4}) & = \norm{\sum_{j=1}^{\overline{K}} \sum_t p_{jt}^{n} (\theta_{jt}^{n})^{\otimes 4} - \sum_{j=1}^{K_0} p_j^{0} (\theta_j^{0})^{\otimes 4}}_1\\
    & = \left\|\sum_{j=K_0+1}^{\overline{K}}\sum_{t} p_{jt}^{n} \left(\theta_{jt}^{n}\right)^{\otimes 4} + \sum_{j=1}^{K_0} \left[\left(\sum_{t} p_{jt}^{n} - p_j^0\right) (\theta_j^{0})^{\otimes 4} + \right.\right.\\
    & \hspace{4cm} + \left. \left. \sum_{t} p_{jt}^{n} \left((\theta_{j}^{0} + \delta_{jt}^{n})^{\otimes 4} - (\theta_j^{0})^{\otimes 4}\right)\right]\right\|_1\\
    & = \left\|\sum_{j=K_0+1}^{\overline{K}}\sum_{t} p_{jt}^{n} \left(\theta_{jt}^{n}\right)^{\otimes 4} + \sum_{j=1}^{K_0} \epsilon_j^{n} (\theta_j^{0})^{\otimes 4} + \left(\Delta_1^{n} + \Delta_2^{n} + \Delta_3^{n} + \Delta_4^{n}\right)\right\|_1,
\end{align*}
where the $\Delta$ terms can be compactly represented using the transpose function (see Section~\ref{sec:introduction}):
$$\Delta_1^{n} = \sum_{j=1}^{{K}_0} \left(T_{(1,2,3,4)} + T_{(2,1,3,4)} + T_{(3,1,2,4)} + T_{(4,1,2,3)}\right) \left(\left(\sum_{t=1}^{s_j} p_{jt}^{n}\delta_{jt}^{n}\right) \otimes \theta_j^0 \otimes \theta_j^0\otimes \theta_j^0\right),$$
\begin{align*}
\Delta_2^{n} = & \sum_{j=1}^{K_0} \sum_{t=1}^{s_j} p_{jt}^{n} \left(T_{(1,2,3,4)} + T_{(1,3,2,4)} + T_{(1,4,2,3)}\right. \\
& \left.+ T_{(2,3,1,4)} + T_{(2,4,1,3)} +T_{(3,4,1,2)}\right) \left(\delta_{jt}^{n}\otimes \delta_{jt}^{n}\otimes \theta_j^0 \otimes \theta_j^0\right).\end{align*}

$$\Delta_3^{n} = \sum_{j=1}^{K_0} \sum_{t=1}^{s_j} p_{jt}^{n} \left(T_{(1,2,3,4)} + T_{(2,1,3,4)} + T_{(3,1,2,4)} + T_{(4,1,2,3)}\right) \left(\delta_{jt}^{n}\otimes \delta_{jt}^{n}\otimes \delta_{jt}^{n} \otimes \theta_j^0\right).$$
and 
$$\Delta_4^{n} = \sum_{j=1}^{K_0} \sum_{t=1}^{s_j} p_{jt}^{n} \left(\delta_{jt}^{n}\right)^{\otimes 4}.$$

\noindent\textbf{Step 3: Non-vanishing coefficients.} Recall that 
$$W_2^2(G_n, G_0)\asymp \sum_{j=1}^{\overline{K}}\left( \left|\epsilon_j^{n}\right| + \sum_{t=1}^{s_j} p_{jt}^{n} \norm{\delta_{jt}^{n}}^2\right),$$
which is of order 2 with respect to $\delta_j^{n}$'s. Therefore, $\Delta_3^{n} / W_2^2(G_n, G_0)\to 0, \Delta_4^{n} / W_2^2(G_n, G_0)\to 0$ as $n\to \infty$. Besides, let 
$$m_n = \max \left\{\left|\epsilon_j^{n}\right|, \left|\sum_{t} p_{jt}^{n}\left(\theta_{jtv}^{n} - \theta_{jv}^0\right) \right|, p_{jt}^{n}\left(\delta_{jtv}^{n}\right)^{2} \right\}_{j, t, v},$$ 
we have $m_n \gtrsim W_2^2(G_n, G_0)$. By extracting a subsequence if needed, we can assume that
$$\dfrac{\epsilon_j^{n}}{m_n} \to \epsilon_j \in [-1, 1], \quad \dfrac{\sum_{t} p_{jt}^{n}\delta_{jt}^{n}}{m_n} \to \xi_j \in [-1, 1]^{V}, \quad \dfrac{\sqrt{p_{jt}^{n}}\delta_{jt}^{n}}{m_n^{1/2}} \to \delta_{jt} \in [-1, 1]^{V},$$
at least one of them is not zero, and $\epsilon_{j} \geq 0$ for all $j\in [K_0 + 1, \overline{K}]$. Moreover, because of~\eqref{eq:N2-independent-delta-prop}, we have that
\begin{equation}\label{eq:N2-independent-delta0-prop}
\sum_{j=1}^{K_0} \epsilon_j = 0, \quad \sum_{v=1}^{V} \xi_{jv} = 0,\quad \sum_{v=1}^{V} \delta_{jtv} = 0, \quad \forall t\in [s_j], j\in [\overline{K}].\end{equation}

\noindent\textbf{Step 4: Deriving the identifiability equation.} Putting the limits above together, we have
\begin{align*}
    0 & = \lim_{n\to \infty}\dfrac{2d_{TV}(p^{\mathscr{M}}_{G_n, 4}, p^{\mathscr{M}}_{G_0, 4})}{W_2^2(G_n, G_0)}\\
    & = \lim_{n\to \infty}\dfrac{2d_{TV}(p^{\mathscr{M}}_{G_n, 4}, p^{\mathscr{M}}_{G_0, 4})}{m_n}\\
    & = \norm{\sum_{j=K_0+1}^{\overline{K}}\epsilon_j (\theta_j^0)^{\otimes 4} + \sum_{j=1}^{K_0} \left[\epsilon_j (\theta_j^0)^{\otimes 4} + \text{Sym}\left(\xi_{j} \otimes (\theta_j^0)^{\otimes 3}\right) + \sum_{t=1}^{s_j} \text{Sym}\left(\delta_{jt}^{\otimes 2} \otimes (\theta_j^0)^{\otimes 2}\right) \right]}_1,
\end{align*}
where we write for short:
$$\text{Sym}\left(\xi_{j} \otimes (\theta_j^0)^{\otimes 3}\right) = \left(T_{(1,2,3,4)} + T_{(2,1,3,4)} + T_{(3,1,2,4)} + T_{(4,1,2,3)}\right) \left(\xi_j \otimes \left(\theta_j^0\right)^{\otimes 3}\right)$$
and
\begin{align*}\text{Sym}\left(\delta_{jt}^{\otimes 2} \otimes (\theta_j^0)^{\otimes 2}\right) = (T_{(1,2,3,4)} + T_{(1,3,2,4)} + T_{(1,4,2,3)} + T_{(2,3,1,4)} \\ + T_{(2,4,1,3)} +T_{(3,4,1,2)})  \left(\left(\delta_{jt}\right)^{\otimes 2} \otimes \left(\theta_j^0\right)^{\otimes 2}\right)
\end{align*}
Hence,
\begin{equation}\label{eq:identifiability-N2-independent}
    \sum_{j=K_0 + 1}^{\overline{K}} \epsilon_j (\theta_j^0)^{\otimes 4} + \sum_{j=1}^{K_0} \left[\epsilon_j (\theta_j^0)^{\otimes 4} + \text{Sym}\left(\xi_{j} \otimes (\theta_j^0)^{\otimes 3}\right) + \sum_{t=1}^{s_j} \text{Sym}\left(\delta_{jt}^{\otimes 2} \otimes (\theta_j^0)^{\otimes 2}\right)\right]= 0.
\end{equation}
\noindent\textbf{Step 5: Deriving a contradiction.} 
We first show that for all $j\in [K_0+1, K]$, we must have $\theta_{j}^0 \in \Span\{\theta_k^0 \}_{k=1}^{K_0}$. Indeed, for every vector $\eta \perp \Span\{\theta_k^0\}_{k=1}^{K_0}$ (i.e., $\innprod{\eta, \theta_k^0} = 0$ for all $k\in [K_0]$), by taking inner product of $\eta$ with all dimensions of the tensor in the LHS of~\eqref{eq:identifiability-N2-independent}, we have
$$\sum_{j = K_0+1}^{\overline{K}} \epsilon_j \innprod{\theta_j^0, \eta}^{4} = 0,$$
which, together with $\epsilon_j\geq 0$, implies $\innprod{\theta_j^0, \eta} = 0$ as for all $j\in [K_0 + 1, \overline{K}]$. Because this holds for all $\eta \perp \Span\{\theta_k^0\}_{k=1}^{K_0}$, we must have $\theta_{j}^0 \in \Span\{\theta_k^0\}_{k=1}^{K_0}$, for all $j\in [K_0 + 1, \overline{K}]$.

Next, for every vector $\eta \perp \Span\{\theta_k^0\}_{k=1}^{K_0}$, taking inner product of the first and the second dimension of LHS of~\eqref{eq:identifiability-N2-independent} with $\eta$, we are left with
$$\sum_{j=1}^{K_0} \left(\sum_{t}\innprod{\delta_{jt}, \eta}^2\right) (\theta_j^0)^{\otimes 2} = 0.$$
However, the matrices $(\theta_j^0)^{\otimes 2}$, for $j=1, \dots, K_0$, are linear independent because $\{\theta_j^0\}_{j=1}^{K_0}$ are. Therefore, 
$$\sum_{t}\innprod{\delta_{jt}, \eta}^2 = 0, \quad\forall j\in [K_0], $$
which, by the same argument as above, also implies that $\delta_{jt} \in \Span\{\theta_{k}^0\}_{k=1}^{K_0}$ for all $j, t$.

Now, because $\theta_1^0, \dots, \theta_{K_0}^0$ are linearly independent, there exists a dual basis, i.e., a set of vectors $\eta_1, \dots, \eta_{K_0}\in \Rbb^{V}$ such that $\innprod{\theta_i^0, \eta_j} = 1_{(i=j)}$ and $\eta_i \in \Span\{\theta_k^0\}_{k=1}^{K_0}$ for any $i, j\in [K_0]$. For two indices $i_1\neq i_2$, taking inner product of $\eta_{i_1}$ with the first and second dimensions of the LHS of~\eqref{eq:identifiability-N2-independent} and $\eta_{i_2}$ with the third and forth, we have
\begin{equation*}
    \sum_{j=K_0+1}^{\overline{K}} \epsilon_j \innprod{\theta_{j}^0, \eta_{i_1}}^2 \innprod{\theta_{j}^0, \eta_{i_2}}^2 +  \sum_{t=1}^{s_{i_1}} \innprod{\delta_{i_1 t}, \eta_{i_2}}^2 + \sum_{t=1}^{s_{i_2}} \innprod{\delta_{i_2 t}, \eta_{i_1}}^2 = 0.
\end{equation*}
which implies that
\begin{equation}\label{eq:ident-N2-independent-1}
    \epsilon_j \innprod{\theta_{j}^0, \eta_{i_1}}^2 \innprod{\theta_{j}^0, \eta_{i_2}}^2 = 0, \quad\forall j\in [K_0+1, \overline{K}],
    \end{equation}
and
\begin{equation}\label{eq:ident-N2-independent-2}    
    \innprod{\delta_{i_1 t}, \eta_{i_2}} = 0 \quad \forall t,  i_1 \neq i_2.
\end{equation}
Recall the assumption that $\theta_j^0 \not \in \{\theta_k^0\}_{k=1}^{K_0}$  for all $j\in [K_0+1, \overline{K}]$ at the end of Step 1 and all of them have $\ell_1$-norm being one. Therefore, when representing $\theta_j^0$ as a linear combination of $\{\theta_k^0\}_{k=1}^{K_0}$, there exist at least two non-zero coefficients. Hence, for any such $j$, there exist $i_1\neq i_2\in [K_0]$ such that $\innprod{\theta_j^0, \eta_{i_1}}\neq 0$ and $\innprod{\theta_j^0, \eta_{i_2}}\neq 0$. Plugging this into Eq.~\eqref{eq:ident-N2-independent-1}, we have $\epsilon_j = 0 \forall j\in [K_0+1, \overline{K}]$.

Next, from Eq.~\eqref{eq:ident-N2-independent-2}, we have that $\delta_{it} \perp \eta_j$ for all $j\neq i, j\in [K_0]$. Besides, $\delta_{it} \in \Span\{\theta_{j}^0\}_{j=1}^{K_0} = \Span\{\eta_{j}\}_{j=1}^{K_0}$. Therefore, $\delta_{it}$ is a multiple of $\theta_i^{0}$ for all $t, i$. However, the summation over all elements of the former vector is 0, while the latter is 1. Hence, $\delta_{it} = 0$ for all $t, i\in [K_0]$. The Eq.~\eqref{eq:identifiability-N2-independent} finally becomes:
$$\sum_{j=1}^{K_0} \epsilon_j (\theta_j^0)^{\otimes 4} + \text{Sym}\left(\xi_{j} \otimes (\theta_j^0)^{\otimes 3}\right) = 0.$$
Proceeding similarly to Step 5 of the previous proof, we have $\epsilon_j = 0$, $\xi_j = 0\in \Rbb^{V}$ for all $j\in [K_0]$. Hence, all the coefficients of the identifiability Eq.~\eqref{eq:identifiability-N2-independent} are 0, which is contradictory to the statement at the end of the third step. Hence, the inverse bound~\eqref{eq:inverse-bound-N2-independent} holds by means of proving by contradiction.
\end{proof}

\subsection{Proof of Theorem~\ref{thm:inverse-bounds} (c): With the anchor-word assumption}
In this section, we discuss the inverse bounds under the assumption that true topics are anchor-word~\cite{arora2012learning}. For $G_0 = \sum_{i=1}^{K_0} \tilde{\alpha}_i^0 \delta_{\theta_i^0}$, this means that for every $k \in [K_0]$, there exists an "anchor word" $v = v(i)$ such that $\theta_{i v} > 0$ but $\theta_{j v} = 0$ for all $j\in [K_0]\setminus\{i\}$. We are able to show that $\mathfrak{N}_1(G_0) \leq 2$ in the exact-fitted setting but fail to bound $\mathfrak{N}_2(G_0)$ in the over-fitted setting. We suspect that extra conditions on the true topics are required, as described in~\cite{bing2020fast}. 

\subsubsection{Proof of the exact-fitted inverse bound for anchor-word topics} 
\begin{proof}[Proof of Theorem~\ref{thm:inverse-bounds}(c) regarding $\mathfrak{N}_1(G_0)$.]
We aim to show that for $\theta_1^0, \dots, \theta_{K_0}^{0}$ satisfy the anchor-word condition, we have
\begin{equation}\label{eq:inverse-bound-N1-anchor-word}
    d_{TV}(p_{G, 2}, p_{G_0, 2}) \gtrsim W_1(G, G_0),
\end{equation}
as $d_{TV}(p_{G, 2}, p_{G_0, 2})\to 0$ for $G\in \Ecal_{K_0}$. Without the loss of generality, we can re-index $[V]$ so that $j$ is the anchor word of topic $\theta_j^0$ for all $j\in [K_0]$, i.e., $\theta_{jj}^0 > 0$ and $\theta_{jv}^0 = 0$ for all $v\in [K_0]\setminus\{j\}$. Arguing similar to the proof of Theorem~\ref{thm:inverse-bounds}(b), we use the strategy of proving by contradiction and obtain a sequence 
$$G_n = \sum_{j=1}^{K_0} p_j^{n} \delta_{\theta_j^{n}} \subset \Ecal_{K_0} $$
such that $p_j^n \to p_j^0, \theta_j^{n} \to \theta_j^0$ (so that $G_n\to G_0$) and
$$\dfrac{d_{TV}(p_{G_n, 2}, p_{G_0, 2})}{W_1(G_n, G_0)}\to 0.$$
Repeat all the steps in that proof, let $m_n = \max\{|\epsilon^n_j|, |\delta^n_{jv}|\}_{j,v=1}^{K_0, V}$, 
$$\epsilon_j = \lim \dfrac{\epsilon_j^{n}}{m_n}, \quad \delta_{j} = \lim \dfrac{\delta_j^{n}}{m_n}, $$ 
and we arrive at proving the following identifiability equation:
\begin{equation}\label{eq:identifiability-N1-anchor-word}
\sum_{j=1}^{K_0} \epsilon_j (\theta_j^0)^{\otimes 2} + \theta_j^0 \otimes \delta_j + \delta_j \otimes \theta_j^0 = 0,
\end{equation}
which only has trivial solution, where variable $\epsilon_j\in \Rbb, \delta_j\in \Rbb^{V}$ for $j\in [K_0]$ satisfying 
$$\sum_{j=1}^{K_0} \epsilon_j = 0, \quad \sum_{v=1}^{V} \delta_{jv} = 0, \quad \forall j \in [K_0].$$
Notice that because of the anchor-word condition, we have $\delta_{jv}^{n} - \delta_{jv}^{0}\geq 0$ for all $v\in [K_0] \setminus \{j\}$ and $n\in \Nbb$. This implies that $\delta_{jv}\geq 0\forall v\in [K_0] \setminus \{j\}$. 

We start by consider any $\eta \perp \Span\{\theta_1^0, \dots, \theta_{K_0}^0\}$. Multiplying $\eta$ to the left of the LHS of Eq.~\eqref{eq:identifiability-N1-anchor-word}, we have
$$\sum_{j=1}^{K_0} \innprod{\delta_j, \eta} \theta_j = 0.$$
By the linear independence of the true topics, we have $\innprod{\delta_j, \eta} = 0$ for all $j\in [K_0]$ and $\eta \perp \Span\{\theta_1^0, \dots, \theta_{K_0}^0\}$. Hence $\delta_j \in \Span\{\theta_1^0, \dots, \theta_{K_0}^0\}$. Let $\delta_{i} = \sum_{j=1}^{K_0}c_{ij} \theta_{j}^0$. We have $c_{ij}\geq 0$ for all $i\neq j$, because $\delta_{jv}\geq 0\forall v\in [K_0] \setminus \{j\}$ and because of the anchor-word condition. Summing up over the elements of both sides, we also have $\sum_{j} c_{ij} = 0$ for all $i\in [K_0]$. Eq.~\eqref{eq:identifiability-N1-anchor-word} becomes:
\begin{equation}
    \sum_{j=1}^{K_0} \epsilon_j (\theta_j^0)^{\otimes 2} + \sum_{i, j = 1}^{K_0} c_{ij} \theta_i^0 \otimes \theta_j^0 + \sum_{i, j = 1}^{K_0} c_{ij} \theta_j^0 \otimes \theta_i^0 = 0. 
\end{equation}
For a dual-basis $\eta_1, \dots, \eta_{K_0}$ of $\theta_1^0, \dots, \theta_{K_0}^0$, multiply left and right of the above equation to $\eta_i$ and $\eta_j$, respectively, we have
\begin{align*}
    \epsilon_j + 2c_{jj} & = 0, & \forall j\in [K_0]\\
    c_{ij} + c_{ji} & = 0, & \forall i\neq j\in [K_0].
\end{align*}
Because $c_{ij}\geq 0$ for all $i\neq j$, we must have $c_{ij} = 0$. By combining with the fact that $\sum_{j} c_{ij} = 0$, it further implies $c_{jj} = 0$. This, in turn, deduces that $\epsilon_j = 0$ for all $j\in [K_0]$. Hence, the identifiability~\eqref{eq:identifiability-N1-anchor-word} only has trivial solutions, which conclude the inverse bound.
\end{proof}

\section{Proofs of convergence for densities and parameter estimation}\label{sec:proofs-convergence}
The ultimate goal of this section is to show the convergence and posterior contraction rate of density estimation for the LDA model. Then, by combining with the inverse bounds developed above, we arrive at the rate for parameter estimation. To achieve this goal, we apply the theory for M-estimation \cite{Vandegeer, gine2021mathematical} and the frequentist asymptotic theory for Bayesian methods \cite{ghosal2017fundamentals}. The key requirement of both methods is to bound the "size" of the density space of the LDA and mixture models. We will first review the general bounds for the geometry of those models with respect to Hellinger distance, Total Variation distance, and KL divergence. 
\subsection{Geometry of Multinomial models}\label{subsec:geometry-multinom}
Denote the parameter space of topics by $\Psi \subset \Delta^{V-1}$ and assume that $\Psi$ is bounded away from the boundary of $\Delta^{V-1}$, i.e., there exists a positive constant $c_0$ such that $\theta_{v} \geq c_0$ for all $\theta = (\theta_{v})_{v=1}^{V} \in \Psi$ and $v\in [V]$. For $X_1, \dots, X_N\overset{iid}{\sim} \Multi(\theta)$, denote the probability mass function of $(X_1, \dots, X_N)$ to be $p_{\theta, N}$.
\begin{lemma}\label{lem:distance-multinom}
\begin{enumerate} For any two probability vectors $\theta, \theta' \in \Psi$ and positive integer $N$, we have:
    \item[(a)] (Hellinger distance)
    $$d_{H}(p_{\theta, N}, p_{\theta', N})\leq \dfrac{\sqrt{N}}{2\sqrt{2c_0}} \norm{\theta - \theta'};$$
    \item[(b)] (KL divergence) 
    $$d_{KL}(p_{\theta, N}, p_{\theta', N})\leq \dfrac{N}{c_0} \norm{\theta - \theta'}.$$
\end{enumerate}
\end{lemma}
\begin{proof}
\begin{enumerate}
    \item[(a)] For any $\theta = (\theta_v)_{v=1}^{V}\in \Psi$, denote $\sqrt{\theta}:= (\sqrt{\theta_v})_{v=1}^{V}$. We have
    \begin{align*}
        d_H^2(p_{\theta, N}, p_{\theta', N}) & = 1 - \left(\sum_{x_1, \dots, x_N=1}^{V} \sqrt{\theta_{x_1} \cdots \theta_{x_N}} \sqrt{\theta'_{x_1} \cdots \theta'_{x_N}} \right)\\
        & = 1 - \left(\sum_{v=1}^{V} \sqrt{\theta_v \theta_v'} \right)^N\\
        & = 1 - \left(1 - \dfrac{1}{2}\norm{\sqrt{\theta} - \sqrt{\theta'}}\right)^{N}\\
        & \leq \dfrac{N}{2} \norm{\sqrt{\theta} - \sqrt{\theta'}}^2\\
        & = \dfrac{N}{2} \sum_{v=1}^{V} (\sqrt{\theta_v} - \sqrt{\theta_v'})^2 \\
        & = \dfrac{N}{2} \sum_{v=1}^{V} \dfrac{(\theta_v - \theta_v')^2}{(\sqrt{\theta_v} + \sqrt{\theta_v'})^2}\\
        & \leq \dfrac{N}{8c_0} \sum_{v=1}^{V} (\theta_v - \theta_v')^2\\
        & = \dfrac{N}{8 c_0} \norm{\theta - \theta'}^2,
    \end{align*}
    where we use the fact that $(1-\eta)^{N} \geq 1 - N\eta$ for any $\eta\in [0, 1]$ for the first inequality and $\theta_v, \theta_v' \geq c_0$ in the second inequality.
    \item[(b)] For $\theta, \theta' \in \Psi$, we have
    \begin{align*}
        d_{KL}(p_{\theta, N}, p_{\theta', N}) &= Nd_{KL}(p_{\theta, 1}, p_{\theta', 1})\\
        & = N \sum_{v=1}^{V} \theta_{v} (\log \theta_{v} - \log \theta_{v}') \\
        &\leq N \sum_{v=1}^{V} \theta_{v} \left|\log \theta_{v} - \log \theta_{v}'\right| \\
        &\leq N \sum_{v=1}^{V} \theta_{v} \dfrac{\left|\theta_{v} - \theta_{v}'\right|}{c_0} \\
        &\leq \frac{N}{c_0} \norm{\theta} \norm{\theta - \theta'} \\
        &\leq \frac{N}{c_0}\norm{\theta - \theta'}.
    \end{align*}
    where the second inequality is due to an application of the intermediate value theorem to $\log$ function and the facts that $\theta_v, \theta_v' \geq c_0$; the third inequality is because of the Cauchy-Schwarz inequality; the last inequality is because $\norm{\theta}^2 = \sum_v \theta_v^2 \leq \sum \theta_v = 1$. 
\end{enumerate}  

Remark: From the proof of part (a), we note that the upper bound of the Hellinger can be written in terms of $\theta$ without using the lower bound $c_0$ as
\begin{equation}\label{eq:upper_bound_h2}
    d_H^2(p_{\theta,N}, p_{\theta',N}) \leq \frac{N}{2}\norm{\sqrt{\theta} - \sqrt{\theta'}}^2.
\end{equation}
\end{proof}

\subsection{Relevant Results for the Posterior Contraction Theorem}\label{subsec:geometry-mixture-multinom}
In this section, we demonstrate the two crucial ingredients we require in our proofs for the posterior contraction results. To recall, we have i.i.d. data $X^1,\dots,X^m\sim p^{\mathscr{L}}_{G_0,N}$ for $G_0\in\mathcal{E}_{K_0}$ and we place a prior $G\sim\Pi$ on $\mathcal{E}_K$ with $K\geq K_0$ (we refer to the case $K=K_0$ as being \textit{exact-fitted} and $K>K_0$ as \textit{over-fitted}). We denote $\mathcal{O}_K=\cup_{k\leq K} \mathcal{E}_k$ to denote the space of all mixing measures with at most $K$ atoms and $\mathcal{E}^*=\cup_{k\geq 1} \mathcal{E}_k$ to denote the space of all such mixing measures with finite atoms.

The first of these results concerns the prior $\Pi$ and we wish to demonstrate that $\Pi$ places sufficient mass on a small neighborhood of the true data-generating mixing measure $G_0$. This neighborhood is defined in terms of the KL-ball around $p_{G_0,N}$. 
% Before presenting the theorem, we define a regular density as:
% \begin{definition}
%     A probability measure $P$ on $\Delta^{K-1}$ is called regular, if for some constant $c=c(K)$,
%     $$\inf_{\eta^*\in\Delta^{K-1}} P\left(\left\{\eta\in \Delta^{K-1}\mid\norm{\eta-\eta^*}<\epsilon\right\}\right)\geq c\epsilon^{K-1}.$$
% \end{definition}
% If $P$ has a density $p$ with respect to the Lebesgue measure on $\Rbb^{K-1}$ and $p\geq c'$, i.e., $p$ is bounded away from 0, then $P$ is trivially regular. Dirichlet distribution on $\Delta^{K-1}$ parameter $\gamma\in\Rbb^K$ is regular if $\gamma_k\leq 1$ for all $k$. Need to prove this. 

\begin{proposition}[KL-Ball Property of the Prior]\label{prop:kl_ball_prior}
    Consider a prior $\Pi$ such that  for any $G=\sum_k \tilde{\alpha}_k\delta_{\theta_k}$in its support, $\theta_1,\dots,\theta_k$ are uniformly bounded away from the boundary of $\Delta^{V-1}$, i.e., $\min_v \theta_{kv}\geq c_0$ for all $k\in[K]$ for some constant $c_0>0$. Suppose that under $\Pi$, $\theta_k$'s and $\tilde{\alpha}$ are all independent and the induced prior distributions for each $\theta_k$ possess a density with respect to the Lebesgue measure on $\Rbb^{V-1}$ which is bounded away from 0 and $\infty$ on $\{x\in\Delta^{V-1}\mid \min_{v} x_v \geq c_0\}$. Also suppose that the induced distribution on $\tilde{\alpha}$ is 1-regular.
    If $G_0$ is in the support of $\Pi$, then a constant $C$ depending only on $K,V$, we have for any $\epsilon>0$,
    $$\Pi\left(\left\{G\in\mathcal{O}_K \mid d_{KL}(p^{\mathscr{L}}_{G_0,N}, p^{\mathscr{L}}_{G,N}) < \epsilon\right\}\right)\geq C\left(\frac{\epsilon}{N^2}\right)^{KV-1}$$
\end{proposition}
\begin{proof}
Firstly, 
$$\Pi\left(\left\{G\in\mathcal{O}_K \mid d_{KL}(p^{\mathscr{L}}_{G_0,N}, p^{\mathscr{L}}_{G,N}) < \epsilon\right\}\right) \geq \Pi\left(\left\{G\in\mathcal{E}_K \mid d_{KL}(p^{\mathscr{L}}_{G_0,N}, p^{\mathscr{L}}_{G,N}) < \epsilon\right\}\right)$$
    and hence, we focus on $G\in \mathcal{E}_K$. By Proposition \ref{prop:metric-equivalent}, we have
    $$d_{KL}(p^{\mathscr{L}}_{G_0,N}, p^{\mathscr{L}}_{G,N})\leq N d_{KL}(p_{G_0,N}^{\mathscr{M}}, p_{G,N}^{\mathscr{M}}).$$
    Now, suppose $q$ is a coupling between $\tilde{\alpha}_0$ and $\tilde{\alpha}$, i.e. $\sum_{k\in [K_0]} q_{k,k'} = \tilde{\alpha}_{k'}$ and $\sum_{k'\in [K]} q_{k,k'} = \tilde{\alpha}_{k}^0$. Recall that we denote $p_{\theta, N}$ by the density of $x_1,\dots,x_N\sim \Multi(\theta)$ (i.i.d.) we can use the convexity of KL-divergence to upper bound the KL between two such mixture of product multinomial distributions:
    \begin{align*}
        d_{KL}(p_{G_0,N}^{\mathscr{M}}, p_{G,N}^{\mathscr{M}}) &= d_{KL}\left(\sum_{k,k'} q_{k,k'} p_{\theta_k^0,N}, \sum_{k,k'} q_{k,k'} p_{\theta_{k'},N}\right) \leq \sum_{k,k'} q_{k,k'}d_{KL}(p_{\theta_k^0,N}, p_{\theta_{k'},N}).
    \end{align*}
    Now, for any $\theta_0,\theta\in\Delta^{V-1}$ such that $\min_v \theta_v\geq c_0$, we have $d_{KL}(p_{\theta_0, N}, p_{\theta, N}) \leq \dfrac{N}{c_0} \norm{\theta_0 - \theta}$.  Plugging this upper bound into the previous display and minimizing over $q$, we obtain that $d_{KL}(p_{G_0,N}^{\mathscr{M}}, p_{G,N}^{\mathscr{M}})\leq N W_1(G, G_0)/c_0$. Thus, together with the first display, we obtain
    $$d_{KL}(p^{\mathscr{L}}_{G_0,N}, p^{\mathscr{L}}_{G,N}) \leq \frac{N^2}{c_0}W_1(G, G_0).$$

    Now, consider any $G_0\in\mathcal{E}_{K_0}$ in the support of $\Pi$, which is a distribution on $\mathcal{E}_K$ with $K\geq K_0$. For $\tilde{\alpha}_0\in\Delta^{K_0-1}$, denote $(\tilde{\alpha}_0,0)\in\Delta^{K-1}$ to be the probability vector of length $K$, obtained by placing $K-K_0$ zeros after $\tilde{\alpha_0}$. Consider the set $\mathcal{G}\subset \mathcal{E}_K$ defined as
    $$\mathcal{G}=\left\{G=\sum_{k\in[K]} \tilde{\alpha}_k\delta_{\theta_k}\middle\vert\, \norm{\tilde{\alpha}-(\tilde{\alpha}_0,0)}<\epsilon; \norm{\theta_k-\theta_k^0}<\epsilon, k\in [K_0]; \, \norm{\theta_k-\theta_1^0}<\epsilon, K_0< k\leq K\right\}.$$
    For $K=K_0$, we can ignore the last condition in the set. Firstly, we note that given the conditions on $\Pi$, we have for some constant $C_1$ depending on $K, V$ only,
    \begin{align*}
        \Pi(\mathcal{G}) &= \Pi\left(\norm{\tilde{\alpha}-(\tilde{\alpha}_0,0)}<\epsilon\right) \times\prod_{k\in[K_0]} \Pi\left(\norm{\theta_k-\theta_k^0}<\epsilon\right) \times \prod_{k=K_0+1}^{K} \Pi\left(\norm{\theta_k-\theta_1^0}<\epsilon\right)\\
        &\geq C \epsilon^{K-1} \epsilon^{K(V-1)} = C_1\epsilon^{K(V-1)+K-1}.
    \end{align*}
    Since the true topics $\theta_1^0,\dots,\theta_K^0$ are bounded away from the boundary of $\Delta^{V-1}$ by $c_0$, for sufficiently small $\epsilon$, for any $G\in \mathcal{G}$, the associated topics are also bounded away from the boundary by say $c_0/2$. 
    Now, for any $G\in\mathcal{G}$, we have $W_1(G, G_0)\leq \epsilon$. This follows from \cite{wei2022convergence} Lemma 3.2(b). Thus, we have 
    $$\Pi(d_{KL}(p^{\mathscr{L}}_{G_0,N}, p^{\mathscr{L}}_{G,N}) < \epsilon) \geq  \Pi(W_1(G,G_0)<c_0\epsilon/2N^2) \geq \Pi(\mathcal{G}) \geq C_2 \left(\frac{\epsilon}{N^2}\right)^{KV-1}$$
    where $C_2$ is another constant only depending on $K, K_0, V$. 
    
    % \DD{Sunrit: can you help to check if $C_2$ needs to be depend on $K$ and $V$? And also $C_3$ // Sunrit: The dependence on $K,V$ of $C_2$ comes from $C_1$, which only depends on the assumptions on prior $\Pi$ - I think we can remove dependence of $K,V$ from it. For $C_3$ below, it is free of $V$ it seems, but depends on $K$ due to the first term}
    % Remark: The constant can be much improved by taking $\mathcal{G}$ to be union over all possible such pairings, however it does not affect the rate with respect to $\epsilon$, note that $K,V$ are fixed constants in the analysis.
\end{proof}

The second crucial property concerns the entropy of the model class. Consider $\mathcal{G}_K = \{p^{\mathscr{L}}_{G,N} \mid G\in \mathcal{E}_K, G \text{ is in the support of } \Pi\}$ and $\bar{\mathcal{G}}_K = \cup_{k\leq K} \mathcal{G}_k$, where $\Pi$ satisfies the conditions in the preceding proposition. Note $\bar{\mathcal{G}}_K\subseteq \bar{\mathcal{G}}_{K'}$ for $K'\geq K$. We upper bound the Hellinger entropy of this space in the following result. We denote $N(\epsilon, \mathcal{P}, d_{H})$ to be the $\epsilon-$covering number of $\mathcal{P}$ (a class of probability models over some space) with respect to the Hellinger metric.

\begin{proposition}[Model Entropy]\label{prop:entropy_model}
    For some constant $C_3$ depending on $K,V$, we have $N(\epsilon, \bar{\mathcal{G}}_K, d_{H}) \leq C_3(N/\epsilon^2)^{K-1+K(V-1)/2}$.
\end{proposition}

% \begin{proof}
%     For any $G, G'\in \mathcal{E}_K$ in the support of $\Pi$, we have
%     \begin{align*}
%         d_H^2(p_{G,N}, p_{G',N}) &\leq d_{KL}(p_{G,N}, p_{G',N}) \\
%         &\leq \frac{N^2}{c_0}W_1(G, G') \\
%         &\leq c N^2 \min_{\tau} \sum_{k\in[K]} \left(\norm{\theta_k - \theta_{\tau(k)}'} + |\tilde{\alpha}_k - \tilde{\alpha}_{\tau(k)}'|\right)
%     \end{align*}
%     where the minimum in the last line is over permutations $\tau$ on $[K]$ - this last inequality is due to Lemma 3.2(b) from \cite{wei2022convergence}. This implies that if we have an $\epsilon/2K-$ cover for $\Delta^{V-1}$ in the $\ell_2$ norm (for each of the $\theta$'s) and a $\epsilon/2$-cover for $\Delta^{K-1}$ in the $\ell_1$ norm (for $\tilde{\alpha}$), then we can construct a $\sqrt{cN^2\epsilon}$-cover for $\mathcal{G}_K$ in the $d_H$ metric. Thus,
%     $$N(\epsilon, \mathcal{G}_K, d_H) \leq N\left(\frac{\epsilon^2}{2cN^2 K}, \Delta^{V-1}, \norm{\cdot}_2\right)^K \times N\left(\frac{\epsilon^2}{2cN^2}, \Delta^{K-1}, \norm{\cdot}_1\right) \leq C_3 \left(\frac{N^2}{\epsilon^2}\right)^{KV-1}$$
%     where we used covering number bound for the probability simplex under $\ell_1$ norm and covering number bounds for compact subsets of $\Rbb^V$ under $\ell_2$ norm from \cite{ghosal2017fundamentals} Appendix.
% \end{proof}
\begin{proof}
    Noting that $d_H^2$ is a $f-$divergence with $f(t)=\frac{1}{2}(\sqrt{t}-1)^2$, we can use the joint convexity of $f-$divergences. Given $G,G'\in \bar{\mathcal{G}}_K$, for any coupling $q=(q_{k,k'})_{k,k'}$ of the associated $\tilde{\alpha}$ and $\tilde{\alpha}'$, we have (using Equation \eqref{eq:L_leq_M} with $C_1=C_1(N,\bar{\alpha})$ )
    \begin{align*}
    d_H^2(p^{\mathscr{L}}_{G,N}, p^{\mathscr{L}}_{G',N}) &\leq C_{1}^2 d_H^2(p^{\mathscr{M}}_{\tilde{\alpha}, \Theta}, p^{\mathscr{M}}_{\tilde{\alpha}', \Theta'}) \\
        &=C_1 d_H^2\left(\sum_k \tilde{\alpha}_k p_{\theta_k,N}, \sum_{k'} \tilde{\alpha}_{k'}' p_{\theta_k',N}\right)\\ &\leq C_1^2 d_{H}^2\left(\sum_{k,k'} q_{k,k'}p_{\theta_k,N}, \sum_{k,k'} q_{k,k'}p_{\theta_{k'}',N}\right) \\
        &\leq C_1^2\sum_{k,k'} q_{k,k'}d_H^2\left(p_{\theta_k,N}, p_{\theta_{k'}',N}\right) \\
        &\leq \frac{C_1^2 N}{2}\sum_{k,k'} q_{k,k'} \norm{\sqrt{\theta_k}-\sqrt{\theta_{k'}'}}^2
    \end{align*}
    where the last inequality uses Equation \eqref{eq:upper_bound_h2}. Minimizing over all possible couplings $q$, we obtain the following upper bound for the squared-Hellinger distance
    \begin{equation}
        d_H^2\left(p^{\mathscr{L}}_{G,N}, p^{\mathscr{L}}_{G',N}\right) \leq \frac{C_1^2 N}{2}W_2^2\left(\sum_k \tilde{\alpha}_k \delta_{\sqrt{\theta_k}}, \sum_{k'} \tilde{\alpha}_{k'}'\delta_{\sqrt{\theta_{k'}'}}\right).
    \end{equation}

    Consider the space $S_+=\{\sqrt{\theta} \in\mathbb{R}^V: \theta\in\Delta^{V-1}\}$. Note that for any $w\in S_+$, $\norm{w}=1$ (since $\theta\in\Delta^{V-1}$) and all coordinates are positive $w_v\geq 0$ for all $v\in[V]$. Thus, $S_+$ can be associated as the positive part of the surface of the unit sphere in $V-$dimensions. Furthermore, for $\theta,\theta'\in\Delta^{V-1}$, let $w=\sqrt{\theta}\in S_+, w'=\sqrt{\theta'}\in S_+$, then $d(\theta, \theta') = \sqrt{\sum_v (w_v-w_v')^2} = \norm{w - w'}$ is the usual Euclidean distance between the corresponding $w,w'$ in $S_+\subset \mathbb{R}^V$. Using this, the above display reduces to 
    \begin{equation}
        d_H^2(p^{\mathscr{L}}_{G,N}, p^{\mathscr{L}}_{G',N}) \leq \frac{C_1^2 N}{2} W_2^2 \left(\sum_k \tilde{\alpha}_k \delta_{w_k}, \sum_{k'}\tilde{\alpha}_{k'}'\delta_{w_{k'}'} \right)
    \end{equation}
    where $G=\sum_k \tilde{\alpha}_k \delta_{\theta_k}$ and $w_k=\sqrt{\theta}_k$ (and similarly for $G'$). In this case, now, with $w_k, w_k'\in S_+$, the $2-$Wasserstein distance is with respect to the usual Euclidean distance on $S_+$. Note that $G$ and $G'$ could have different number of atoms. For $G=\sum_k \tilde{\alpha}_k\delta_{\theta_k}$, we denote $H_G = \sum_k \tilde{\alpha}_k \delta_{w_k}$, with $w_k=\sqrt{\theta}_k$, which is a discrete measure with same number of atoms as $K$ supported on $S_+$. Thus, we have 
    $$d_H^2(p^{\mathscr{L}}_{G,N}, p^{\mathscr{L}}_{G',N}) \leq \frac{C_1^2 N}{2} W_2^2(H_G, H_{G'}).$$
    
    Suppose we have an $\epsilon/2-$cover for $S_+$ in $\ell_2$ metric (for $w$), say $\mathcal{A}_1$, and a $\epsilon^2/8-$cover for $\Delta^{K-1}$ in the $\ell_1$ metric (for $\tilde{\alpha}$), say $\mathcal{A}_2$. Now, given any $G\in\mathcal{O}_K$ such that $p_{G,N}\in\bar{\mathcal{G}}_K$, suppose $G=\sum_{k\in[K']} \tilde{\alpha}_{k} \delta_{\theta_k}$ for some $K'\leq K$, we can express it as $G=\sum_{k\in[K]} \tilde{\alpha}_k\delta_{\theta_k}$ where $\tilde{\alpha}_k=0$ for $k=K'+1,\dots, K$ and $\theta_k = \tilde{\theta}_1$ for $k=K'+1,\dots,K$ (note that this embeds $G\in \mathcal{O}_K$ into $\mathcal{E}_K$). Write $w_k=\sqrt{\theta}_k$ for $k\in[K]$. Given the covers mentioned earlier, it is possible to get $\tilde{w}_k\in\mathcal{A}_1$ such that $\norm{w_k - \tilde{w}_k}_2\leq \epsilon/2$ for $k\in[K]$ (if fact, $w_k=\tilde{w}_1$ for $k>K'$). Similarly, we can also get $a\in \mathcal{A}_2$ such that $\norm{\tilde{\alpha} - a}_1\leq \epsilon^2/8$. Now, consider $\tilde{H} = \sum_{k\in[K]} a_k\delta_{\tilde{w}_k}$ and $H'=\sum_{k\in[K]} a_k \delta_{w_k}$. Now,
    \begin{align*}
        W_2(H_G, \tilde{H}) &\leq W_2(H_G, H') + W_2(H', \tilde{H}) \\
        &\leq \sqrt{\sum_{k\in[K]} |\tilde{\alpha}_k - a_k|}\text{diam}(S_+) + \sqrt{\sum_{k\in[K]} a_k \norm{w_k - \tilde{w}_k}_2^2} \\
        &\leq \sqrt{2}\frac{\epsilon}{2\sqrt{2}} + \frac{\epsilon}{2} = \epsilon.
    \end{align*}
    noting that $\sum_k a_k = 1$ and $\text{diam}(S_+)=\sqrt{2}$. Thus, for such $\tilde{H}$ constructed from the covers, we have $W_2^2(H_G, \tilde{H}) \leq \epsilon^2$. Let $\mathcal{A}=\{\sum_{k\in[K]} a_k\delta_{\tilde{w}_k^2}: a\in\mathcal{A}_2, \tilde{w}_2,\dots,\tilde{w}_K\in\mathcal{A}_1\}\subset \mathcal{E}_K\subset\mathcal{O}_K$ (note for any $w\in S_+$, $w^2\in\Delta^{V-1}$). Then, $|\mathcal{A}|\leq |\mathcal{A}_1|^K |\mathcal{A}_2|$. The above discussion shows that given any $G\in\mathcal{O}_K$, there exists $\tilde{G}\in \mathcal{A}$ such that $d_H^2(p^{\mathscr{L}}_{G,N}, p^{\mathscr{L}}_{\tilde{G},N})\leq C_1^2 N \epsilon^2 / 2$. Thus, $\mathcal{A}$ is a $\delta-$cover for $\bar{\mathcal{G}}_K$ in $d_H$ metric, where $\delta=C_1 \epsilon \sqrt{N/2}$. This leads to
    \begin{equation}
        N(\epsilon, \bar{\mathcal{G}}_K, d_H) \leq N\left(\frac{\epsilon}{C_1\sqrt{2N}}, S_+, \norm{\cdot}_2\right)^K\times N\left(\frac{\epsilon^2}{4C_1^2 N}, \Delta^{K-1}, \norm{\cdot}_1\right).
    \end{equation}
    The second term in the product on the right side is upper bounded as
    $$N\left(\frac{\epsilon^2}{4C_1^2 N}, \Delta^{K-1}, \norm{\cdot}_1\right) \lesssim \left(\frac{N}{\epsilon^2}\right)^{K-1}.$$
    For the first term, let $M=N(\delta, S_+, \norm{\cdot}_2)$ for $\delta<2$. 
    % Trivially, since $S_+\subset B_1$ (where $B_\delta$ is a Euclidean ball in $\mathbb{R}^V$ of radius $\delta$), we have $M\leq N(\delta, B_1, \norm{\cdot}_2) \lesssim (1/\delta)^V$. {\color{red}{Can we do better?-dim of $S_+$ is $V-1$}} 
    If $w_1,\dots,w_M$ is a cover, then by comparing volumes, we get
    $$M \text{vol}(B_{\delta/2}) \leq \text{vol}(S_+ + B_{\delta/2}),$$
     where the Minkowski sum $S+B_\delta\subset\{w:1-\delta \leq \norm{w}_2 \leq 1+\delta\}$. So, letting $\mathcal{V}=\text{vol}(B_1)$, we have $\text{vol}(B_{\delta/2})=(\delta/2)^V \mathcal{V}$ and $\text{vol}(S_++B_{\delta/2}) < ((1+\delta/2)^V - (1-\delta/2)^V)\mathcal{V}$, which implies
    \begin{equation}
        N(\delta, S_+, \norm{\cdot}_2) < \frac{(2+\delta)^V - (2-\delta)^V}{\delta^V} \lesssim (1/\delta)^{V-1}.
    \end{equation}
    Thus, combining the two cases we have
    \begin{equation}
        N(\epsilon, \bar{\mathcal{G}}_K, d_H) \lesssim \left(\frac{N}{\epsilon^2}\right)^{K-1 + K(V-1)/2}.
    \end{equation}
    
    % Thanks to Lemma 3.2(b) from \cite{wei2022convergence}, we can further upper bound this as
    % \begin{equation}
    %     d_H^2(p_{G,N}, p_{G',N}) \leq C_2^2\min_{\tau} \sum_k \left(\norm{w_{\tau(k)} - w_k'}_2^2 + \left|\tilde{\alpha}_{\tau(k)} - \tilde{\alpha}_k'\right|\right)
    % \end{equation}
    % where $C_2^2 = C_1^2 N/2$ and the minimum is over all permutations $\tau$ of $[K]$ (note that $\text{diam}(S_+)=\sqrt{2}$, hence $\text{diam}(S_+)^2/2 = 2$, resulting in no additional multiplicative constant). Thus, if we have an $\epsilon/\sqrt{2K}-$cover for $S_+$ in $\ell_2$ metric (for $w$) and a $\epsilon^2/K-$cover for $\Delta^{K-1}$ in the $\ell_1$ metric (for $\tilde{\alpha}$), then we can construct a $C_2\epsilon-$cover for $\bar{\mathcal{G}}_K = \mathcal{G}_K$ in the $d_H$ metric. Hence, combining all these, we get
    % \begin{equation}
    %     N(\epsilon, \bar{\mathcal{G}}_K, d_H) \leq N\left(\frac{\epsilon}{C_2\sqrt{2K}}, S_+, \norm{\cdot}_2\right)^K\times N\left(\frac{\epsilon^2}{C_2^2 K}, \Delta^{K-1}, \norm{\cdot}_1\right).
    % \end{equation}
\end{proof}
\subsection{Proof of Theorem~\ref{theorem:contraction_posterior}}\label{appD-proof}

\noindent\textbf{Proof of Part (a)}
\begin{proof}
    We verify the conditions in Theorem 8.11 in \cite{ghosal2017fundamentals}. In the notations of the theorem, $\mathcal{P} = \mathcal{P}_{m,1} =  \mathcal{G}_K$ with $K\geq K_0$ and  $\mathcal{P}_{m,2}=\emptyset$, and $\Pi_m=\Pi$ (the fixed chosen prior). Taking $\epsilon_{m,N} = \tilde{C}\sqrt{\log (mN) / m}$ for some suitably large constant $\tilde{C}$ (possibly depending on $K,V$, to be chosen during the verification steps), we verify conditions (i) and (ii) (note that condition (iii) is trivially true). 

    For condition (i), we have 
    \begin{align*}
        \frac{\Pi(G : j\epsilon_m < d_H(p^{\mathscr{L}}_{G,N}, p^{\mathscr{L}}_{G_0,N})\leq 2 j \epsilon_m)}{\Pi(G\, : \, d_{KL}(p^{\mathscr{L}}_{G_0,N}, p^{\mathscr{L}}_{G,N}) < \epsilon_{m,N}^2)} &\leq \frac{1}{\Pi(G\, : \, d_{KL}(p^{\mathscr{L}}_{G_0,N}, p^{\mathscr{L}}_{G,N}) < \epsilon_{m,N}^2)} \\
        &\leq \frac{1}{C}\epsilon_{m,N}^{-2(KV-1)} N^{2(KV-1)} \\
        &= \exp\left\{-\log C - 2(KV-1)\left(\log \epsilon_{m,N} - \log N\right) \right\} \\
        &\leq \exp\left\{ m\epsilon_{m,N}^2/16\right\}
    \end{align*}
    where the first inequality follows by using the trivial upper bound of 1 for the numerator, the second inequality utilizes Proposition \ref{prop:kl_ball_prior}, and the last inequality follows by choosing a sufficiently large constant $\tilde{C}$ (which only depends on $K, V$). Thus, we have 
    $$\frac{\Pi(G : j\epsilon_m < d_H(p^{\mathscr{L}}_{G,N}, p^{\mathscr{L}}_{G_0,N})\leq 2 j \epsilon_m)}{\Pi(G\, : \, d_{KL}(p^{\mathscr{L}}_{G_0,N}, p^{\mathscr{L}}_{G,N}) < \epsilon_{m,N}^2)} \leq \exp\left\{ m\epsilon_{m,N}^2 j^2/16\right\}$$
    for all $j\geq 1$.  
    % \DD{Sunrit: Can you help to check this detail? I think the rate is still $(\log(mN)/m)^{1/2}$. I will be very happy if we can remove the dependence of $K$ and $V$ of the constant in this theorem. I think it is possible from the calculation in Proposition 8 and 9. Sunrit: I think it is free of $V$, but probably $K$ is still there due to its requirement in entropy}

    To verify condition (ii), we have
    \begin{align*}
        \sup_{\epsilon\geq \epsilon_{m,N}} \log N&\left(\frac{1}{2}\epsilon, \{p^{\mathscr{L}}_{G,N}\in \mathcal{P}_{m,1}\mid d_H(p_{G,N}, p_{G_0,N})\leq 2\epsilon\}, d_H\right) \\
        &\leq \log N\left(\frac{\epsilon_{m,N}}{2}, \mathcal{G}_K, d_H \right) \\
        &\leq \log \left(C_3\left(\frac{2\sqrt{N}}{\epsilon_{m,N}}\right)^{2(K-1)+K(V-1)}\right) \\
        &= \log \tilde{C}_3 - (KV+K-2)\left(\log \epsilon_{m,N} - (\log N)/2\right) \\
        &\leq m\epsilon_{m,N}^2
    \end{align*}
    where the first inequality follows by using the whole space $\mathcal{G}_K$, instead of the restriction; the second inequality uses Proposition \ref{prop:entropy_model} and the last line again follows by similarly choosing a sufficiently large $\tilde{C}$ (depending on $K,V$). Thus, for a single sufficiently large constant $\tilde{C}$ depending on $K,V$, both conditions (i) and (ii) are satisfied by taking $\epsilon_{m,N}=\tilde{C}\sqrt{\log (mN) / m}$.

    Thus, according to Theorem 8.11 from \cite{ghosal2017fundamentals}, we have for any sequence $M_m\to\infty$,
    $$\Pi\left(G\in\mathcal{E}_K\,:\, d_H(p^{\mathscr{L}}_{G,N}, p^{\mathscr{L}}_{G_0,N})\geq M_m \epsilon_{m,N} \mid X^1_{[N]},\dots,X^m_{[N]}\right) \to 0$$
    in $P_{G_0,N}^\infty$-probability, as $m\to\infty$. Thus, the density contraction rate is $\left(\frac{\log (mN)}{m}\right)^{1/2}$ with respect to $d_H$, up to a constant in terms of $K, V$. 
    
\end{proof}

\noindent\textbf{Proof of Part (b)}

\begin{proof}
Focusing on the exact-fitted case with $K=K_0$, from Section \ref{sec:finer_identifiability}, we know that as long as $N\geq \mathfrak{N}_1(G_0)$, we have $d_{TV}(p^{\mathscr{L}}_{G_0, N}, p^{\mathscr{L}}_{G, N}) \geq C(G_0,K_0) W_1(G, G_0)$. Now, we know that for any two probability densities $p,q$, $d_{H}(p,q) \geq d_{TV}(p,q)/2$. Thus, 
$$\{W_1(G,G_0)\geq 2\epsilon/C(G_0,K_0)\}\Rightarrow \{d_{TV}(p^{\mathscr{L}}_{G,N}, p^{\mathscr{L}}_{G_0,N})\geq 2\epsilon\} \Rightarrow \{d_{H}(p^{\mathscr{L}}_{G,N}, p^{\mathscr{L}}_{G_0,N})\geq \epsilon\},$$
which combined with the conclusion in part (a) immediately gives
\begin{align*}
    \Pi&\left(G\in\mathcal{E}_K\,:\, W_1(G,G_0)\geq 2M_m \epsilon_{m,N}/C(G_0,K_0) \mid X^1_{[N]},\dots,X^m_{[N]}\right)\\
    &\leq \Pi\left(G\in\mathcal{E}_K\,:\, d_H(p^{\mathscr{L}}_{G,N},p^{\mathscr{L}}_{G_0,N})\geq M_m \epsilon_{m,N} \mid X^1_{[N]},\dots,X^m_{[N]}\right) \to 0
\end{align*}
in $(P^{\mathscr{L}}_{G_0,N})^m$-probability as $m\to\infty$. By absorbing the constant $2/C(G_0,K_0)$ into the constant $\tilde{C}$ already present in $\epsilon_{m,N}$, we conclude that for any sequence $M_m\to\infty$,
$$\Pi\left(G\in\mathcal{E}_K\,:\, W_1(G,G_0)\geq M_m \epsilon_{m,N} \mid X^1_{[N]},\dots,X^m_{[N]}\right)\to 0$$
in $(P^{\mathscr{L}}_{G_0,N})^m$-probability as $m\to\infty$ as long as $N\geq \mathfrak{N}_1(G_0)$. Thus, the contraction rate in this case is $\left(\frac{\log (mN)}{m}\right)^{1/2}$ with respect to $W_1$, up to a constant in terms of $K, V$. 

% By Theorem \ref{thm:inverse-bounds}, it follows that the result holds for any fixed $N\geq 3$ (if the true topics $\theta_1^0,\dots,\theta_{K_0}^0$ are linearly independent) or $N\geq 2K_0-1$ in the general case. We note that the theorem is valid and useful even in the case where $N=N(m)$ is increasing as $m\to\infty$, as long as $\log(mN)/m\to 0$. 

% Remark: It seems that increasing $N$ deteriorates the contraction rate for the mixing measure $G_0$ - this is possibly because of the particular proof technique. However, we note that this result suggests that even with fixed $N$, a fully Bayesian procedure recovers the true topics and Dirichlet normalized parameters at the parametric rate $1/\sqrt{m}$ (up to logarithmic terms) asymptotically. This is in contrast to the result in \cite{nguyen2015} which required both $m,N\to\infty$.
\end{proof}

\noindent\textbf{Proof of part (c)}

\begin{proof}
In the over-fitted case $K>K_0$, from the results in Section \ref{sec:finer_identifiability}, we know that as long as $N\geq \mathfrak{N}_2(G_0,K)$, we have $d_{TV}(p^{\mathscr{L}}_{G_0, N}, p^{\mathscr{L}}_{G, N}) \geq C(G_0,K) W_2^2(G, G_0)$. Arguing as in part (b), we obtain
$$\{W_2^2(G,G_0)\geq 2\epsilon/C(G_0,K)\}\Rightarrow \{d_{TV}(p^{\mathscr{L}}_{G,N}, p^{\mathscr{L}}_{G_0,N})\geq 2\epsilon\} \Rightarrow \{d_{H}(p^{\mathscr{L}}_{G,N}, p^{\mathscr{L}}_{G_0,N})\geq \epsilon\},$$
which gives, using the conclusion from part (a),
\begin{align*}
    \Pi&\left(G\in\mathcal{E}_K\,:\, W_2^2(G,G_0)\geq 2M_m \epsilon_{m,N}/C(G_0,K) \mid X^1_{[N]},\dots,X^m_{[N]}\right)\\
    &\leq \Pi\left(G\in\mathcal{E}_K\,:\, d_H(p^{\mathscr{L}}_{G,N},p^{\mathscr{L}}_{G_0,N})\geq M_m \epsilon_{m,N} \mid X^1_{[N]},\dots,X^m_{[N]}\right) \to 0.
\end{align*}
Absorbing the constants $2/C(G_0,K)$ into $\epsilon_{m,N}$, we conclude that for any increasing sequence $M_m\to\infty$,
$$\Pi\left(G\in\mathcal{E}_K\,:\, W_2(G,G_0)\geq M_m \sqrt{\epsilon_{m,N}} \mid X^1_{[N]},\dots,X^m_{[N]}\right)\to 0$$
in $(P_{G_0,N}^{\mathscr{L}})^\infty$-probability as $m\to\infty$ as long as $N\geq \mathfrak{N}_2(G_0,K)$. Thus, the contraction rate in this case is $\left(\frac{\log (mN)}{m}\right)^{1/4}$ with respect to $W_2$, up to a constant in terms of $K, V$. 
\end{proof}
% By Theorem \ref{thm:inverse-bounds}, it follows that the result holds for any fixed $N\geq 4$ (if the true topics $\theta_1^0,\dots,\theta_{K_0}^0$ are linearly independent) or $N\geq K_0+K-1$ in the general case. We note that the theorem is valid and useful even in the case where $N=N(m)$ is increasing as $m\to\infty$, as long as $\log(mN)/m\to 0$. 

\section{Proof of Theorem~\ref{theorem:contraction-allocation}}\label{sec:proof-allocation}
\subsection{Setup and organization of the proof of Theorem~\ref{theorem:contraction-allocation}}
% We are going to show that for a sufficiently large constant $C$ that does not depend on $m$ and $N$, we have
% \begin{equation}\label{eq:latent-allocation-conclusion}
%     \Ebb \Pi\left(\norm{q_1 - q_1^0}\geq C( \delta_N + \epsilon_{mN}) | X_{[N]}^{[m]}\right) \to 0,
% \end{equation}
% where the expectation is taken with respect to $\otimes_{i=1}^{\infty} \left(\otimes_{j=1}^{\infty} \Multi(X_j^i | \sum_{k=1}^{K_0} q_{ik}^0 \theta_{k}^{0}) \otimes \Dir(q_i^0 | \alpha^0)\right)$ and 
% $$\delta_{N} = \left(\dfrac{\log(N)}{N}\right)^{1/2}, \quad \text{and} \quad \epsilon_{mN} = \left(\dfrac{\log(mN)}{m}\right)^{1/2}.$$
% Note that $q_2^0, \dots, q_m^{0}$ does not appear in expression~\eqref{eq:latent-allocation-conclusion} so that they can be integrated out, and the expectation is taken with respect to 
% $$ \otimes_{i=2}^{m} p_{G_0, N}(X_{[N]}^{i}) \otimes \left(\otimes_{j=1}^{N} \Multi(X_j^1 | \sum_{k=1}^{K_0} q_{1k}^0 \theta_{k}^{0}) \otimes \Dir_{\alpha^0}(q_1^0)\right) .$$ 
% To make the presentation more clear, we re-denote $X^1_{[N]}$ to be $\tilde{X}_{[N]}$. The desire limit becomes
Let 
$$\delta_{\tilde{N}} = \left(\dfrac{\log(\tilde{N})}{\tilde{N}}\right)^{1/2}, \quad \text{and} \quad \epsilon_{mN} = \left(\dfrac{\log(mN)}{m}\right)^{1/2}.$$
Recall that we want to prove:
\begin{equation}\label{eq:latent-allocation-conclusion-tilde}
    \Ebb \Pi \left( |\tilde{q} -  \tilde{q}^0\| \geq C_{m, \tilde{N}} (\delta_{\tilde{N}} + \epsilon_{mN}) \bigg| \tilde{X}_{[\tilde{N}]}, X_{[N]}^{[m]} \right)\to 0,
    \end{equation}
for any slowing increasing sequence $C_{m, \tilde{N}} \to \infty$, where the expectation $\Ebb$ is with respect to 
$$(\tilde{X}_{[\tilde{N}]}, X_{[N]}^{[m]})\sim\left(\otimes_{j=1}^{\tilde{N}} \Multi(\tilde{X}_{j} |\sum_{k=1}^{K_0} \tilde{q}_{k}^0 \theta_k^0) \right)\otimes \left(\otimes_{i=1}^{m} p_{G_0, N}^{\mathscr{L}}(X^{i}_{[N]})\right).$$
% To show claim~\eqref{eq:latent-allocation-conclusion-tilde}, we first prove it is true when the expectation $\Ebb$ is with respect to 
% $$\left(\otimes_{j=1}^{\tilde{N}} \Multi(\tilde{X}_j | \sum_{k=1}^{K_0} \tilde{q}_{k}^0 \theta_{k}^{0})\right)\otimes \left(\otimes_{i=1}^{m} p^{\mathscr{L}}_{G_0, N}(X_{[N]}^{i})\right) ,$$ 
% for each fixed $\tilde{q}^0$, and then use the Dominated Convergence Theorem (DCT) for $\tilde{q}^0 \sim \Dir_{\alpha^0}$ to conclude. Therefore, we are treating $\tilde{q}^0$ as a deterministic quantity until the proof of Theorem~\ref{theorem:contraction-allocation} in Section~\ref{subsec:proof-contraction-allocation}. 

To avoid notation cluttering, we denote  $\Multi(\tilde{X}_{[\tilde{N}]}|\tilde{q}^{\top}\Theta)$ by $\tilde{p}^{\tilde{N}}_{\tilde{q}^{\top}\Theta}$ and $\otimes^m p^{\mathscr{L}}_{G, N}(X^{[m]}_{[N]})$ by $p_{G, N}^{m}$. Note that $G = \sum_{k=1}^{K_0} \tilde{\alpha}_k\theta_{k}$ also depends on $\Theta$. Hence, the working model is generally denoted by $\tilde{p}^{\tilde{N}}_{\tilde{q}^{\top} \Theta} \otimes p_{G, N}^{m}$ and true model by $\tilde{p}^{\tilde{N}}_{(\tilde{q}^{0})^{\top} \Theta^0} \otimes p_{G_0, N}^{m}$.

For a matrix $\Theta \in \Rbb^{K_0\times V}$ with $K_0\leq V$, denote $\sigma_{max}(\Theta)$ and $\sigma_{min}(\Theta)$ by its maximal and minimal singular values among the $K_0$ singular values. Because we assumed that $\Theta^0 = (\theta^0_1, \dots, \theta^0_{K_0})$ contains linearly independent rows, $\sigma_{min}(\Theta)$ are positive. We can choose a perturbation level $\epsilon_0$ sufficiently small so that 
$$\sigma_{min}(\Theta) \geq \sigma_{min}(\Theta^0)/2, \quad \forall \Theta \in B(\Theta^0, \epsilon_0),$$ 
where the distance of matrices here is the Frobenius distance. 

Similar to other results in this paper, to show Theorem~\ref{theorem:contraction-allocation}, we first prove the density contraction rate of $p^{\mathscr{L}}_{G, N}$ (to $p^{\mathscr{L}}_{G_0, N}$) and $\tilde{p}_{\tilde{q}^{\top}\Theta}$ (to $\tilde{p}_{(\tilde{q}^0)^{\top}\Theta^0}$) then use "inverse bounds" to transfer the rate to parameter contraction. We sequentially prove those results in the next two sections.

\subsection{Posterior contraction rate for density estimation}
\begin{proposition}\label{prop:rate-density-allocation}
Suppose $G_0 = \sum_{k=1}^{K_0} \tilde{\alpha}_{k}^{0} \delta_{\theta_k^0}\in \Ecal_{K_0}$ and the prior $\Pi$ on $\Ocal_{K_0}$ satisfies assumptions~\href{assume:p1}{(P1)} and~\href{assume:p2}{(P2)}, then for any slowly increasing sequence $C_{m, \tilde{N}}\to \infty$, 
$$\Ebb \Pi\left((d_H(\tilde{p}_{(\tilde{q})^{\top}\Theta}, \tilde{p}_{(\tilde{q}^0)^{\top}\Theta^0})\geq C_{m,\tilde{N}}\delta_{\tilde{N}}) \cup (d_H(p^{\mathscr{L}}_{G, N}, p^{\mathscr{L}}_{G_0, N}) \geq C_{m,\tilde{N}}\epsilon_{mN}))  \mid \tilde{X}_{[\tilde{N}]}, X^{[m]}_{[N]}\right) \to 0,$$
as $m, \tilde{N}\to \infty$, where the expectation is taken with respect to
$$\left(\otimes_{j=1}^{\tilde{N}} \Multi\left(\tilde{X}_j \Big| \sum_{k=1}^{K_0} \tilde{q}_{k}^0 \theta_{k}^{0}\right)\right) \otimes \left(\otimes_{i=1}^{m} p^{\mathscr{L}}_{G_0, N}(X_{[N]}^{i})\right).$$
\end{proposition}
\begin{proof}
% Remember to choose test $\Phi_{mN} = \Phi^1_{mN} + \Phi^2_{N} - \Phi^1_{mN} \Phi^2_{N}$. Then the type one error tends to 0. The type two error becomes
% $$\int_{\Theta, \alpha} \Ebb_{p_{G_0, N}^{m}} (1-\Phi^1_{mN}) \int_q \Ebb_{\tilde{p}_{\tilde{q}^{\top} \Theta}^{N}} (1-\Phi^2_{N}) \to 0, $$
% where the integral is over all parameters $(\alpha, \Theta, \tilde{q})$ such that $d_H(p_{G, N}, p_{G_0, N}) \geq \epsilon_{mN}$ or $d_H(\tilde{p}_{\tilde{q}^{\top}\Theta}, \tilde{p}_{(\tilde{q}^0)^{\top}\Theta^0}) \geq \delta_{N}$.
% Then we need to balance it with the ELBO using the fact that $\log(m) \asymp \log(N)$. 

Define the following set
\begin{align*}
    A_{m,N} &:= \left\{d_H\left(p^{\mathscr{L}}_{G,N},p^{\mathscr{L}}_{G_0,N}\right)\geq C_1\epsilon_{mN}\right\}\cup\left\{d_H\left(\tilde{p}_{\tilde{q}^{\top}\Theta}, \tilde{p}_{(\tilde{q}^0)^{\top}\Theta^0}\right)\geq C_2\delta_{\tilde{N}}\right\}.
\end{align*}
Using Bayes rule, we have the following expression for the posterior of $A_{m,N}$ given the data $\tilde{X}_{[\tilde{N}]}$ and $X^{[m]}_{[N]}$
\begin{align}\label{eq:bayes_prop7}
    \Pi\left(A_{m,N} \mid \tilde{X}_{[\tilde{N}]}, X^{[m]}_{[N]}\right) = \frac{\int_{A_{m,N}} \frac{p_{G,N}^m}{p_{G_0,N}^m}\left(X^{[m]}_{[N]}\right)\frac{\tilde{p}^N_{\tilde{q},\Theta}}{\tilde{p}^N_{\tilde{q}^0,\Theta^0}}\left(\tilde{X}_{[\tilde{N}]}\right)d\text{Dir}_{\alpha}(\tilde{q})d\Pi(\alpha,\Theta)}{\int \frac{p_{G,N}^m}{p_{G_0,N}^m}\left(X^{[m]}_{[N]}\right)\frac{\tilde{p}^N_{\tilde{q},\Theta}}{\tilde{p}^N_{\tilde{q}^0,\Theta^0}}\left(\tilde{X}_{[\tilde{N}]}\right)d\text{Dir}_{\alpha}(\tilde{q})d\Pi(\alpha,\Theta)}
\end{align}
Now, let $\Phi_{m,N}:=\Phi^1_{mN}+\Phi^2_N - \Phi^1_{mN}\Phi^2_N$, where the tests $\Phi^1_{mN}$ (for LDA) and $\Phi^2_N$ (for Multinomial) are as in Lemma~\ref{lem:test-two-problems}. By Lemma \ref{lem:elbo}, we know that the denominator of the above display is greater than the right side of Eq. \eqref{eq:elbo} with probability at least $1-1/D$. Assuming $\alpha^0_{k}< 1$ for each $k\in[K]$, for sufficiently small $\epsilon$, $\{\norm{\alpha - \alpha^0}\leq \epsilon^2\}$ implies that each coordinate of $\alpha$ is also upper bounded by 1 - in such a case, by the regularity property of the prior $\Pi$ on $\alpha,\Theta$ and the Dirichlet distribution with parameter $\alpha$ on $\tilde{q}$, we have the following lower bound
$$\tilde{\Pi}\left(\norm{\alpha - \alpha^0}\leq \epsilon^2/N^2, \norm{\Theta-\Theta^0}\leq \epsilon^2/N^2, \norm{\tilde{q} - \tilde{q}^0}\leq \delta^2 \right) \geq c(\epsilon/N)^{2(KV-1)}\delta^{2(K-1)}$$
for some constant $c>0$. Thus, let us denote by $B_{m,N}$ the event that the denominator in \eqref{eq:bayes_prop7} is greater than or equals $\omega_{m,N, \tilde{N};D}:=c(\epsilon_{mN}/N)^{2(KV-1)}\delta_{\tilde{N}}^{2(K-1)}e^{-2D(\tilde{N}\delta_{\tilde{N}}^2+2m\epsilon_{mN}^2)}$, which occurs with probability at least $(1-1/D)$.

We start to upper bound the posterior probability on the set $A_{m,N}$ using the test $\Phi_{mN}$ and the event $B_{mN}$ (to control the denominator) as follows
\begin{align}\label{eq: density q estimation numerator decompose}
    \Pi\left(A_{m,N} \mid \tilde{X}_{[N]}, X^{[m]}_{[N]}\right) &\leq \Phi_{mN} + \boldsymbol{1}(B_{mN}^c) \nonumber\\
    & \hspace{-2cm} + \omega_{m,N, \tilde{N};D}^{-1}\int_{A_{m,N}} (1-\Phi_{mN}) \frac{p_{G,N}^m}{p_{G_0,N}^m}\left(X^{[m]}_{[N]}\right)\frac{\tilde{p}^{\tilde{N}}_{\tilde{q},\Theta}}{\tilde{p}^{\tilde{N}}_{\tilde{q}^0,\Theta^0}}\left(\tilde{X}_{[\tilde{N}]}\right)d\text{Dir}_{\alpha}(\tilde{q})d\Pi(\alpha,\Theta).
\end{align}
Because of Lemma~\ref{lem:test-two-problems}, we have the type-I error $\Ebb \Phi_{mN} \to 0$ as $m, N \to \infty$. 

Finally, we argue that under the expectation with respect to $P_{G_0,N}^m\otimes \tilde{P}^N_{\tilde{q}^0,\Theta^0}$, each of the three terms above converges to 0, which would prove the proposition. The expectation of the first term goes to 0 by the property of $\Phi^1_{mN}$ and $\Phi^2_N$ (Lemma~\ref{lem:test-two-problems}). The expectation of the second term is simply $1/D$, which can be made arbitrarily small by a suitable choice of $D$. We argue that the expectation of the third term also vanishes under expectation in the remainder of the proof. 

For the expectation of the third term with respect to $P_0 = P_{G_0,N}^m\otimes \tilde{P}^N_{\tilde{q}^0,\Theta^0}$, we can apply Fubini's theorem to get
\begin{align*}
    &\Ebb_{P_0} \int_{A_{m,N}} (1-\Phi_{mN}) \frac{p_{G,N}^m}{p_{G_0,N}^m}\left(X^{[m]}_{[N]}\right)\frac{\tilde{p}^N_{\tilde{q},\Theta}}{\tilde{p}^N_{\tilde{q}^0,\Theta^0}}\left(\tilde{X}_{[\tilde{N}]}\right)d\text{Dir}_{\alpha}(\tilde{q})d\Pi(\alpha,\Theta) \\
    &= \int_{A_{m,N}} \Ebb_{P_{G,N}^m\otimes \tilde{P}_{\tilde{q},\Theta}^{\tilde{N}}} (1-\Phi_{mN})  d\text{Dir}_{\alpha}(\tilde{q})d\Pi(\alpha,\Theta)\\
    &\leq \sup_{A_{m,N}} \Ebb_{P_{G,N}^m\otimes \tilde{P}_{\tilde{q},\Theta}^{\tilde{N}}} \left( 1- \Phi_{mN}\right)\\
    &= \sup_{A_{m,N}} \Ebb_{P_{G,N}^m\otimes \tilde{P}_{\tilde{q},\Theta}^{\tilde{N}}} (1-\Phi_{mn}^1)(1-\Phi_{\tilde{N}}^2) \\
    &\leq \sup_{d_H(p_{G,N}, p_{G_0,N})\geq C_1\epsilon_{mN}} \Ebb_{P_{G,N}^m}(1 - \Phi_{mN}^1) \times \sup_{d_H(\tilde{p}_{\tilde{q},\Theta}, \tilde{p}_{\tilde{q}_0,\Theta_0})\geq C_2\delta_{\tilde{N}}} \Ebb_{\tilde{P}_{\tilde{q},\Theta}}(1-\Phi_{\tilde{N}}^2)\\
    &\leq e^{-C_1m\epsilon_{mN}^2 - C_2\tilde{N}\delta_{\tilde{N}}^2}
\end{align*}
where the last inequality uses the type-II error control for the tests $\Phi_{mN}^1$ and $\Phi_{\tilde{N}}^2$. Hence, the expectation of the third term in the right side of the decomposition in Eq. \eqref{eq: density q estimation numerator decompose} is upper bounded by
\begin{align*}
    &\omega_{m,N, \tilde{N};D}^{-1} \times \exp \left\{-C_1m\epsilon_{mN}^2 - C_2\tilde{N}\delta_{\tilde{N}}^2\right\}\\
    &= c \left(\frac{\epsilon_{mN}}{N}\right)^{-2(KV-1)}\delta_{\tilde{N}}^{-2(K-1)}\exp\left\{2D(\tilde{N}\delta_{\tilde{N}}^2+2m\epsilon_{mN}^2) - C_1 m\epsilon_{mN}^2 - C_2 \tilde{N} \delta_{\tilde{N}}^2\right\}\\
    &= c\left(\frac{mN}{\log(mN)}\right)^{KV-1}\left(\frac{\tilde{N}}{\log \tilde{N}}\right)^{K-1} (mN)^{4D-C_1} \tilde{N}^{2D-C_1}\\
    &\lesssim (mN)^{-(C_1 - 4D-KV+1)} N^{-(C_2 - 2D - K + 1)},
\end{align*}
which goes to 0 (as $m, \tilde{N}\to \infty$) for the choice $C_1 = 4D+KV$ and $C_2 = 2D+K$. Hence, for any arbitrarily slowly increasing $D$, with $C_1, C_2$ chosen as above (hence these are also slowly increasing sequences), the expected posterior mass on $A_{m,N}$ goes to 0.
\end{proof}

\begin{lemma}[Existence of tests]\label{lem:test-two-problems}
For a sequence $\epsilon_{mN} \to 0$ such that $\epsilon_{mN} \geq m^{-1/2}$ and constant $C\geq 1$, there exists a sequence of tests $\Phi^{1}_{mN}$ based on $X_{[N]}^{[m]}$ such that 
\begin{equation}
    \Ebb_{p_{G_0, N}^{m}} \Phi^1_{mN}\leq N(\epsilon_{mN}, \{p^{\mathscr{L}}_{G, N} : G\in \Ocal_{K_0}\}, d_H) e^{-C m\epsilon_{mN}^2},
\end{equation}
and 
\begin{equation}
    \sup_{d_H(p_{G, N}, p_{G_0, N}) \geq C\epsilon_{m N}} \Ebb_{p_{G, N}^{m}} (1- \Phi^1_{mN})\leq e^{-C m\epsilon_{mN}^2}.
\end{equation}
Moreover, for a sequence $\delta_{\tilde{N}} \to 0$ such that $\delta_{\tilde{N}} \geq \tilde{N}^{-1/2}$ and $C\geq 1$, there exists a sequence of tests $\Phi^{2}_{\tilde{N}}$ based on $\tilde{X}_{[\tilde{N}]}$ such that 
\begin{equation}
    \Ebb_{\tilde{p}_{(\tilde{q}^0)^{\top}\Theta^0}^{\tilde{N}}} \Phi^2_{\tilde{N}}\leq N(\delta_{\tilde{N}}, \{\tilde{p}_{\eta} : \eta\in \Delta^{V-1}\}, d_H) e^{-C \tilde{N}\delta_{\tilde{N}}^2},
\end{equation}
and 
\begin{equation}
    \sup_{d_H(\tilde{p}_{\eta}, \tilde{p}_{(\tilde{q}^0)^{\top}\Theta^0}) \geq C\delta_{\tilde{N}}} \Ebb_{\tilde{p}_{\eta}} (1- \Phi^2_{\tilde{N}})\leq e^{-C \tilde{N}\delta_{\tilde{N}}^2}.
\end{equation}
\end{lemma}

\begin{proof}
    This is a direct consequence of Theorem D.5 and D.8 in~\cite{ghosal2017fundamentals} with $K = 1/8$ and $j = C$ when applying for i.i.d. model $p_{G, N}$ and $\tilde{p}_{\eta}$, respectively.
\end{proof}

\begin{lemma}[Evidence lower bound]\label{lem:elbo}
    For any constants $D \geq 1$, $\delta \geq c_0 N^{-1/2}$, and $\epsilon \geq m^{-1/2}$, with $(\tilde{P}^{\tilde{N}}_{\tilde{q}^0, \Theta^0}\otimes P^{m}_{G_0, N})-$probability at least $1-D^{-1}$, we have
    \begin{align}\label{eq:elbo}
    \begin{split}
        & \int_{(\alpha, \Theta)} d\Pi(\alpha, \Theta)\left(\int_{\tilde{q}}  d\Dir_{\alpha}(\tilde{q})\dfrac{\tilde{p}^{\tilde{N}}_{\tilde{q}, \Theta}}{\tilde{p}^{\tilde{N}}_{\tilde{q}^0, \Theta^0}}(\tilde{X}_{[\tilde{N}]})  \right) \dfrac{p_{G, N}^{m}}{p_{G^0, N}^{m}}(X_{[N]}^{[m]})  \\
        & \geq \tilde{\Pi}\left(\norm{\alpha - \alpha^0}\leq (\epsilon/N)^2, \norm{\Theta-\Theta^0}\leq (\epsilon/N)^2, \norm{\tilde{q} - \tilde{q}^0}\leq \delta^2 \right) e^{-2D(\tilde{N}\delta^2 + 2m\epsilon^2)}, 
    \end{split}
    \end{align}
    where we recall that $\tilde{q} \sim \Dir_{\alpha}$ under the distribution $\tilde{\Pi}$.
\end{lemma}
\begin{proof}
    The prior structure is $\tilde{\pi}(\Theta,\alpha,\tilde{q})=\pi_1(\Theta)\pi_2(\tilde{q}|\alpha)\pi_3(\alpha)$, thus, we can write $\tilde{\Pi}(\Theta\in B_1, \alpha\in B_2, \tilde{q}\in B_3) = \Pi_1(B_1)\int_{B_2}\left(\int_{B_3} d\Pi_2(\tilde{q}|\alpha)\right)d\Pi_3(\alpha)$, where $\Pi_2$ is the Dirichlet measure. We let $\Pi=\Pi_1\otimes \Pi_3$, thus the above probability can be expressed as $\int_{B_1\times B_2} \left(\int_{B_3} d\Pi_2(\tilde{q}|\alpha)\right) d\Pi(\Theta,\alpha)$. Write $\overline{\epsilon} = \epsilon / N$. Let us define the sets
    \begin{align*}
        B_1 &= \left\{\Theta\mid \norm{\Theta - \Theta^0} \leq \overline{\epsilon}^2\right\} \\
        B_2 &= \left\{\alpha \mid \norm{\alpha-\alpha^0} \leq \overline{\epsilon}^2\right\}\\
        \Bcal' &= \left\{\tilde{q}\mid\norm{\tilde{q} - \tilde{q}^0} \leq \delta^2\right\},\\
        \Bcal &= B_1 \times B_2.
    \end{align*}
    The whole integral in the statement of the lemma becomes smaller by restricting on $\Bcal\times \Bcal'$. By next dividing both sides of the required inequality by $\tilde{\Pi}(\Bcal\times \Bcal')$, we can treat $\tilde{\Pi}$ as the restricted and normalized measure on $(\Theta,\alpha,\tilde{q})$ which places probability 1 on $\Bcal\times\Bcal'=\{\norm{\alpha - \alpha^0}\leq \overline{\epsilon}^2, \norm{\Theta-\Theta^0}\leq \overline{\epsilon}^2, \norm{\tilde{q} - \tilde{q}^0}\leq \delta^2\}$. Thus, with respect to such a $\tilde{\Pi}$, we have to show that with probability at least $1-1/D$,
    $$I := \int_{\Bcal\times \Bcal'} \frac{p_{G,N}^m}{p_{G_0,N}^m}\frac{\tilde{p}_{\tilde{q},\Theta}^{\tilde{N}}}{\tilde{p}_{\tilde{q}^0,\Theta^0}^{\tilde{N}}}d\tilde{\Pi}(\Theta,\alpha,\tilde{q}) \geq e^{-2D(\tilde{N}\delta^2+2mN^2\overline{\epsilon}^2)}.$$
    
    Now, applying log and by application of Jensen's inequality twice (recall $G$ is a function of $(\Theta,\alpha)$ only), we get
    \begin{align*}
        \log I &\geq \int_{\Bcal} \left[\log \prod_{i=1}^m \frac{p^{\mathscr{L}}_{G,N}(X^i_{[N]})}{p^{\mathscr{L}}_{G_0, N}(X^i_{[N]})} + \log \int_{\Bcal'}  \prod_{j=1}^{\tilde{N}} \frac{\tilde{p}_{\tilde{q},\Theta}(\tilde{X}_j)}{\tilde{p}_{\tilde{q}^0,\Theta^0}(\tilde{X}_j)} d\text{Dir}_{\alpha}(\tilde{q})\right] d\Pi(\alpha,\Theta) \\
        &\geq \int_{\Bcal} \left[\log \prod_{i=1}^m \frac{p^{\mathscr{L}}_{G,N}(X^i_{[N]})}{p^{\mathscr{L}}_{G_0, N}(X^i_{[N]})} + \int_{\Bcal'} \log \prod_{j=1}^{\tilde{N}} \frac{\tilde{p}_{\tilde{q},\Theta}(\tilde{X}_j)}{\tilde{p}_{\tilde{q}^0,\Theta^0}(\tilde{X}_j)} d\text{Dir}_{\alpha}(\tilde{q})\right] d\Pi(\alpha,\Theta) =: Z
    \end{align*}
    Let $x^+ = \max\{0, x\}$, then by another application of Jensen's inequality along with the facts that $(a+b)^+\leq a^+ + b^+$, we have,
    \begin{align*}
       (-Z)^+ &= \left(-\int_{\Bcal} \left[\log \prod_{i=1}^m \frac{p^{\mathscr{L}}_{G,N}(X^i_{[N]})}{p^{\mathscr{L}}_{G_0, N}(X^i_{[N]})} + \int_{\Bcal'}  \log \prod_{j=1}^{\tilde{N}} \frac{\tilde{p}_{\tilde{q},\Theta}(\tilde{X}_j)}{\tilde{p}_{\tilde{q}^0,\Theta^0}(\tilde{X}_j)} d\text{Dir}_{\alpha}(\tilde{q})\right] d\Pi(\alpha,\Theta)\right)^+ \\
       &\leq \int_{\Bcal} \left[\log \prod_{i=1}^m \frac{p^{\mathscr{L}}_{G_0,N}(X^i_{[N]})}{p^{\mathscr{L}}_{G, N}(X^i_{[N]})} + \int_{\Bcal'} \log \prod_{j=1}^{\tilde{N}} \frac{\tilde{p}_{\tilde{q}^0,\Theta^0}(\tilde{X}_j)}{\tilde{p}_{\tilde{q},\Theta}(\tilde{X}_j)} d\text{Dir}_{\alpha}(\tilde{q})\right]^+ d\Pi(\alpha,\Theta) \\
       &\leq \int_{\Bcal} \left[\left(\log \prod_{i=1}^m \frac{p^{\mathscr{L}}_{G_0,N}(X^i_{[N]})}{p^{\mathscr{L}}_{G, N}(X^i_{[N]})}\right)^+ + \left(\int_{\Bcal'} \log \prod_{j=1}^{\tilde{N}} \frac{\tilde{p}_{\tilde{q}^0,\Theta^0}(\tilde{X}_j)}{\tilde{p}_{\tilde{q},\Theta}(\tilde{X}_j)} d\text{Dir}_{\alpha}(\tilde{q})\right)^+\right]d\Pi(\alpha,\Theta)\\
       &\leq \int_{\Bcal} \left[\left(\log \prod_{i=1}^m \frac{p^{\mathscr{L}}_{G_0,N}(X^i_{[N]})}{p^{\mathscr{L}}_{G, N}(X^i_{[N]})}\right)^+ + \int_{\Bcal'} \left(\log \prod_{j=1}^{\tilde{N}} \frac{\tilde{p}_{\tilde{q}^0,\Theta^0}(\tilde{X}_j)}{\tilde{p}_{\tilde{q},\Theta}(\tilde{X}_j)}\right)^+ d\text{Dir}_{\alpha}(\tilde{q})\right]d\Pi(\alpha,\Theta).
    \end{align*}
    Taking an expectation with respect to the data $(X_{[N]}^{[m]}, \tilde{X}_{[N]})\sim P_{G_0,N}^m(X_{[N]}^{[m]})\otimes \tilde{P}_{\tilde{q}^0, \Theta^0}(\tilde{X}_{[\tilde{N}]})$, we have
    \begin{align*}
        \Ebb(-Z)^+ \leq \underbrace{\int_{\Bcal} K^+(p_{G_0,N}^m; p_{G,N}^m) d\Pi(\alpha,\Theta)}_{A_1} + \underbrace{\int_{\Bcal}\int_{\Bcal'} K^+(\tilde{p}_{\tilde{q}^0,\Theta^0}, \tilde{p}_{\tilde{q},\Theta}) d\text{Dir}_{\alpha}(\tilde{q}) d\Pi(\alpha,\Theta)}_{A_2}.
    \end{align*}
    We now deal with $A_1$ and $A_2$ separately. For $A_2$, we note that the $K^+$ is computed between two multinomial distributions, with parameters $(\tilde{N}, \eta^0)$ and $(\tilde{N},\eta)$ where $\eta^0=(\Theta^0)^\top\tilde{q}^0$ and $\eta=\Theta^\top\tilde{q}$. By Lemma \ref{lem:distance-multinom}, we know that $d_{KL}(p_{\eta^0,\tilde{N}}, p_{\eta,\tilde{N}})\leq \tilde{N}\norm{\eta^0-\eta}/c_0$. Under the sets $\Bcal$ and $\Bcal'$, we have 
    \begin{align*}
        \norm{\eta^0-\eta} &= \norm{(\Theta^0)^\top\tilde{q}^0 - \Theta^\top\tilde{q}} \\
        &\leq \norm{(\Theta^0)^\top (\tilde{q}^0 - \tilde{q})} + \norm{(\Theta^0 - \Theta)^\top \tilde{q}} \\
        &\leq \norm{\Theta^0}\norm{\tilde{q}^0-\tilde{q}} + \norm{\Theta^0 - \Theta}\norm{\tilde{q}} \\
        &\leq \sigma_{\max}(\Theta^0) \delta^2 + \overline{\epsilon}^2\\
        &\leq \delta^2 + \overline{\epsilon}^2,
    \end{align*}
    because $\sigma_{\text{max}}\leq 1$ (as rows of $\Theta^0$ are in $\Delta^{V-1}$). Using the above, for any $(\tilde{q},\alpha,\Theta)\in \Bcal\cap\Bcal'$, we have
    $$d_{KL}(\tilde{p}_{\tilde{q}^0,\Theta^0}, \tilde{p}_{\tilde{q},\Theta}) \leq \frac{ \tilde{N}\delta^2}{c_0} + \frac{\tilde{N}\overline{\epsilon}^2}{c_0} =: \overline{d}_{KL}.$$
    Using the fact that $K^+\leq d_{KL} + \sqrt{d_{KL}/2}$ (for example, Lemma B.13 in \cite{ghosal2017fundamentals}) and the choices $(\overline{\epsilon} + \delta)\geq \sqrt{c_0/N}/2$ (so that the upper bound $\overline{d}_{KL}$ above satisfies $\sqrt{\overline{d}_{KL} / 2} \leq \overline{d}_{KL}$, we have the following upper bound for the second term:
    \begin{align*}
        A_2 &\leq 2\int_{\Bcal} \text{Dir}_{\alpha}\left(\left\{\tilde{q}\in\Delta^{K-1} \mid \norm{\tilde{q}-\tilde{q}^0}\leq \delta^2\right\}\right)d\Pi(\alpha,\Theta)\left(\frac{ \tilde{N}\delta^2}{c_0} + \frac{\tilde{N}\overline{\epsilon}^2}{c_0}\right)\\
        &= 2\tilde{\Pi}(\Bcal\times \Bcal') \left(\frac{\tilde{N}\delta^2}{c_0} + \frac{\tilde{N}\overline{\epsilon}^2}{c_0}\right) = 2\left(\frac{ \tilde{N}\delta^2}{c_0} + \frac{\tilde{N}\overline{\epsilon}^2}{c_0}\right),
    \end{align*}
    since $\tilde{\Pi}(\Bcal\times \Bcal') = 1$, by the normalization at the start of the proof.
    
    For $A_1$, we have $d_{KL}(p_{G_0,N}^m, p_{G,N}^m)=m d_{KL}(p^{\mathscr{L}}_{G_0,N}, p^{\mathscr{L}}_{G,N})$. Using the fact that $\norm{B^\top} = \norm{B}\geq \norm{B_i}_2$ where $B_i$'s are the columns of $B$, we get that under $\Bcal$, since $\norm{\Theta-\Theta^0}\leq \overline{\epsilon}^2$, we must have $\norm{\theta_k-\theta_k^0}\leq \overline{\epsilon}^2$ for each $k$. Thus, using the argument from the proof of Lemma \ref{prop:kl_ball_prior}, we have $d_{KL}(p_{G_0,N}^m, p_{G,N}^m)\leq mN^2 W_1(G_0,G)/c_0\leq mN^2\overline{\epsilon}^2/c_0$. Again using the fact that $K^+\leq d_{KL}+\sqrt{d_{KL}/2}$ and the choice $\overline{\epsilon}\geq 1/(N\sqrt{m})$, we have 
    $$A_1 \leq 2\Pi(\Bcal) \frac{mN^2\overline{\epsilon}^2}{c_0}=\frac{2mN^2\overline{\epsilon}^2}{c_0},$$
    since due to the normalization of $\tilde{\Pi}$, we have $\Pi(\Bcal)=1$, being the marginal of $\tilde{\Pi}$, which is supported on $\Bcal\times \Bcal'$.
    
    Combining the two, we get
    \begin{align*}
        \Ebb (-Z)^+ &\leq 2\left[\frac{mN^2\overline{\epsilon}^2}{c_0} + \frac{\tilde{N}\delta^2}{c_0} + \frac{\tilde{N}\overline{\epsilon}^2}{c_0}\right].
    \end{align*}
    % where we used that fact that $\tilde{\Pi}$ is concentrated on $\Bcal\times\Bcal'$ and $\psi$ is a function such that $\Pi(\Bcal)/\tilde{\Pi}(\Bcal\times\Bcal') \leq \psi(\overline{\epsilon},\delta).$

    % Let us explore the function $\psi$. By definition,
    % \begin{align*}
    %     \tilde{\Pi}(\Bcal\times\Bcal') &= \int_{\Bcal} d\text{Dir}_{\alpha}(\Bcal') d\Pi(\Theta,\alpha) \\
    %     &\geq \int_{\Bcal} \frac{\Gamma(\sum_k \alpha_k)}{\prod_k \Gamma(\alpha_k)}\left(\frac{\delta}{K}\right)^{K-1} d\Pi(\Theta,\alpha) \\
    %     &= \left(\frac{\delta}{K}\right)^{K-1}\int_{B_1} \int_{B_2} \frac{\Gamma(\sum_k \alpha_k)}{\prod_k \Gamma(\alpha_k)} d\Pi_3(\alpha) d\Pi_1(\Theta) \\
    %     &\geq \frac{1}{C}   \left(\frac{\delta}{K}\right)^{K-1} \Pi(\Bcal)
    % \end{align*}
    % where we use Lemma 4 \cite{nguyen2015} (subject to $K\leq V$ and $\alpha_k\leq 1$ for all $k\in[K]$) and $C = \inf_{\alpha\in B_2} \prod_k \Gamma(\alpha_k)/\Gamma(\sum_k \alpha_k) > 0$ is a constant only depending on prior $\Pi_3$ on $\alpha$ - {\color{red}{can this be improved? Need an upper bound independent of $\alpha$, however local to $\tilde{q}$ should give more control. Or perhaps just include the assumption that all $\alpha\leq 1$}}. This immediately says that we can choose $\psi(\overline{\epsilon},\delta)=C(K/\delta)^{K-1}$. 
    
    Now, by Markov's inequality, we have
    \begin{align*}
        P[Z < -2D\phi(\overline{\epsilon},\delta)] &\leq \frac{\Ebb (-Z)^+}{2D\phi(\overline{\epsilon},\delta)}
    \end{align*}
    Recalling that $I\geq \exp(Z)$, we have $\{I\leq \exp(-\vartheta)\}\subset\{Z<-\vartheta\} $, which immediately gives
    $$P[I \leq \exp(-2D\phi(\overline{\epsilon},\delta))]\leq \frac{m N^2\overline{\epsilon}^2 +\tilde{N}\delta^2 + \tilde{N}\overline{\epsilon}^2}{D c_0 \phi(\overline{\epsilon},\delta)}\leq  c_1\frac{2mN^2\overline{\epsilon}^2 + \tilde{N}\delta^2}{D c_0 \phi(\overline{\epsilon},\delta)},$$
    since we assume $\tilde{N}\leq c_1 mN$ for some constant $c_1$ not depending on $m, N, \tilde{N}$ (WLOG we can assume $c_1 \geq 1$).
    Taking $\phi(\overline{\epsilon},\delta) = c_1\left[ \tilde{N}\delta^2 + 2m N^2\overline{\epsilon}^2\right]/c_0$, we have the desired result.
\end{proof}
Lemma~\ref{lem:test-two-problems} and Lemma~\ref{lem:elbo} are adaptations of Lemma D.5 and Lemma 6.26 in~\cite{ghosal2017fundamentals} to our setting, where we need to uniformly control the convergence of $\tilde{q}$ under a small perturbation of $(\alpha, \Theta)$. We will use the two lemmas above to prove Theorem~\ref{theorem:contraction-allocation} first while deferring their proofs to the next section. In the following proof, the constant $C$ can differ from line to line, but it does not depend on $m$ and $N$.

\subsection{Posterior contraction rate for allocation $q$}\label{subsec:proof-contraction-allocation}
\begin{proof}[Proof of Theorem~\ref{theorem:contraction-allocation}]
Taking the complement of the set in the conclusion of Proposition~\ref{prop:rate-density-allocation}, we have
$$\Ebb \Pi\left((d_H(p^{\mathscr{L}}_{G, N}, p^{\mathscr{L}}_{G_0, N}) \leq C\epsilon_{mN}) \cap (d_H(\tilde{p}_{\tilde{q},\Theta}, \tilde{p}_{\tilde{q}^0,\Theta^0}) \leq C\delta_{N})  \mid \tilde{X}_{[N]}, X^{[m]}_{[N]}\right) \to 1,$$
where $C$ is a constant does not depend on $m$ and $N$, and the expectation is taken with respect to 
$$\left(\otimes_{j=1}^{N} \tilde{p}_{\tilde{q}^0, \Theta^0}(\tilde{X}_j)\right) \otimes \left(\otimes_{i=1}^{m} p^{\mathscr{L}}_{G_0, N}(X_{[N]}^{i})\right).$$
In the proof, the constant $C$ can vary from line to line but it never depends on $m$ and $N$. 
Let $S_{mN}: = \{(\tilde{q}, \alpha, \Theta): d_H(p^{\mathscr{L}}_{G, N}, p^{\mathscr{L}}_{G_0, N}) \leq C\epsilon_{mN}, d_H(\tilde{p}_{\tilde{q},\Theta}, \tilde{p}_{\tilde{q}^0,\Theta^0}) \leq C\delta_{N}\}$, and consider any tuple $(\tilde{q}, \alpha, \Theta) \in S_{mN}$. From the exact-fitted inverse bound for linearly independent topics (part (b) in Theorem~\ref{thm:inverse-bounds}), we have 
$$W_1(G, G_0)\leq C d_{TV}(p^{\mathscr{L}}_{G, 3}, p^{\mathscr{L}}_{G_0, 3})\leq C d_H(p^{\mathscr{L}}_{G, N}, p^{\mathscr{L}}_{G_0, N}) \leq C\epsilon_{mN},$$
where $G = \sum_{i=1}^{K_0} \tilde{\alpha}_{j}\delta_{\theta_j}$. Hence, from~\eqref{eq:Wasserstein-equivalent-Voronoi}, there exists a permutation $\tau$ of $[K_0]$ such that
$$\sum_{k=1}^{K_0} \norm{\Theta_{\tau(k)} - \Theta_{k}^0} \leq C \epsilon_{mN}.$$
WLOG, we assume this permutation is identity, which yields $\norm{\Theta - \Theta^0}_{F} \leq C \epsilon_{mN}$, where $\norm{\cdot}_{F}$ is the Frobenius norm. Besides, we have 
\begin{align*}
C \delta_N & \geq d_H(\tilde{p}_{\tilde{q}, \Theta}, \tilde{p}_{\tilde{q}^0, \Theta^0}) \\
& \geq \dfrac{1}{\sqrt{2}} d_{TV}(\tilde{p}_{\tilde{q}, \Theta}, \tilde{p}_{\tilde{q}^0, \Theta^0})\\
& \geq \dfrac{1}{\sqrt{2}} \norm{(\Theta^0)^{\top} \tilde{q}^0 - \Theta^{\top} \tilde{q}}\\
& \geq \dfrac{1}{\sqrt{2}} \left(\norm{(\Theta^0)^{\top} (\tilde{q}^0 - \tilde{q})} - \norm{(\Theta^0 - \Theta)^{\top} \tilde{q}}\right)\\
& \geq \dfrac{1}{\sqrt{2}} \left(\norm{(\Theta^0)^{\top} (\tilde{q}^0 - \tilde{q})} - C\epsilon_{mN}\right)
\end{align*}
where we use the fact that $l_1$ norm is greater or equal to $l_2$ norm (for probability vectors) in the third inequality, triangle inequality in the fourth inequality, and the fact that $\norm{(\Theta^0 - \Theta)^{\top} \tilde{q}} \leq \norm{(\Theta^0 - \Theta)}_{F} \norm{q} \leq \norm{(\Theta^0 - \Theta)}_{F}$ in the last one. Therefore,
$$\norm{\tilde{q}^0 - \tilde{q}} \leq \dfrac{1}{\sigma_{\text{min}}(\Theta^0)} \norm{(\Theta^0)^{\top}(\tilde{q}^0 - \tilde{q})}_F \leq \dfrac{1}{\sigma_{\text{min}}(\Theta^0)}\left(\sqrt{2}C\delta_N + C\epsilon_{mN} \right).$$
Hence, 
$$\Ebb \Pi\left(\norm{\tilde{q}^0 - \tilde{q}} \leq C\left(\left(\dfrac{\log(mN)}{m}\right)^{1/2} + \left(\dfrac{\log(N)}{N}\right)^{1/2} \right)  \mid \tilde{X}_{[N]}, X^{[m]}_{[N]}\right) \to 1,$$
where the expectation $\Ebb$ is with respect to 
$$ \otimes_{i=1}^{m} p^{\mathscr{L}}_{G_0, N}(X_{[N]}^{i}) \otimes \left(\otimes_{j=1}^{N} \Multi\left(\tilde{X}_j \Big| \sum_{k=1}^{K_0} \tilde{q}_{k}^0 \theta_{k}^{0}\right)\right).$$ 
\end{proof}

\section{Proof of other theoretical results}\label{sec:auxiliary-result}

\subsection{Proof of Proposition~\ref{prop:identifiability}}
\begin{proof} From Theorem~\ref{thm:equivalence-LDA-mixture} and the remark afterward, we see that for two latent mixing measure $G_0 = \sum_{k=1}^{K_0} \tilde{\alpha}^0_k \delta_{\theta^0_k}$ and $G = \sum_{k=1}^{K} \tilde{\alpha}_k \delta_{\theta_k}$, the equation $p^{\mathscr{L}}_{G_0, N} = p^{\mathscr{L}}_{G, N}$ is equivalent to $p^{\mathscr{M}}_{G_0, N} = p^{\mathscr{M}}_{G, N}$, which is, in turn, equivalent to 
\begin{equation}
    \sum_{k=1}^{K_0} \tilde{\alpha}^0_k (\theta^0_k)^{\otimes N} = \sum_{k=1}^{K} \tilde{\alpha}_k \theta_k^{\otimes N}.
\end{equation}
Hence, the exact-fitted identifiability is the direct consequence of the generalized Kruskal theorem for $N-$way arrays (see, for example, Theorem 1 of \cite{lovitz2023generalization}). The over-fitted identifiability is a consequence of Corollary 33 in \cite{lovitz2023generalization}. The other results are adapted from~\cite{vandermeulen2019operator}.
\end{proof}

\subsection{On the distances between distributions}
\begin{lemma}\label{lem:fact-distance}
    For any $d\in \{d_{TV}, d^{2}_{H}, d_{KL}\}$, we have the following properties:
        \begin{enumerate}
            \item[(i)] $d(\cdot, \cdot)$ is a convex function of both arguments;
            \item[(ii)] For any two densities $p, q$ being commonly dominated by measure $\mu$ on $\Rbb^{N}$. Let $S$ be a subspace of $\Rbb^{N}$, and $p_S$ and $q_S$ are projections of $p$ and $q$ onto $S$, respectively. Then $d(p(x_{S}), q(x_{S})) \leq d(p, q)$; 
            \item[(iii)] $d(p_1 \otimes p_2, q_1 \otimes q_2) \leq d(p_1, q_1) + d(p_2, q_2)$ for any distributions $p_1, p_2, q_1, q_2$.
        \end{enumerate}
        The last inequality becomes equality for $d = d_{KL}$.
\end{lemma}
By specializing the dominating measure to be counting measure on $[V]^{N}$, we have the needed results in the proof of Proposition~\ref{prop:metric-equivalent}.
\begin{proof}[Proof of Lemma~\ref{lem:fact-distance}]
\begin{enumerate}
    \item[(i)] This is due to the variational characterization of $d$ when seeing as $f-$divergences. See, e.g., Lemma 1 of~\cite{nguyen2013convergence} and also~\cite{nguyen2010estimating}.
    \item[(ii)] \textbf{Case $d = d_{TV}$:} we have
    $$d_{TV}(p, q) = \sup_{\norm{f}_{\infty} \leq 1} |\Ebb_{X\sim p} f(X) - \Ebb_{X\sim q} f(X)|,$$
    and
    $$d_{TV}(p_S, q_S) = \sup_{\norm{g}_{\infty} \leq 1} |\Ebb_{X\sim p_S} g(X) - \Ebb_{X\sim q_S} g(X)|.$$ 
    Observing that each $g: S\to \Rbb$ with $\norm{g}_{\infty} \leq 1$ can be extended into a $\overline{g}: \Rbb^{N}\to \Rbb$ with $\norm{\overline{g}}_{\infty} \leq 1$ by mapping coordinate outside of $S$ to 0. Such extended maps form a subspace of $\{f: \Rbb^{N} \to \Rbb: \norm{f}_{\infty}\leq 1\}$. Hence, $d_{TV}(p_S, q_S)\leq d_{TV}(p, q)$.
    
    \noindent\textbf{Case $d = d_{H}$:} Write $\Rbb^{N} = S \times T$, where $T$ is another subspace of $\Rbb^{N}$. We have
    \begin{align*}
        d_{H}^2(p, q) & = \dfrac{1}{2} \int_{S \times T} (\sqrt{p(s, t)} - \sqrt{q(s, t)})^2 d\mu(s, t)\\
        & = 1 - \int_{S\times T} \sqrt{p(s, t) q(s, t)} d\mu(s, t) \\
        & \geq 1 - \int_{S} \left(\int_{T} p(s, t) d\mu(t)\right)^{1/2} \left(\int_{T} q(s, t) d\mu(t)\right)^{1/2} d\mu(s) \\
        & =  1 - \int_{S} \sqrt{p(s) q(s)} d\mu(s, t)\\
        & = d_{H}^2(p_S, q_S),
    \end{align*}
    where we used Holder's inequality. Hence, the inequality is correct for $d=d^2_H$.
    
    \noindent\textbf{Case $d = d_{KL}$:} We have
    \begin{align*}
        d_{KL}(p \| q) - d_{KL}(p_S \| q_S) & = \Ebb_{(s, t) \sim p} \left[\log p(s, t) - \log q(s, t) - (\log p(s) - \log q(s)) \right]\\
        & = \Ebb_{(s, t)\sim p} [\log p(t | s) - \log q(t | s)]\\
        & = \Ebb_{s} d_{KL}(p(\cdot | s) \| q(\cdot | s)) \geq 0.
    \end{align*}

    \item[(iii)] \textbf{Case $d = d_{TV}$:} Denote $P_i$ by the distribution of $p_i$, for $i=1,2$. Consider any measurable set $A_1$ and $A_2$, we have
    \begin{align*}
        & |(P_1 \otimes P_2)(A_1 \times A_2) - (Q_1 \otimes Q_2)(A_1 \times A_2)|  = |P_1(A_1) P_2(A_2) - Q_1(A_1) Q_2(A_2)|\\
        & \leq |P_1(A_1) P_2(A_2) - Q_1(A_1) P_2(A_2)| + |Q_1(A_1) P_2(A_2) - Q_1(A_1) Q_2(A_2)|\\
        &\leq |P_1(A_1) - Q_1(A_1)| + |P_2(A_2) - Q_2(A_2)|\\
    \end{align*}
    Hence, 
    $$\sup_{A_1, A_2} |(P_1 \otimes P_2)(A_1 \times A_2) - (Q_1 \otimes Q_2)(A_1 \times A_2)|\leq d_{TV}(P_1, Q_1) + d_{TV}(P_2, Q_2).$$
    Because it is correct for any rectangular set $A_1\times A_2$, and one can approximate any measurable set $A$ in the product space by finite disjoint of such rectangles, we have 
    $$d_{TV}(P_1\otimes Q_1, P_2\otimes Q_2) = \sup_{A} |(P_1 \otimes P_2)(A) - (Q_1 \otimes Q_2)(A)|\leq d_{TV}(P_1, Q_1) + d_{TV}(P_2, Q_2).$$
    \noindent\textbf{Case $d = d_{H}$:} Denote the affinity between $p$ and $q$ to be $\rho(p, 
 q) = \int \sqrt{pq} \leq 1$, we have
    \begin{align*}
        d_H^2(p_1 \otimes p_2, q_1 \otimes q_2) & = 1 - \rho(p_1 \otimes p_2, q_1\otimes q_2)\\
        & = 1 - \rho(p_1, q_1) \rho(p_2, q_2)\\
        & \leq 2 - \rho(p_1, q_1) - \rho(p_2, q_2)\\
        & = d_H^2(p_1, q_1) + d_H^2(p_2, q_2)
    \end{align*}
    
    \noindent\textbf{Case $d = d_{KL}$:} We have
    \begin{align*}
        d_{KL}(p_1\otimes p_2 \| q_1 \otimes q_2) & = \Ebb_{(s_1, s_2) \sim p_1\otimes p_2} [\log p_1(s_1) p_2(s_2) - \log q_1(s_1) q_2(s_2)]\\
        & = \Ebb_{(s_1 \sim p_1} [\log p_1(s_1) - \log q_1(s_1) ] + \Ebb_{(s_2 \sim p_2} [\log p_2(s_2) - \log q_2(s_2)]\\
        & = d_{KL}(p_1 \| q_1) + d_{KL}(p_2 \| q_2). 
    \end{align*}
\end{enumerate}
\end{proof}

\end{appendix}

\end{document}